\documentclass[12pt]{amsart}
\usepackage{hyperref}
\usepackage{latexsym,color}
\title[Measure solutions of the Boltzmann equation, part I]{On Measure
  Solutions of the Boltzmann Equation, part I: Moment Production and
  Stability Estimates}

\author{Xuguang Lu and Cl\'{e}ment Mouhot}

\def\signxl{\bigskip \begin{center} {\sc Xuguang Lu\par\vspace{3mm}
Tsinghua University\par
Department of Mathematical Sciences\par
Beijing 100084, P.R.,
CHINA\par\vspace{3mm}
e-mail:} \tt{xglu@math.tsinghua.edu.cn} \end{center}}

\def\signcm{\bigskip \begin{center} {\sc
Cl\'ement Mouhot\par\vspace{3mm}
University of Cambridge\par
DPMMS, Centre for Mathematical Sciences\par
Wilberforce Road,
Cambridge CB3 0WA,
UK\par\vspace{3mm}
e-mail:} \tt{C.Mouhot@dpmms.cam.ac.uk} \end{center}}

\begin{document}
% End of preamble and beginning of text.
\maketitle                            % Produces the title.

\newcommand{\p}{\partial}
\newcommand{\og}{\omega}
\newcommand{\Og}{\Omega}
\newcommand{\dt}{\delta}
\newcommand{\Dt}{\Delta}
\newcommand{\ld}{\lambda}
\newcommand{\Ld}{\Lambda}
\newcommand{\Gm}{\Gamma}
\newcommand{\gm}{\gamma}
\newcommand{\vp}{\varphi}
\newcommand{\vep}{\varepsilon}
\newcommand{\ep}{\epsilon}
\newcommand{\vh}{\varrho}
\newcommand{\vap}{\varphi}
\newcommand{\kp}{\eta}
\newcommand{\Sg}{\Sigma}
\newcommand{\fa}{\forall}
\newcommand{\fr}{\frac}
\newcommand{\sg}{\sigma}
\newcommand{\ept}{\emptyset}
\newcommand{\btd}{\nabla}
\newcommand{\btu}{\bigtriangleup}
\newcommand{\tg}{\triangle}
\newcommand{\la}{\langle}
\newcommand{\ra}{\rangle}
\newcommand{\Th}{{\cal T}^h}
\newcommand{\wt}{\widetilde}
\newcommand{\ul}{\underline}
\newcommand{\Ups}{\Upsilon}
\newcommand{\be}{\begin{equation}}
\newcommand{\ee}{\end{equation}}
\newcommand{\ba}{\begin{array}}
\newcommand{\ea}{\end{array}}
\newcommand{\bea}{\begin{eqnarray}}
\newcommand{\eea}{\end{eqnarray}}
\newcommand{\beas}{\begin{eqnarray*}}
\newcommand{\eeas}{\end{eqnarray*}}
\newcommand{\dpm}{\displaystyle }
\newcommand{\intt}{\int\!\!\!\!\int}
\newcommand{\inttt}{\int\!\!\!\!\int\!\!\!\!\int}
\newcommand{\intttt}{\int\!\!\!\!\int\!\!\!\!\int\!\!\!\!\int}

\theoremstyle{theorem}
\newtheorem{theorem}{Theorem}
\newtheorem{lemma}[theorem]{Lemma}
\newtheorem{conjecture}[theorem]{Conjecture}
\newtheorem{corollary}[theorem]{Corollary}
\newtheorem{proposition}[theorem]{Proposition}
\theoremstyle{definition}
\newtheorem{definition}[theorem]{Definition}

\def\theThm{{\arabic{section}.\arabic{theorem}}}
\numberwithin{equation}{section}
\numberwithin{theorem}{section}

\theoremstyle{remark}
\newtheorem{remark}[theorem]{Remark}
\newtheorem{remarks}[theorem]{Remarks}
\newtheorem{examples}[theorem]{Examples}
\newtheorem{example}[theorem]{Example}

%\vskip 0.1in \baselineskip 18.6pt

\begin{abstract}
  The spatially homogeneous Boltzmann equation with hard potentials is
  considered for measure valued initial data having finite mass and
  energy. We prove the existence of \emph{weak measure solutions},
  with and without angular cutoff on the collision kernel; the proof
  in particular makes use of an approximation argument based on the
  \emph{Mehler transform}. Moment production estimates in the usual
  form and in the exponential form are obtained for these
  solutions. Finally for the Grad angular cutoff, we also establish
  uniqueness and strong stability estimate on these solutions.
\end{abstract}

\bigskip


\bigskip


\tableofcontents

\section {Introduction}
\label{sec1}

In this paper we study the spatially homogeneous Boltzmann equation
for hard interaction potentials with or without angular cutoff.
The initial data are assumed to be positive Borel measures having
finite moments up to order $2$. Our main results are the existence and
stability of measure solutions that have polynomial and exponential
moment production properties.

\subsection{The spatially homogeneous Boltzmann equation}

\subsubsection{The equation}
\label{sec:equation}

Before introducing the main results, let us recall the Boltzmann
equation for $L^1$-solutions and basic notations. The equation for the
space homogeneous solution takes the form
\begin{equation}\label{(B)}
\fr{\p }{\p t}f_t(v)=Q(f_t,f_t)(v)\,, \quad (v,t)\in
\mathbb{R}^N\times(0,\infty)\,,\quad N\ge 2
\end{equation}
with some given initial data $f_t(v)|_{t=0}=f_0(v)$
 and $Q$ is the collision integral defined by
\begin{equation}\label{(1.1)}
Q(f,f)(v)=\intt_{\mathbb{R}^N\times
\mathbb{S}^{N-1}}B(v-v_*,\sg)\Big(f(v')f(v_*')-f(v)f(v_*)\Big)
{\rm d}\sg {\rm d}v_*\,,
\end{equation}
where $v,v_*$ and $v', v_*'$ stand for velocities of two particles
respectively after and before their collision,
\begin{equation}\label{(1.v)} v'=\fr{v+v_*}{2}+\fr{|v-v_*|}{2}\sg\,,\quad
v_*'=\fr{v+v_*}{2}-\fr{|v-v_*|}{2}\sg\,,\quad \sg \in
\mathbb{S}^{N-1}\,.\end{equation} The above relation between $v,v_*$ and $v',
v_*'$ shows that the collision is elastic:
$$v'+v_*'=v+v_*\,,\quad |v'|^2+|v_*'|^2 = |v|^2 + |v_*|^2\,.$$

\subsubsection{The collision kernel}
\label{sec:collision-kernel} The collision kernel $B(z,\sg)$ under
consideration is assumed to be a function of
$(|z|,\fr{z}{|z|}\cdot\sg)$, i.e.
\begin{equation}\label{(1.B)} B(z,\sg)=\bar{B}(|z|,\cos\theta),\quad
\cos\theta=\fr{z}{|z|}\cdot\sg,\quad \theta\in[0,\pi]
\end{equation}
where  $(r,t)\mapsto \bar{B}(r,t)$ is a non-negative Borel function on
$[0,\infty)\times[-1,1]$ satisfying
\begin{equation}\label{(1.2)}\forall\, t \in (-1,1),\quad
r\mapsto \bar{B}(r, t)\,\,\,{\rm is\,\,\, continuous\,\,\,on}\,\,\,
[0,\infty),
\end{equation}
\begin{equation}\label{(1.3)}
 \bar B(r,t)\le (1+r^2)^{\gm/2} b(t),\quad 0<\gamma \le 2 .
\end{equation}

In this paper most of the results are concerned with the case
\begin{equation}\label{(1.4)}
B(z,\sg)=|z|^{\gm}b(\cos\theta),\quad 0<\gamma\le 2
\end{equation}
which corresponds to the so-called hard potential molecular
interactions.

The function $t\mapsto b(t)$ in \eqref{(1.3)}-\eqref{(1.4)} has some  weighted integrability. We shall
consider several options for the assumptions on $b(\cdot)$. Our strongest assumption
is that $b(\cdot)$ as a function of $\sigma$ is integrable on the sphere
$\mathbb{S}^{N-1}$, which means
\[
\int_{0}^{\pi}b(\cos\theta)\sin^{N-2}\theta\,{\rm d}\theta<\infty
\]
which is the Grad's angular cutoff.  However more singular situations can
be considered. The minimal assumption is that $b(\cos \theta) \sin^2 \theta$
is integrable on the sphere as a function of $\sigma$ (this
corresponds physically to an angular momentum), i.e.
\[
\int_{0}^{\pi}b(\cos\theta)\sin^{N}\theta\,{\rm d}\theta < \infty\,.
\]

In dimension $N=3$, it is well known that for the hard spheres model
the function $b(\cdot)$ is constant, whereas for hard potential models (without
angular cutoff), there is only weighted integrability:
\[
\int_{0}^{\pi}b(\cos\theta)\sin \theta\,{\rm d}\theta=\infty,\quad
\int_{0}^{\pi}b(\cos\theta)\sin^{2} \theta\,{\rm d}\theta<\infty\,.
\]
More precisely, given an interaction potential$\phi(r) = C \, r^{1-s}$ for
$C>0$ and $s>3$, we obtain the following formula
from the physics literature \cite{Ce88} in dimension $N=3$:
\begin{eqnarray*}&&
B(z,\sg) = |z|^{\gm} b(\cos\theta),\qquad  \gm=\frac{s-5}{s-1}\,;
\\
&&
b(\cos \theta) \sin\theta \sim C' \, \theta^{-1-\frac{2}{s-1}} \quad ( \theta\to 0^+ )
\end{eqnarray*}
for some constant $C'>0$, and hard potential interactions correspond
to $s>5$.

In this paper we consider the following different assumptions:
\begin{eqnarray*}
&{\bf (H0)} \quad &0<\gm\le 2\,,\quad A_2:=\left|\mathbb{S}^{N-2}\right|\int_{0}^{\pi}
b(\cos\theta)\sin^N\theta\,{\rm d}\theta<\infty, \\
&{\bf (H1)} \quad &0<\gm\le 2\,,\quad \int_{0}^{\pi}b(\cos\theta)\sin^N\theta\,(1+
|\log(\sin\theta)|){\rm d}\theta<\infty, \\
&{\bf (H2)} \quad &1<\gm<2\,,\quad \int_{0}^{\pi}b(\cos\theta)\sin^{N-2\nu}\theta\, {\rm
d}\theta<\infty\,,\quad \nu=2-2/\gm \in (0,1),\\
&{\bf (H3)} \quad &\gm=2\,,\,\,\, \exists \, p\in (1,\infty)\,\,\,\mbox{ s.t.
} \,\,\, \int_{0}^{\pi}[b(\cos\theta)]^{p}\sin^{N-2}\theta\,{\rm d}\theta
<\infty,\\
&{\bf (H4)} \quad & 0<\gm\le 2\,,\quad A_0:=\left|\mathbb{S}^{N-2}\right|\int_{0}^{\pi}
b(\cos\theta)\sin^{N-2}\theta\,{\rm d}\theta<\infty.
\end{eqnarray*}
Observe that ${\bf (H3)}|_b \Rightarrow {\bf (H4)}|_b \Rightarrow {\bf
  (H2)}|_b \Rightarrow {\bf (H1)}|_b \Rightarrow {\bf (H0)}|_b$, where
for instance ${\bf (H 3)}|_b$ denotes the assumption with respect to
$b(\cdot)$ in ${\bf (H3)}$. Note also that ${\bf (H3)}|_b$ and ${\bf
  (H4)}|_b$ corresponds to the angular cutoff case (short-range
interactions), whereas ${\bf(H0)}|_b, {\bf(H1)}|_b$ and ${\bf
  (H2)}|_b$ allow for non-locally integrable functions $b(\cdot)$ on
the sphere, i.e. non-cutoff cases (long-range interactions).

\subsubsection{Dual form of the collision operator}
\label{sec:dual-form-collision}

For any ${\bf n}\in\mathbb{S}^{N-1}$, let
\[\mathbb{S}^{N-2}({\bf n})=\{\og\in \mathbb{S}^{N-1}\,\,|\,\,\og\cdot {\bf n}=0\,\}\,
\quad (N\ge 3)\]
and in dimension $N=2$ let
\[
\mathbb{S}^{0}({\bf n})=\{ -{\bf n}^{\bot} \,,\,{\bf
n}^{\bot}\}\,\quad \mbox{where } {\bf n}^{\bot}\in \mathbb{S}^1 \mbox{
satisfies } {\bf n}^{\bot}\cdot {\bf n}=0\,.
\]
Then for any $g\in L^1(\mathbb{S}^{N-1})$ or $g\ge 0$ (measurable) on
$\mathbb{S}^{N-1}$ we have
$$\int_{\mathbb{S}^{N-1}}g(\sg){\rm d}\sg=\int_{0}^{\pi}\sin^{N-2}\theta
\left(\int_{\mathbb{S}^{N-2}({\bf n})}g(\cos\theta {\bf n}+\sin\theta\,\og){\rm
d}\og\right) {\rm d}\theta$$ where ${\rm d}\og$ is the Lebesgue spherical
measure on $\mathbb{S}^{N-2}({\bf n})$ and in case $N=2$ we define
$$\int_{\mathbb{S}^{0}({\bf n})}g(\og) {\rm d}\og=g(-{\bf n}^{\bot})+g({\bf
n}^{\bot})\,.$$ Let $|\mathbb{S}^{N-2}({\bf n})|=\int_{\mathbb{S}^{N-2}({\bf
n})}{\rm d}\og$, etc. Then $|\mathbb{S}^{N-2}({\bf n})|=|\mathbb{S}^{N-2}|$ for
$N\ge 3$, $|\mathbb{S}^{0}({\bf n})| =|\mathbb{S}^{0}|=2$ for $N=2$.

By classical calculation one has
\begin{multline}\label{(1.5)}\langle
 Q(f,g),\vp\rangle:=
\int_{\mathbb{R}^N}Q(f,g)(v)\vp(v){\rm d}v =\fr{1}{2}\intt_{\mathbb{R}^N
\times \mathbb{R}^N}L_B[\Dt\vp](v,v_*) f(v)g(v_*){\rm d}v{\rm d}v_*
\end{multline}
where
\[
\Dt\vp := \Dt\vp(v,v_*, v', v_*')=\vp(v')+\vp(v_*')-\vp(v)-\vp(v_*)\,,
\]
\begin{equation}\label{(1.6)}
L_B[\Delta \vp](v,v_*) :=
\int_{0}^{\pi}\bar{B}(|v-v_*|,\cos\theta)\sin^{N-2}\theta
\left(\int_{\mathbb{S}^{N-2} ({\bf n})} \Delta \vp \,{\rm d}\og\right){\rm
d}\theta
\end{equation}
and $\sg=\cos\theta\,{\bf n}+\sin\theta\,\og\,, {\bf
  n}=(v-v_*)/|v-v_*|$ for $v\neq v_*$; ${\bf n}={\bf
  e}_1=(1,0,\dots,0)$ for $v=v_*.$

Observe that when assuming one of the assumptions {\bf (H0)}, {\bf
  (H1)}, {\bf (H2)} (non-cutoff cases), the collision
operator in the dual form \eqref{(1.5)} above is well-defined thanks to the
cancellations in the symmetric difference $\Dt\vp$ of $\vp\in
C^2(\mathbb{R}^N)$. Basic estimates on $\Dt\vp$ are as follows (see for
instance \cite[Lemma~3.2]{MR2546739}): For all $(v,v_*,\sg)\in \mathbb{R}^N
\times \mathbb{R}^N \times \mathbb{S}^{N-1}$ one has
\begin{equation}\label{(1.7)}
|\Dt\vp|\le \sqrt{2}\left(\max_{|\xi|\le
\sqrt{|v|^2+|v_*|^2}}|\nabla\vp(\xi)|\right)|v-v_*|\sin\theta\,;
\end{equation}
\begin{equation}\label{(1.8)}
\left|\int_{\mathbb{S}^{N-2}({\bf n})}\Dt\vp \,{\rm d}\og\right|\le |{\mathbb
S}^{N-2}|\left(\max_{|\xi|\le
\sqrt{|v|^2+|v_*|^2}}|H_{\vp}(\xi)|\right)|v-v_*|^{2}\sin^2\theta\,,
\end{equation}
where $\nabla\vp$, $H_{\vp}$ are gradient and Hessian matrix of $\vp$.
Consequently the Boltzmann equation \eqref{(B)} in a weak form can be written
\begin{equation}\label{(1.9)}
\int_{\mathbb{R}^N}\vp(v)f_t(v){\rm d}v= \int_{\mathbb{R}^N}\vp(v)f_0(v){\rm
d}v+\int_{0}^{t} \langle Q(f_{\tau}, f_{\tau}),\vp\rangle
{\rm d}\tau \,.
\end{equation}
From the estimate \eqref{(1.8)} it is easily seen that if $A_2<\infty$
(minimal assumption) then $L_B[\Dt\vp]$ is well-defined for all
$\vp\in C^2(\mathbb{R}^N)$.

In fact we shall prove in Proposition~\ref{prop2.1} (see
Section~\ref{sec2}) that $(v,v_*)\mapsto L_B[\Dt\vp](v,v_*)$ is also
continuous on $\mathbb{R}^N \times \mathbb{R}^N$. Furthermore if
\[
\int_{0}^{\pi}b(\cos\theta)\sin^{N-1}\theta\,{\rm d}\theta<\infty
\]
then from the estimate \eqref{(1.7)} one sees that
\[
L_B[|\Dt\vp|](v,v_*)=\int_{\mathbb{S}^{N-1}}B(v-v_*,\sg)|\Dt\vp|\,{\rm
  d}\sg<\infty
\]
so that $L_B$ coincides with the simpler formula
\begin{equation}\label{(1.10)}
  L_B[\Dt\vp](v,v_*)=\int_{\mathbb{S}^{N-1}}B(v-v_*,\sg)\Dt\vp\,{\rm
    d}\sg\,.
\end{equation}

The collision integral \eqref{(1.5)} and the equation \eqref{(1.9)}
for $L^1$-functions are naturally extended to finite Borel
measures. For every $0\le s<\infty$, let ${\mathcal
  B}_s(\mathbb{R}^N)=({\mathcal B}_s(\mathbb{R}^N),\|\cdot\|_s)$ be
the Banach space of real Borel measures on $\mathbb{R}^N$ having
finite total variations up to order $s$, i.e.
$$\|\mu\|_s:=\int_{\mathbb{R}^N}\langle v\rangle^s {\rm d}|\mu|(v)<\infty,\quad
\langle v\rangle:=(1+|v|^2)^{1/2}$$ where the positive Borel measure $|\mu|$
is the total variation of $\mu$. In particular
$\|\mu\|=\|\mu\|_0=|\mu|(\mathbb{R}^N)$ is simply the total variation of
$\mu$.  Let
$${\mathcal B}_s^+(\mathbb{R}^N)
=\left\{ \mu\in {\mathcal B}_s(\mathbb{R}^N)\,|\,\, \mu\ge 0\right\}\,.$$ In
accordance with \eqref{(1.5)} we now define for every $\mu,\nu\in{\mathcal
  B}_s(\mathbb{R}^N)$ and every suitable smooth function $\vp$
\begin{equation}\label{(1.11)}
 \langle Q(\mu,\nu),\,\vp\rangle:=\fr{1}{2}
\intt_{\mathbb{R}^N \times \mathbb{R}^N}L_B[\Dt\vp](v,v_*) {\rm
  d}\mu(v){\rm d}\nu(v_*)\,.
\end{equation}
Our test function space for defining measure weak solutions is chosen $C^2
_b(\mathbb{R}^N) $, where
$$
C^k_b(\mathbb{R}^N)=\left\{\vp\in C^k(\mathbb{R}^N)\,\Big |\,\,\sum_{|\alpha
|\le k}
  \sup_{v\in\mathbb{R}^N}|\p ^{\alpha}\vp(v)|<\infty\,\right\}\,.$$
Finally by analogy with ${\mathcal B}_s(\mathbb{R}^N)$ we introduce the class
$L^{\infty}_{-s}(\mathbb{R}^N)$ of {\it locally bounded} Borel functions such
that
$$\psi\in L^{\infty}_{-s}(\mathbb{R}^N)\,\Longleftrightarrow\,
\|\psi\|_{L^{\infty}_{-s}}:=\sup_{v\in\mathbb{R}^N}|\psi(v)|\langle
v\rangle^{-s}<\infty$$ and we define
$$L^{\infty}_{-s}\cap C^k(\mathbb{R}^N)=\left\{\vp\in C^k(\mathbb{R}^N)\,\Big|\,\,\sum_{|\alpha |\le k}
\|\p ^{\alpha}\vp\|_{L^{\infty}_{-s}}<\infty\,\right\}\,,\quad s\ge
0, \ k \in \mathbb{N}\,.$$

\subsection{Previous results and references}

Let us give a short (and non-exhaustive) overview of the main
previous results and references related to the subject of this paper.

\subsubsection{Cauchy theory for the spatially homogeneous Boltzmann
  equation for hard potentials with cutoff}
The first rigorous mathematical result is due to Carleman
\cite{MR1555365,Carleman} who proved existence and uniqueness of
solutions in $L^1 \cap L^\infty$ with pointwise moment bounds, for
hard spheres interactions. A general Cauchy theory was later
developed by Arkeryd \cite{MR0339665,MR0339666} who proved existence
and uniqueness of solutions in $L^1 \cap L \log L$ with $L^1$ moment
bounds. More recently optimal results were obtained by Mischler and
Wennberg \cite{MR1697562} (see also Lu \cite{MR1716814}), and we refer
to the references therein for a more extensive bibliography.  \medskip

\subsubsection{Cauchy theory for the spatially homogeneous Boltzmann
  equation for hard potentials without cutoff}
This theory is much more recent, and not complete at now. As far as
existence of solutions is concerned let us mention the seminal works
of Villani \cite{MR1650006} and then Alexandre and Villani
\cite{MR1857879}. As far as uniqueness of solutions is concerned (in
the general far from equilibrium regime), let us mention the works
\cite{MR1675367,MR2283785,MR2511651,MR2398952} based on Wasserstein
metrics and probabilistic tools, and the work \cite{MR2525118} based
on \textit{a priori} estimates. Finally let us mention the related
recent works in the {\em perturbative close-to-equilibrium regime}
(but without assuming spatial homogeneity) of Gressman and Strain
\cite{MR2629879} on the one hand, and Alexandre, Morimoto, Ukai, Xu,
Yang \cite{MR2677982} on other hand.  \medskip

\subsubsection{Polynomial moment bounds}
The first seminal result of the propagation of polynomial moments that
exists initially for ``variable hard spheres'' (hard potentials with
angular cutoff) is due to Elmroth \cite{MR684411} and makes use of
so-called ``Povzner's inequalities'' \cite{MR0142362}. Then
Desvillettes \cite{MR1233644} proved, for the same model, the
appearance of any polynomial as soon as a moment of order strictly
higher than $2$ exists initially (see also \cite{MR1264851}). Finally
optimal results were obtained in \cite{MR1697562} again.

\subsubsection{Exponential moment bounds}
The first seminal result of propagation of moments of exponential form is due
to Bobylev \cite{MR1478067}, still in the case of short-ranged interactions.
Significant improvements of these results were later obtained in
\cite{MR2096050}. Let us also mention the related result of propagation of
pointwise Maxwellian bound in \cite{MR2533928}. Inspired by the same
techniques, the appearance of exponential moments was first obtained by the
second author together with Mischler in \cite{MR2264623,Mcmp}, see
also the recent work \cite{ACGM}.

\subsection{Definitions of measure solutions}

Let us start with a notion of \emph{measure weak solutions}, where the
time evolution is defined in the integral sense.

\begin{definition}[Measure weak solutions]\label{def:measure-weak}
  Let $B(z,\sg)$ be given by  \eqref{(1.B)}-\eqref{(1.2)}-\eqref{(1.3)}
  with $\gm$ and $b(\cdot)$ satisfying {\bf (H0)}.
Let
  $F_0\in {\mathcal B}_2^{+}(\mathbb{R}^N)$ and $\{F_t\}_{t\ge
    0}\subset {\mathcal B}_2^{+}(\mathbb{R}^N)$. We say that
  $\{F_t\}_{t\ge 0}$, or simply $F_t$, is a \textbf{measure weak
    solution} of Eq.~\eqref{(B)} associated with the initial datum
  $F_0$, if it satisfies the following (i)-(ii):
  \begin{itemize}
  \item[(i)] $\sup\limits_{t\ge 0}\|F_t\|_2<\infty\,.$
\item[(ii)] For every $\vp\in C^2_b(\mathbb{R}^N)$,
\[
\left\{
\begin{array}{l}\displaystyle
\intt_{\mathbb{R}^N \times \mathbb{R}^N}|L_B[\Dt\vp](v,v_*)| {\rm
d}F_{t}(v){\rm d}F_{t}(v_*) <\infty\,,\quad \forall\, t>0 \vspace{0.3cm} \\
\displaystyle t\mapsto \langle Q(F_t, F_t),\,\vp\rangle \quad {\rm belongs
\,\,to}\quad C((0,\infty))\cap L^1_{{\rm loc}}([0,\infty)) \vspace{0.3cm} \\
\displaystyle \int_{\mathbb{R}^N}\vp(v){\rm d}F_t(v)=
\int_{\mathbb{R}^N}\vp(v){\rm d}F_0(v)+\int_{0}^{t}\langle Q(F_{\tau},
F_{\tau}),\,\vp\rangle {\rm d}\tau \quad \forall\, t\ge 0\,.$$
\end{array}
\right.
\]
\end{itemize}

Moreover a measure weak solution $F_t$ is called a
\textbf{conservative solution} if it conserves the mass, momentum and
energy, i.e.
$$\int_{\mathbb{R}^N}\left(\begin{array}{lll} 1
    \\
    v\\
    |v|^2\end{array}\right){\rm d}F_t(v)= \int_{\mathbb{R}^N}\left(\begin{array}{lll} 1
    \\
    v\\
    |v|^2\end{array}\right){\rm d}F_0(v)\qquad \forall\, t\ge 0\,.$$
\end{definition}

Note that every measure weak solution conserves the mass because the
constant $\vp=1$ belongs to $C^2_b(\mathbb{R}^N)$ and $\Dt\vp=0$. The
conservations of the momentum and energy are formally true since one
also has $\Dt\vp=0$ for $\vp=v_j$, $j=1,2,\dots, N$ and $\vp = |v|^2$,
but these $\vp$ do not belong to $C^2_b(\mathbb{R}^N)$. In fact under
the assumption {\bf (H1)}, one can follow the same argument in
\cite{MR1893976} to construct a weak solution of Eq.~\eqref{(B)} such
that the energy is increasing.

Now let us consider a stronger notion of \emph{measure strong
  solutions}
under the angular cutoff assumption {\bf (H4)}.  Let $B(z,\sg)$ be given by
\eqref{(1.B)}-\eqref{(1.2)}-\eqref{(1.3)} with $b(\cdot)$ satisfying $A_0<\infty$. Then we can
define bilinear operators (see Proposition \ref{prop1.4} below)
\[
Q^{\pm}: {\mathcal B}_{s+\gm}(\mathbb{R}^N)\times{\mathcal
  B}_{s+\gm}(\mathbb{R}^N)\to {\mathcal B}_s(\mathbb{R}^N)
\quad  (s\ge 0)
\]
and
\begin{equation}\label{(1.14)}
Q(\mu,\nu):=Q^{+}(\mu,\nu)-Q^{-}(\mu,\nu)
\end{equation}
through Riesz's representation theorem by
\begin{equation}\label{(1.15)}
\int_{\mathbb{R}^N}\psi(v){\rm d}Q^{+}(\mu,\nu)(v) =
\intt_{\mathbb{R}^N \times \mathbb{R}^N}L_B[\psi](v,v_*)
{\rm d}\mu(v){\rm d}\nu(v_*)\,,
\end{equation}
\begin{equation}\label{(1.16)}
\int_{\mathbb{R}^N}\psi(v){\rm d}Q^{-}(\mu,\nu)(v) = \intt_{\mathbb{R}^N
\times \mathbb{R}^N} A(v-v_*)\psi(v)d \mu(v){\rm d}\nu(v_*)
\end{equation}
for all $\psi\in L^{\infty}_{-s}\cap C(\mathbb{R}^N)$, where
\begin{equation}\label{(1.17)}
L_B[\psi](v,v_*)=\int_{\mathbb{S}^{N-1}}B(v-v_*,\sg) \psi(v')\,{\rm
d}\sg,\quad A(z) =\int_{{\mathbb S}^{N-1}}B(z,\sg){\rm d}\sg
\end{equation}
and recall that ${\bf n} =
(v-v_*)/|v-v_*|$ in $b({\bf n}\cdot \sg)$ is replaced
by a fixed unit vector ${\bf e}_1$ for $v=v_*$.

Recall that the norm $\|\mu\|_s$ of $\mu\in {\mathcal
  B}_s(\mathbb{R}^N)\,(s\ge 0)$ can be estimated in terms of compactly
smooth test functions: For all $k\ge 0$
\begin{equation}
\|\mu\|_s=\sup_{\vp\in C_c^{k}(\mathbb{R}^N),\, \|\vp\|_{L^{\infty}_{-s}}\le
1}\Big|\int_{\mathbb{R}^N}\vp{\rm
d}\mu\Big|\,.
\label{(5.1)}
\end{equation}

We are now ready for stating the definition of \emph{measure strong
solutions}, for which some time-differentiability is assumed in total
variation topology.

\begin{definition}[Measure strong solutions] Let $B(z,\sg)$ be given
  by \eqref{(1.B)}-\eqref{(1.2)}-\eqref{(1.3)} with $\gm$ and
  $b(\cdot)$ satisfying {\bf (H4)}. Let $F_0\in {\mathcal
    B}_2^{+}(\mathbb{R}^N)$ and $\{F_t\}_{t\ge 0}\subset {\mathcal
    B}_2^{+}(\mathbb{R}^N)$. We say that $F_t$ is a {\bf measure strong
  solution} of Eq.{\rm \eqref{(B)}} associated with the initial datum
  $F_{t}|_{t=0}=F_0$, if it satisfies the following (i)-(ii):
\begin{itemize}
\item[(i)] $\sup\limits_{t\ge 0}\|F_t\|_2<\infty\,.$

\item[(ii)] $t\mapsto F_t\in C([0,\infty);{\mathcal B}_2(\mathbb{R}^N))\cap
C^1([0,\infty);{\mathcal B}_0(\mathbb{R}^N))$ and
\begin{equation}\label{def-strong}
\fr{{\rm d}}{{\rm d}t}F_t=Q(F_t,F_t)\,,\quad
t\in[0,\infty)\,.
\end{equation}
\end{itemize}
\end{definition}

Note that from \eqref{(5.9)}-\eqref{(5.10)}-\eqref{(5.11)} in
Proposition~\ref{prop1.4}  the strong continuity of
\[
t\mapsto F_t\in C([0,\infty);{\mathcal B}_2(\mathbb{R}^N))
\]
implies the strong continuity $t\mapsto Q(F_t,F_t) \in C([0,\infty);{\mathcal
B}_0(\mathbb{R}^N))$, so that the differential equation \eqref{def-strong} is
equivalent to the integral equation
\begin{equation}\label{def-strong-bis}
F_t=F_0+\int_{0}^{t}Q(F_s, F_s){\rm d}s\,,\quad t\ge 0
\end{equation}
where the integral is taken in the Riemann sense or generally in the
Bochner sense.  Recall also that here the derivative $\fr{{\rm d}
}{{\rm d}t}\mu_t$ and integral $\int_{a}^{b}\nu_t{\rm d}t$ as measures
are defined by
$$\left(\fr{{\rm d} }{{\rm d}t}\mu_t\right)(E)=\fr{{\rm d}}{{\rm d}t} \mu_t(E)\,,\quad
\left(\int_{a}^{b}\nu_t {\rm d}t\right)(E)= \int_{a}^{b}\nu_t (E) {\rm d}t$$ for
all Borel sets $E\subset \mathbb{R}^N$.

Note also that if a strong measure solution $F_t$ is absolutely continuous
with respect to the Lebesgue measure for all $t\ge 0$, i.e. ${\rm
d}F_t(v)=f_t(v){\rm d}v$, then it is easily seen that $f_t$ (after
modification on a $v$-null set) is a mild solution of Eq.~\eqref{(B)}. That
is, $(t,v)\mapsto f_t(v)$ is nonnegative and Lebesgue measurable on
$[0,\infty)\times\mathbb{R}^N$ and for every $t\ge 0$, $v\mapsto f_t(v)$
belongs to $L^1_2(\mathbb{R}^N)$, $\sup_{t\ge 0}\|f_t\|_{L^1_2}<\infty$, and
there is a Lebesgue null set $Z_0\subset \mathbb{R}^N$ (which is independent
of $t$) such that
\begin{equation*}
\left\{
\begin{array}{l}\displaystyle
\int_{0}^{t}Q^{\pm}(f_{\tau},f_{\tau})(v){\rm d}\tau<\infty\quad
\forall\,t\in[0,\infty)\,,\quad \forall\, v\in \mathbb{R}^N\setminus
 Z_0\vspace{0.3cm} \\ \displaystyle
f_t(v)=f_0(v)+\int_{0}^{t}Q(f_{\tau},f_{\tau})(v){\rm d}\tau\,,\quad
\forall\,t\in[0,\infty) \,, \quad \forall\,v\in \mathbb{R}^N\setminus
 Z_0\,.
\end{array}
\right.
\end{equation*}
Here
$$L^1_s(\mathbb{R}^N)=\left\{\,f\in L^1(\mathbb{R}^N)\,|\,
  \|f\|_{L^1_s}:=\int_{\mathbb{R}^N}|f(v)|\langle v\rangle^s {\rm
    d}v<\infty\,\right\}\,,\quad s\ge 0\,.$$ From classical measure
theory \cite[Theorem 6.13, page 149]{MR0210528}: if ${\rm d}\mu(v)=f(v){\rm
d}v$ for $f\in L^1_s(\mathbb{R}^N)$, then ${\rm d}|\mu|(v)=|f(v)|{\rm d}v$ and
hence $\|\mu\|_s=\|f\|_{L^1_s}$.

For any positive measure $\mu\in {\mathcal B}_2^{+}(\mathbb{R}^N)$ we finally
introduce the following continuous function $r\mapsto\Psi_{\mu}(r)$ on
$[0,\infty)$:
\begin{equation}\label{(1.20)}
\Psi_{\mu}(r)=r+r^{1/3}+ \int_{|v|>r^{-1/3}}|v|^2{\rm d}\mu(v)\,,\quad
r>0 \quad \mbox{ with } \quad
\Psi_{\mu}(0)=0
\end{equation}
which quantifies the localization of the energy of $\mu$.

\subsection{Main results}
\label{sec:main-results}

Our first main result is the following

\begin{theorem}[Existence of solutions and moment production
  estimates without cutoff]\label{theo1}
  Suppose that $B(z,\sg)=|z|^{\gm}b(\cos\theta)$ satisfies {\bf
    (H1)}. Given any initial datum $F_0\in {\mathcal
    B}^{+}_2(\mathbb{R}^N)$ with $\|F_0\|_0\neq 0$, we have
\begin{itemize}
\item[(a)] The Eq.~\eqref{(B)} always has a conservative measure weak
  solution $F_t$ satisfying $F_t|_{t=0}=F_0$.

 \item[(b)] Let $F_t$ be a measure weak
  solution of Eq.~\eqref{(B)} associated with the initial datum $F_0$ satisfying
\begin{equation}\label{(1.E)}\|F_t\|_2\le
  \|F_0\|_2\quad\forall\, t>0;\quad
\sup\limits_{t\ge t_0}\|F_t\|_{s}<\infty \quad \forall \, t_0>0, \ \forall \,
s>2. \end{equation} Then $F_t$ is conservative, i.e. $F_t$ conserves the mass,
momentum, and energy.

\item[(c)] The Eq.~\eqref{(B)} always has a conservative measure weak
  solution $F_t$ with $F_t|_{t=0}=F_0$ which satisfies
the following moment production estimate:
\begin{equation}\label{(1.12)}
\|F_t\|_{s}\le {\mathcal K}_s(F_0) \left(1+\frac{1}{t}
\right)^{\fr{s-2}{\gm}}\quad \forall\,t>0\,,\quad \forall\,
s\ge 2
\end{equation}
where
\begin{equation}\label{(1.12*)}
{\mathcal K}_s(F_0)=\|F_0\|_2\left(2^{s+7}\fr{\|F_0\|_2}{\|F_0\|_0}
\left(1+\frac{1}{16\|F_0\|_2A_2\gm}\right)\right)^{\fr{s-2}{\gm}}\,.
\end{equation}

\item[(d)] If in addition either $0<\gm\le 1$ or one of the
  assumptions {\bf (H2)}, {\bf (H3)} is satisfied, then every solution
  $F_t$ in part (c) (or generally in part (b)) satisfies a moment
  production estimate of exponential form:
\begin{equation}\label{(1.13)}
 \int_{\mathbb{R}^N} e^{\alpha(t)\langle v\rangle ^{\gm}} {\rm d}F_t(v)\le
2\|F_0\|_0\qquad \forall\,t>0\,,
\end{equation}
where \begin{eqnarray*}&& \alpha(t)=2^{-s_0}
\fr{\|F_0\|_0}{\|F_0\|_2}(1-e^{-\beta t})\,,\quad \beta =16\|F_0\|_2A_2\gm>0
\end{eqnarray*} and $1<s_0<\infty$
depends only on $b(\cdot)$ and $\gm$.
\end{itemize}
\end{theorem}

It is possible to deduce from the previous theorem some more
conventional moment estimates in exponential form where the constant
in the argument of the exponential moment remains time-dependent:
\begin{corollary}\label{corollary 1.4}
  Under the same assumptions on $B(z,\sg)$ and the initial datum $F_0$
  in Theorem~\ref{theo1}, there exists a conservative measure weak solution
  $F_t$ of Eq.~\eqref{(B)} such that for any $0<s<\gm$ and any $c>0$
  $$\int_{\mathbb{R}^N}
  e^{c\langle v\rangle^{s}} {\rm d}F_t(v)\le
  \left(e^{\alpha_s(t)}+2\right)\|F_0\|_0\qquad \forall\, t>0$$ where
$$\alpha_s(t)=c\left(\fr{c}{\alpha(t)}\right)^{\fr{s}{\gm-s}}\,.$$
\end{corollary}

\begin{proof}[Proof of Corollary~\ref{corollary 1.4}]  The proof of this Corollary is quite short and we can present
it here. As a consequence of Theorem~\ref{theo1} there exists a conservative
measure weak solution $F_t$ of Eq.~\eqref{(B)} such that $F_t$ satisfies
\eqref{(1.13)}. For any $t>0$, by definition of $\alpha_s(t)$ and $0<s<\gm$ we
have
$$c\langle v\rangle^{s}>\alpha_s(t)\,\,\,\Longrightarrow\,\,\,
c\langle v\rangle^{s}=\left(\fr{\alpha_s(t)}{c\langle
v\rangle^s}\right)^{\fr{\gm-s}{s}}\alpha(t)\langle
v\rangle^{\gm}<\alpha(t)\langle v\rangle^{\gm}\,.$$ Thus
\begin{eqnarray*}
  && \int_{\mathbb{R}^N}e^{c\langle v\rangle^{s}} {\rm
    d}F_t(v) =\int_{\{c\langle v\rangle^s\le \alpha_s(t)\}}
    e^{c\langle v\rangle^{s}} {\rm d}F_t(v)+
  \int_{\{c\langle v\rangle^s> \alpha_s(t)\}}e^{c\langle v\rangle^{s}}
  {\rm d}F_t(v)
  \\
  &&\le e^{\alpha_s(t)}\|F_0\|_0+ \int_{\{c\langle v\rangle^s>
    \alpha_s(t)\}}e^{\alpha(t)\langle v\rangle^{\gm}} {\rm d}F_t(v)\le
    e^{\alpha_s(t)}\|F_0\|_0
  +2\|F_0\|_0\,.
\end{eqnarray*}
\end{proof}

Our second main result of this paper is

\begin{theorem}[Uniqueness and stability estimates for locally
  integrable $b(\cdot)$]\label{theo2}
  Let $B(z,\sg)=|z|^{\gm}b(\cos\theta)$ satisfy {\bf (H4)}. Given any
  initial datum $F_0\in {\mathcal B}^{+}_2(\mathbb{R}^N)$ with $\|F_0\|_0\neq 0$, we have

\begin{itemize}

\item[(a)] Every conservative measure weak solution of Eq.~\eqref{(B)}
  is a strong solution, while every measure strong solution of
  Eq.~\eqref{(B)} is a measure weak solution.

\item[(b)] Let $F_t$ be a measure strong solution of Eq.~\eqref{(B)}
  with the initial datum $F_0$ satisfying $\|F_t\|_2\le \|F_0\|_2$ for all
  $t\ge 0.$ Then $F_t$ in fact conserves the mass, momentum and energy.

\item[(c)]  There exists a unique
conservative measure strong solution $F_t$ of Eq.~\eqref{(B)} such that
$F_t|_{t=0}=F_0$. Therefore $F_t$ satisfies the
moment production estimates in Theorem \ref{theo1} .

\item[(d)] Let $F_t$ be the unique conservative measure strong solutions of Eq.~\eqref{(B)}
 with the initial datum $F_0$ and
 let $G_t$ be a conservative measure strong solutions of
Eq.~\eqref{(B)} on the time interval $[\tau, \infty)$ with an initial datum
$G_t|_{t=\tau}=G_\tau\in{\mathcal B}^{+}_2(\mathbb{R}^N)$ for some $\tau\ge
0$. Then:
\begin{itemize}
\item If $\tau=0$, then
\begin{equation}\label{(1.21)}
\|F_t-G_t\|_{2}\le  \Psi_{F_0}(\|F_0-G_0\|_2)e ^{C(1+t)},\quad t\ge 0
\end{equation}
where $\Psi_{F_0}$ is given by \eqref{(1.20)}, $C={\mathcal R}(\gm, A_0,
A_{2}\,\|F_0\|_0, \|F_0\|_2)$ is an explicit positive continuous function on
$(\mathbb{R}_{>0})^5$.

\item If $\tau>0$, then
\begin{equation}\label{(1.21*)}
\|F_t-G_t\|_{2}\le\|F_\tau-G_\tau\|_2 e^{c_{\tau}(t-\tau)},\,\quad t\in [\tau,
\infty)
\end{equation}
where $c_{\tau}= 4A_0({\mathcal K}_{2+\gm}(F_0)+\|F_0\|_2)(1+\fr{1}{\tau})$,
${\mathcal K}_{2+\gm}(F_0)$ is given in (\ref{(1.12*)}) with $s=2+\gm$.
\end{itemize}

\item[(e)] If $F_0$ is absolutely continuous with respect to the
  Lebesgue measure, i.e. ${\rm d}F_0(v)=f_0(v){\rm d}v$ with $0\le
  f_0\in L^1_2(\mathbb{R}^N)$, then the unique conservative measure
  strong solution $F_t$ with the initial datum $F_0$ is also
  absolutely continuous with respect to the Lebesgue measure: ${\rm
    d}F_t(v)=f_t(v){\rm d}v$ for all $t\ge 0$, and $f_t$ is the unique
  conservative mild solution of Eq.~\eqref{(B)} with the initial datum
  $f_0$.

\item[(f)] If $F_0$ is not a Dirac mass and let $F_t$ be the unique
  measure strong solution of Eq.~\eqref{(B)} with the initial datum
  $F_0$, then there is a sequence $\{f^n_t\}$ of conservative
  $L^1$-solutions of Eq.~\eqref{(B)} with initial data $0\le f^n_0\in
  L^1_2(\mathbb{R}^N)$ satisfying
\begin{equation}\label{(1.22)}
  \int_{\mathbb{R}^N} \left( \begin{array}{c} 1 \\ v \\
      |v|^2 \end{array} \right) f^n_0(v) {\rm d}v=\int_{\mathbb{R}^N}
  \left( \begin{array}{c} 1 \\ v \\
      |v|^2 \end{array} \right)  {\rm
    d}F_0(v)\,,\quad  n=1,2,\dots
\end{equation}
such that
\begin{equation}\label{(1.23)}
\lim_{n\to\infty}\int_{\mathbb{R}^N}\vp(v) f^n_t(v){\rm d}v
=\int_{\mathbb{R}^N}\vp(v){\rm d}F_t(v)\qquad \forall\,\vp\in C_b(\mathbb{R}^N),\quad
\forall\,t\ge 0\,.
\end{equation}
\end{itemize}
\end{theorem}

\begin{remark}The trivial case $\|F_0\|_0=0$, i.e. $F_0=0$, is
  excluded from the above theorems since $F_0=0$ implies that
  $F_t\equiv 0$ is the unique conservative measure solution of
  Eq.~\eqref{(B)}.
\end{remark}

\begin{remark} An application of the estimate \eqref{(1.21*)} for solutions
with different initial times  will be seen in our next paper concerning the
rate of convergence to equilibrium.
\end{remark}

\begin{remark} In the second part of this work we shall prove the
  exponential convergence to equilibrium (for bounded angular function
  $b(\cdot)$): $\|F_t-M\|_0\le Ce^{-ct}$ where $M$ is the
  Maxwellian (Gaussian) with the same mass, momentum and energy as
  $F_0$ (assuming that $F_0$ is not a single Dirac mass and
  $\|F_0\|_0\neq 0$), $C, c>0$ are constants depending only on
  $N$, $b(\cdot)$, $\gm$ and the mass, momentum and energy of
  $F_0$. This result will allow us to improve the stability estimate
  \eqref{(1.21)} to be uniform in time:
  \[\sup_{t\ge
    0}\|F_t-G_t\|_{2}\le {\wt \Psi}_{F_0}(\|F_0-G_0\|_2)
  \]
  for some explicit continuous function ${\wt \Psi}_{F_0}(r)$ on
  $[0,\infty)$ satisfying ${\wt \Psi}_{F_0}(0)=0$.
\end{remark}

\subsection{Strategy and plan of the paper}
\label{sec:strategy-plan-paper}

We shall first in Section~\ref{sec2} prove some continuity and Lipschitz
estimates on the collision operator $Q$ in (weighted) total variation
topology. In Section~\ref{sec3} we shall prove moment estimates, first on the
kernel $L_B$ and then on the collision operator $Q$, plus several technical
lemmas on fractional binomial expansions, on the beta function and on some ODE
estimates. After these two sections which remain purely at the level of
functional inequalities, we shall start considering the time evolution problem
and tackle the proof of the first main Theorem~\ref{theo1} in
Section~\ref{sec4}: the main step in the construction of weak measure
solutions is based on an approximation argument with the help of the Mehler
transform, and the moment estimates on the solutions will be proved with the
help of the functional results in the previous section. Finally in
Section~\ref{sec5} we shall prove the second main Theorem~\ref{theo2} by
carefully revisiting the uniqueness estimates known for functions in the case
of measures.

\section{Regularity estimates on the collision operator}
\label{sec2}

We shall prove in this section some continuity and Lipschitz estimates on the
collision operator in the (weighted) total variation topology. It will be
useful for defining measure weak solutions of Eq.~\eqref{(B)} as we mentioned
in Section~\ref{sec1}, but also for proving weak convergence of approximate
solutions, which leads to the existence of measure weak solutions. We start
with a preliminary useful representation of the collision velocities.

\subsection{Representations of $\langle v'\rangle^2, \langle v_*'\rangle^2$}

We first begin this section with a preliminary technical computation.

For any $v,v_*\in\mathbb{R}^N$, let us define
$${\bf h} = \fr{v+v_*}{|v+v_*|}\quad {\rm for}\quad v+v_*\neq
0\,;\quad {\bf h}={\bf e}_1=(1,0,...,0)\quad {\rm for}\quad v+v_*=0$$ and
recall that ${\bf n}= (v-v_*)/|v-v_*|$ when $v \not= v_*$ and ${\bf n}={\bf
e}_1$ else. By  \eqref{(1.v)} we have
\begin{equation}\label{(1.24)}
\left\{
\begin{array}{l} \displaystyle
\langle v'\rangle^2 := 1+|v'|^2 = \fr{\langle v\rangle ^2+\langle
v_*\rangle^2}{2} +
\fr{|v+v_*||v-v_*|}{2}({\bf h}\cdot \sg) \vspace{0.3cm} \\
\displaystyle \langle v_*'\rangle^2 :=  1+ |v'_*|^2 = \fr{\langle v\rangle
^2+\langle v_*\rangle^2}{2} - \fr{|v+v_*||v-v_*|}{2}({\bf h}\cdot \sg) \,.
\end{array}
\right.
\end{equation}
Let us also define the unit vector
\[
{\bf j}=\fr{{\bf
    h}-({\bf h}\cdot{\bf n}){\bf n}} {\sqrt{1-({\bf h}\cdot{\bf
      n})^2}} \mbox{ for } |{\bf h}\cdot{\bf n}|<1 \mbox{ and } {\bf
  j}= {\bf e}_1 \mbox{ for } |{\bf h}\cdot{\bf n}|=1 \,.
\]
Then with the change of variables
$\sg=\cos\theta {\bf n}+\sin\theta
\,\og\,,\,\og\in\mathbb{S}^{N-2}({\bf n})$, we have $${\bf h}\cdot \sg
= ({\bf h}\cdot {\bf n})\cos\theta + \sqrt{1-({\bf h}\cdot{\bf
    n})^2}\,\sin\theta\,({\bf j}\cdot\og)\,, \quad
\og\in\mathbb{S}^{N-2}({\bf n})$$
so that we get another representation:
\begin{equation}\label{(1.25)}
\left\{
\begin{array}{l} \displaystyle
\langle v'\rangle^2=\langle v\rangle ^2\cos^2\theta/2+\langle
v_*\rangle^2\sin^2\theta/2 + \sqrt{|v|^2|v_*|^2-(v\cdot v_*)^2}\,
\sin\theta\,({\bf j}\cdot\og) \vspace{0.3cm}\\ \displaystyle
\langle v_*'\rangle^2=\langle v\rangle ^2\sin^2\theta/2+\langle
v_*\rangle^2\cos^2\theta/2 - \sqrt{|v|^2|v_*|^2-(v\cdot
v_*)^2}\, \sin\theta\,({\bf j}\cdot\og)\,.
\end{array}
\right.
\end{equation}

\subsection{Continuity estimate on the collision operator}

\begin{proposition}[Continuity of the collision operator]\label{prop2.1}
  Let $B(z,\sg)$ be given by \eqref{(1.B)}-\eqref{(1.2)}-\eqref{(1.3)}
  with $b(\cdot)$ satisfying {\bf (H0)}.  Then
\begin{itemize}
\item[(I)] The function $(v,v_*)\mapsto L_B[\Dt\vp](v,v_*)$ is continuous on
  $\mathbb{R}^N \times \mathbb{R}^N$ for all $\vp \in
  C^2(\mathbb{R}^N)$.

\item[(II)] Let $B_n(z,\sg)=\bar{B}_n(|z|,\cos\theta)$
satisfy \eqref{(1.2)} and
\begin{equation}\label{(2.B1)}
\bar{B}_n(r,t)\nearrow \bar{B}(r,t)\quad (n\to\infty)\quad
\forall\,(r,t)\in[0,\infty)\times(-1,1).
\end{equation}
  Then for any
$\vp\in C^2(\mathbb{R}^N)$ and any $0<R<\infty$
\begin{equation}\label{(2.B2)}
\sup_{|v|+|v_*|\le R}|L_{B_n}[\Dt\vp](v,v_*)-L_{B}[\Dt\vp](v,v_*) |\to 0\quad
(n\to\infty).
\end{equation}
Moreover let $\vp_n\in C^2(\mathbb{R}^N)$ satisfy
\begin{equation}\label{(2.B3)}
\lim_{n\to\infty}\vp_n(v)= \vp(v)\,\,\,\forall\, v \in\mathbb{R}^N\,;\quad
\sup_{n\ge 1}\sup_{|v|\le R}\sum_{|\alpha|\le 2}|{\p
^{\alpha}}\vp_n(v)|<\infty\quad \forall\,R<\infty \,.
\end{equation}
Then
\begin{equation}\label{(2.B4)}
L_{B_n}[\Dt\vp_n](v,v_*)\to L_{B}[\Dt\vp](v,v_*)\quad (n\to\infty)\quad
\forall\, (v,v_*)\in\mathbb{R}^N\times\mathbb{R}^N\,.
\end{equation}
\end{itemize}
\end{proposition}

\begin{proof}[Proof of Proposition~\ref{prop2.1}] Let us write
\begin{equation}\label{(2.B5)}
L_{B}[\Dt\vp](v,v_*)=\int_{0}^{\pi}\bar{B}(|v-v_*|,\cos\theta)
\sin^{N}\theta\, L[\Dt\vp](v,v_*,\theta){\rm d}\theta
\end{equation}
where \begin{eqnarray*}&&
L[\Dt\vp](v,v_*,\theta)=\fr{1}{\sin^2\theta}\int_{\mathbb{S}^{N-2}({\bf
n})}\Dt\vp \,{\rm d}\og,\quad 0<\theta<\pi.\end{eqnarray*} Recalling
\eqref{(1.8)} we have
\begin{equation}\label{(2.B6)}\sup_{0<\theta<\pi}|L[\Dt\vp](v,v_*,\theta)|\le
|\mathbb{S}^{N-2}|\left(\max_{|\xi|\le \sqrt{|v|^2+|v_*|^2}}|H_{\vp}(\xi)|\right)|v-v_*|^{2}.
\end{equation}
\medskip

\noindent
{\bf Part~(I).}   For any $0<R<\infty$, consider decomposition
$$B(z,\sg)=B (z,\sg)\wedge R+(B(z,\sg)-R )^{+}$$
where $x\wedge y=\min\{x,y\}, (x-y)^+=\max\{x-y,0\}$. We have
\begin{eqnarray*}&& L_{B}[\Dt\vp](v,v_*)=L_{B\wedge
R}[\Dt\vp](v,v_*)+L_{(B-R)^{+}}[\Dt\vp](v,v_*),\\
&& L_{B\wedge R}[\Dt\vp](v,v_*) =\int_{\mathbb{S}^{N-1}}[B(v-v_*,\sg)\wedge
R]\Dt\vp\,{\rm d}\sg.\end{eqnarray*} Fix any
$(v_0,{v_*}_0)\in\mathbb{R}^N\times\mathbb{R}^N$. Applying
\eqref{(2.B5)}-\eqref{(2.B6)} to $L_{(B-R)^{+}}[\Dt\vp]$ and recalling the
assumption \eqref{(1.3)} we have
$$
\sup_{|v-v_0|^2+|v_*-{v_*}_0|^2\le 1}|L_{(B-R)^{+}}[\Dt\vp](v,v_*)|\le
C_{\vp}\int_{0}^{\pi}\Big(C_{\gm}b(\cos\theta) - R\Big)^{+}\sin^{N}\theta\,{\rm
d}\theta=: I_{\vp,\gm}(R)$$ where  $C_{\vp}, C_{\gm}$ are finite constants
depending only on $\vp,\gm, {v_0},{v_*}_0$. Therefore
\begin{multline}\label{(2.B7)}
 |L_B[\Dt\vp](v,v_*)- L_B[\Dt\vp](v_0,{v_*}_0)|
\\
\le |L_{B\wedge R}[\Dt\vp](v,v_*)- L_{B\wedge
R}[\Dt\vp](v_0,{v_*}_0)|+I_{\vp,\gm}(R)\quad
\forall\,|v-v_0|^2+|v_*-{v_*}_0|^2\le 1. \end{multline} Let
$(\Dt\vp)_0=\vp({v_0}')+\vp({v_*}_0')-\vp(v_0)-\vp({v_*}_0).$ Applying
\eqref{(2.B5)} to $L_{B\wedge R}[\Dt\vp]$ and using the assumption
\eqref{(1.2)} we have
 \begin{eqnarray*}&& |L_{B\wedge
R}[\Dt\vp](v,v_*)- L_{B\wedge
R}[\Dt\vp](v_0,{v_*}_0)|\\
&&\le C_{\vp}|\mathbb{
S}^{N-2}|\int_{0}^{\pi}\Big|\bar{B}(|v-v_*|,\cos\theta)\wedge
R-\bar{B}(|v_0-{v_*}_0|,\cos\theta)\wedge R\Big|\sin^{N-2}\theta\,{\rm d}\theta\\
&&+R\int_{\mathbb{S}^{N-1}}\Big|\Dt\vp- (\Dt\vp)_0\Big|\,{\rm d}\sg \to 0\quad
{\rm as}\quad (v,v_*)\to (v_0,{v_*}_0).\end{eqnarray*}  Also by assumption
$\int_{0}^{\pi}b(\cos\theta)\sin^{N}\theta\,{\rm d}\theta <\infty$ we have
$I_{\vp,\gm}(R)\to 0$ as $R\to+\infty$. Thus from \eqref{(2.B7)}, by first
letting $(v,v_*)\to (v_0,{v_*}_0)$ and then letting $R\to+\infty$, we obtain
$$\limsup_{(v,v_*)\to
(v_0,{v_*}_0)}|L_B[\Dt\vp](v,v_*)- L_B[\Dt\vp](v_0,{v_*}_0)|=0\,.$$
\medskip

\noindent
{\bf Part~(II).} By  assumption \eqref{(2.B1)} and \eqref{(1.3)} we have
\[
\bar{B}_n(r,\cos\theta)\le
\bar{B}_{n+1}(r,\cos\theta)\le\bar{B}(r,\cos\theta)\le
(1+r^2)^{\gm/2}b(\cos\theta)
\]
which together with \eqref{(1.2)} implies that the functions
$$r\mapsto \int_{0}^{\pi}\bar{B}_n(r,\cos\theta)\sin^N\theta\,{\rm
d}\theta,\quad r\mapsto \int_{0}^{\pi}\bar{B}(r,\cos\theta)\sin^N\theta\,{\rm
d}\theta$$ are all continuous on $[0,\infty)$. Thus by first using
\eqref{(2.B1)} and dominated convergence and then using Dini's theorem  we
conclude that  for any $0<R<\infty$
$$\int_{0}^{\pi}\Big(\bar{B}(r,\cos\theta)-\bar{B}_n(r,\cos\theta)\Big)\sin^N\theta\,{\rm
d}\theta\to 0\,\,\, (n\to 0) \,\,\,{\rm uniformly \,\,\,in}\,\,\, r\in[0,R].$$
Therefore applying \eqref{(2.B5)}-\eqref{(2.B6)} to $L_{B-B_n}[\Dt\vp]$ we
have, for any $0<R<\infty$,
 \begin{eqnarray*}&& \sup_{|v|+|v_*|\le R}|L_{B}[\Dt\vp](v,v_*)-
L_{B_n}[\Dt\vp](v,v_*)| =\sup_{|v|+|v_*|\le R}|L_{B-B_n}[\Dt\vp](v,v_*)|\\
&&\le C_{\vp,R}\sup_{r\in[0, R]}\int_{0}^{\pi}\Big(\bar{B}(r,\cos\theta)-
\bar{B_n}(r,\cos\theta)\Big)\sin^{N}\theta\, {\rm d}\theta\to 0\quad
(n\to\infty)\end{eqnarray*} where $C_{\vp,R}=\sup\limits_{|\xi|\le R
}|H_{\vp}(\xi)| R^2$.

Finally for any $(v,v_*) \in{\mathbb{R}^N\times\mathbb{R}^N}$, using
\eqref{(2.B3)} and denoting $r=|v-v_*|$ we have by dominated convergence that
\begin{eqnarray*}&& \left|L_{B}[\Dt\vp](v,v_*)-
L_{B_n}[\Dt\vp_n](v,v_*)\right| \le \left|L_{B}[\Dt(\vp-\vp_n)](v,v_*)\right|
+ \left|L_{B-B_n}[\Dt\vp_n](v,v_*)\right|
\\
&&\le \int_{0}^{\pi}\bar{B}(r,\cos\theta) \sin^{N}\theta\,
\left|L[\Dt(\vp-\vp_n)](v,v_*,\theta)\right|{\rm d}\theta\\
&& + \,C \int_{0}^{\pi}\left(\bar{B}(r,\cos\theta)-
  \bar{B_n}(r,\cos\theta)\right)\sin^{N}\theta\,{\rm
  d}\theta\longrightarrow 0\quad (n\to\infty)\end{eqnarray*}
which concludes the proof.
\end{proof}

\subsection{A continuity estimate for product measures}
\label{sec:cont-estim-prod}

We shall now prove a continuity property for product measures which will prove
useful for the construction of weak measure solutions.

\begin{proposition}[A continuity property of product
  measures]\label{prop2.2}
  Let $0\le s_j<\infty$, $\{\mu_j^n\}_{n=1}^{\infty}\subset {\mathcal
    B}_{s_j}^{+}(\mathbb{R}^{N_j})\,,\,\mu_{j}\in {\mathcal
    B}_{0}^{+}(\mathbb{R}^{N_j})$ satisfy
\begin{equation}\label{(2.5)}
\sup\limits_{n\ge 1}\|\mu_j^n\|_{s_j}<\infty,\, \quad j=1, 2,\dots,
k\,;
\end{equation}
\begin{equation}\label{(2.6)}
\lim_{n\to\infty}\int_{\mathbb{R}^{N_j}}\vp_j{\rm d}\mu_j^n
=\int_{\mathbb R^{N_j}}\vp_j{\rm d}\mu_j\,,\quad \forall\,\vp_j\in
C_c^{\infty}(\mathbb{R}^{N_j}), \, \quad j=1, 2,\dots, k \,.
\end{equation}
Then
\begin{equation}\label{(2.7)}
  \mu_j\in {\mathcal
    B}_{s_j}^{+}(\mathbb{R}^{N_j}),\quad\|\mu_j\|_{s_j}\le
  \liminf_{n\to\infty}\|\mu_j^n\|_{s_j},\quad
  j=1, 2,\dots, k\,.
\end{equation}
Moreover if $\Psi_n, \Psi\in C(\mathbb{R}^{N_1}\times
\mathbb{R}^{N_2}\times\cdots \times \mathbb{R}^{N_k})$ satisfy
\begin{equation}\label{(2.8)}
\lim_{|{\bf x}|\to \infty}\sup_{n\ge 1}\fr{|\Psi_n({\bf
x})|}{\sum_{j=1}^{k}\langle x_j\rangle^{s_j}} =0\,,\quad
\lim_{n\to\infty}\sup_{|{\bf x}|\le R} |\Psi_n({\bf x})-\Psi({\bf
x})|=0
\end{equation}
for all $0<R<\infty$, where ${\bf x}=(x_1,x_2, \dots,x_k) \in
\prod_{j=1}^k\mathbb{R}^{N_j} $, then
\begin{equation}\label{(2.9)}
\lim_{n\to\infty}\int_{\prod_{j=1}^k\mathbb{R}^{N_j}} \Psi_n {\rm
d}(\mu_1^n\otimes\mu_2^n\otimes\cdots\otimes\mu_k^n) =
\int_{\prod_{j=1}^k\mathbb{R}^{N_j}}\Psi {\rm
d}(\mu_1\otimes\mu_2\otimes\cdots\otimes\mu_k)\,.
\end{equation}
\end{proposition}

\begin{proof}[Proof of Proposition~\ref{prop2.2}] First \eqref{(2.7)}
  easily follows from Fatou's Lemma.  Let us prove
  \eqref{(2.9)}. Let
\[
M=\sup_{n\ge 1}\left\{\|\mu_1^n\|_{s_1}\,,\,\|\mu_2^n\|_{s_2}\,,\,\dots\,\,,
\|\mu_k^n\|_{s_k}\right\},
\]
$$
\nu^n=\mu_1^n\otimes\mu_2^n\otimes\cdots\otimes\mu_k^n\,,\quad
\nu=\mu_1\otimes\mu_2\otimes\cdots\otimes\mu_k\,. $$ By assumption on
$\Psi_n,\Psi$, for any $\vep>0$ there exist $R\ge 1, n_{\vep}\ge 1$ such that
\begin{equation}\label{(2.10)}
\left|\Psi_n({\bf x})\right|\,,\,\,|\Psi({\bf
  x})|<\vep\sum_{j=1}^{k}\langle x_j\rangle^{s_j} ,\quad \forall \,|{\bf x}|>
R,\quad \forall\, n\ge n_{\vep};
\end{equation}
\begin{equation}\label{(2.11)}
\left|\Psi_n({\bf x})-\Psi({\bf x})\right|<\vep,\qquad
\forall \, |{\bf x}|\le 2kR\,,\quad \forall\, n\ge n_{\vep}.
\end{equation}
On the other hand, by polynomial approximation, there exists a
polynomial $P({\bf x})$ such that
\begin{equation}\label{(2.12)}
\left|\Psi({\bf  x})-P({\bf x})\right|<\vep\qquad \forall\,|{\bf x}|\le 2kR\,.
\end{equation}

Choose $\chi_{j}^R\in C_{c}^{\infty}(\mathbb{R}^{N_j})$ satisfying $0\le
\chi_{j}^R(x_j)\le 1$ on $\mathbb{R}^{N_j}$ and $\chi_{j}^R(x_j)=1$ for
$|x_j|\le R$ and $\chi_{j}^R(x_j)=0$ for $|x_j|\ge 2R$.  If we write $P({\bf
x})=\sum_{i=1}^{m}\prod_{j=1}^{k}P_{i,j}(x_j) $ where $m\in \mathbb{N}$ and
$P_{i,j}(x_j)$ are polynomials in $x_j$, then
$$P({\bf x})\prod_{j=1}^{k}\chi_{j}^R(x_j)=
\sum_{i=1}^{m}\prod_{j=1}^{k}\vp_{i,j}(x_j)
$$
where $\vp_{i,j}(x_j)=P_{i,j}(x_j)\chi_{j}^R(x_j)$. Then consider the
decomposition:
\begin{eqnarray*}
  && \int_{\prod_{j=1}^k \mathbb{R}^{N_j}}\Psi_n{\rm d}\nu^n- \int_{\prod_{j=1}^k\mathbb{R}^{N_j}}\Psi{\rm d}\nu=
  \int_{\prod_{j=1}^k\mathbb{R}^{N_j}}\Psi_n\left(1-{\prod}_{j=1}^{k}\chi_{j}^R\right)
  {\rm d}\nu^n \\
  &&+ \int_{\prod_{j=1}^k\mathbb{R}^{N_j}}\left(\Psi_n-\Psi\right)
  {\prod}_{j=1}^{k}\chi_{j}^R{\rm
    d}\nu^n+\int_{\prod_{j=1}^k\mathbb{R}^{N_j}}
  (\Psi-P){\prod}_{j=1}^{k}\chi_{j}^R{\rm d}\nu^n \\
  &&+
  \left[\sum_{i=1}^{m}\prod_{j=1}^{k}\int_{\prod_{j=1}^k\mathbb{R}^{N_j}}
 \vp_{i,j}{\rm d}\mu^n_j-\sum_{i=1}^{m}\prod_{j=1}^{k}
    \int_{\prod_{j=1}^k\mathbb{R}^{N_j}}\vp_{i,j}{\rm d}\mu_j \right]
  \\
  &&+\int_{\prod_{j=1}^k\mathbb{R}^{N_j}}(P-\Psi) {\prod}_{j=1}^{k}\chi_{j}^R{\rm
    d}\nu +\int_{\prod_{j=1}^k{\bf
      R}^{N_j}}\Psi\Big({\prod}_{j=1}^{k}\chi_{j}^R-1\Big) {\rm d}\nu
  \\
  &&:= I_{n,1}+ I_{n,2}+I_{n,3}+I_{n,4}+I_{5}+I_{6}.
\end{eqnarray*}
Since $1-{\prod}_{j=1}^{k}\chi_{j}^R(x_j) =0$ for all $|{\bf x}|\le
R$, and ${\prod}_{j=1}^{k}\chi_{j}^R(x_j)=0$ for all $|{\bf
  x}|>2kR$, it follows from
\eqref{(2.10)}-\eqref{(2.11)}-\eqref{(2.12)} that for all $n\ge
n_{\vep}$ \begin{eqnarray*}&& |I_{n,1}|+|I_6|\le 2\vep\int_{|{\bf x}|>R}
\sum_{j=1}^{k}\langle x_j\rangle^{s_j}{\rm d}\nu^n \le 2\vep kM^k\,,\\
&& |I_{n,2}|+|I_{n,3}|+|I_5|\le 2\vep\int_{|{\bf x}|\le 2kR}{\rm
  d}\nu^n + \vep\int_{|{\bf x}|\le 2kR}{\rm d}\nu \le 3\vep
M^k\,.\end{eqnarray*}

 For
$I_{n,4}$, since  $\vp_{i,j}\in C_{c}^{\infty}(\mathbb{R}^{N_j})$, it follows
from the assumption of the lemma that \begin{eqnarray*}&& I_{n,4}=
\sum_{i=1}^{m}\left({\prod}_{j=1}^{k}\int_{\mathbb{R}^{N_j}}\vp_{i,j} {\rm
d}\mu_{j}^n -{\prod}_{j=1}^{k}\int_{\mathbb{R}^{N_j}}\vp_{i,j} {\rm
d}\mu_j\right)\to 0 \quad (n\to\infty).\end{eqnarray*} Therefore
$$\limsup_{n\to\infty}\left|\int_{\prod_{j=1}^{k}\mathbb{R}^{N_j}}
\Psi_n{\rm d}\nu^n- \int_{\prod_{j=1}^{k}\mathbb{R}^{N_j}}\Psi{\rm
d}\nu\right| \le 5kM^k\vep\,.$$ This proves \eqref{(2.9)} by letting $\vep\to
0^+$.
\end{proof}

\subsection{Weighted Lipschitz regularity of the collision operator}
\label{sec:lipsch-regul-coll}

Let us prove some (weighted) Lipschitz properties on the collision
operator acting on Borel measures, in the (weighted) total variation
topology.

\begin{proposition}[A weighted Lipschitz property on the collision
  operator]
\label{prop1.4}  Let $B(z,\sg)$ be given by
  \eqref{(1.B)}-\eqref{(1.2)}-\eqref{(1.3)} with $b(\cdot)$
  satisfying {\bf (H4)}. Then
\[
Q^{\pm}: {\mathcal B}_{s+\gm}(\mathbb{R}^N)\times {\mathcal
B}_{s+\gm}(\mathbb{R}^N)\to {\mathcal B}_{s}(\mathbb{R}^N)\quad (s\ge 0)
\] are bounded
and
\begin{equation}\label{(5.9)}
\left\|Q^{\pm}(\mu,\nu)\right\|_s\le
2^{(s+\gm)/2}A_0\left(\|\mu\|_{s+\gm}\|\nu\|_0+\|\mu\|_0
\|\nu\|_{s+\gm}\right)\,,
\end{equation}
\begin{equation}\label{(5.10)}
\left\|Q^{\pm}(\mu,\mu)-Q^{\pm}(\nu,\nu)\right\|_s \le 2^{(s+\gm)/2}A_0
\left(\|\mu+\nu\|_{s+\gm}\|\mu-\nu\|_{0}+\|\mu+
\nu\|_{0}\|\mu-\nu\|_{s+\gm}\right)
\end{equation}
and hence
\begin{equation}\label{(5.11)}
  \left\|Q(\mu,\mu)-Q(\nu,\nu)\right\|_{0}\le
  2^{1+(s+\gm)/2}A_0 \left(\|\mu+\nu\|_{\gm}\|\mu-\nu\|_{0}+\|\mu+\nu\|_{0}\|\mu-\nu\|_{\gm}\right)\,.
\end{equation}

Finally for all $\mu\in {\mathcal B}_{\gm}(\mathbb{R}^N)$ and all $\vp\in
  C_b^2(\mathbb{R}^N)$, there holds
\begin{equation}\label{(5.12)}
\left\langle Q(\mu,\mu),\,\vp \right\rangle=\int_{\mathbb{R}^N}\vp {\rm
d}Q(\mu,\mu)
\end{equation}
where the left-hand side of \eqref{(5.12)} is defined in \eqref{(1.11)}.
\end{proposition}

\begin{proof}[Proof of Proposition \ref{prop1.4}] By elementary
  inequalities
$$
 \langle v'\rangle^s\le (\langle v\rangle^2+\langle v_*\rangle^2)^{s/2},\quad
  \left(1+|v-v_*|^2\right)^{\gm/2}\le 2^{\gm/2}\left(\langle
  v\rangle^2+\langle
v_*\rangle^2\right)^{\gm/2}
$$ and the assumption on $B$ we have for any $\vp\in C_c(\mathbb{R}^N)$ with
$\|\vp\|_{L^{\infty}_{-s}}\le 1$
$$|\vp(v')| B(v-v_*,\sg)\le \langle
v'\rangle^s\left(1+|v-v_*|^2\right)^{\gm/2}b(\cos\theta)
 \le 2^{(s+\gm)/2}(\langle v\rangle ^{s+\gm}+\langle v_*\rangle ^{s+\gm})b(\cos\theta)$$
and hence
$$
\intt_{\mathbb{R}^N \times \mathbb{R}^N} L_B\left[|\vp|\right](v,v_*){\rm
d}(|\mu|\otimes|\nu|)\le A_02^{(s+\gm)/2}
\left(\|\mu\|_{s+\gm}\|\nu\|_{0}+\|\mu\|_0\|\nu\|_{s+\gm}\right)\,,
$$
$$
\intt_{\mathbb{R}^N \times \mathbb{R}^N}A(v-v_*)|\vp(v)|{\rm d}(|\mu|\otimes
|\nu|) \le A_0 2^{(s+\gm)/2}
\left(\|\mu\|_{s+\gm}\|\nu\|_{0}+\|\mu\|_0\|\nu\|_{s+\gm}\right)\,.
$$
These imply \eqref{(5.9)}. The inequality \eqref{(5.10)} follows from
\eqref{(5.9)} and the following identities: \begin{eqnarray*}&&
Q^{\pm}(\mu,\mu)-Q^{\pm}(\nu,\nu)=\fr{1}{2}Q^{\pm}(\mu+\nu, \mu-\nu)
+\fr{1}{2}Q^{\pm}(\mu-\nu, \mu+\nu) \,.\end{eqnarray*} Next recall
$B(v-v_*,\sg)=\bar{B}(|v-v_*|,\fr{v-v_*}{|v-v_*|}\cdot\sg)$. By changing
variables $\sg\to -\sg$, $v\leftrightarrow v_*$ and using Fubini's theorem we
have $$\int_{\mathbb{R}^N}\vp{\rm
d}Q^{+}(\mu,\mu)=\fr{1}{2}\intt_{\mathbb{R}^N\times\mathbb{R}^N}
\Big(\int_{{\mathbb{S}^{N-1}}}B(v-v_*,\sg) \Big( \vp(v')+\vp(v_*')\Big)
d\sg\Big){\rm d}\mu(v){\rm d}\mu(v_*).$$ A similar symmetry for
$\int_{\mathbb{R}^N}\vp{\rm d}Q^{-}(\mu,\mu)$ is obvious. The difference of
the two is equal to $\langle Q(\mu,\mu),\vp\rangle$. This proves
\eqref{(5.12)}.
\end{proof}

\section {Moment estimates on the collision operator}
\label{sec3}

In this section we shall prove several inequalities on the moments of the
collision operator which will be useful for the moment estimates of the weak
measure solutions we shall construct.

\subsection{Analytical toolbox}
\label{sec:analytical-toolbox}

Let us first collect and prove some useful analytical results.

\begin{lemma}[Fractional binomial expansion] \label{lem3.1}
  Let $p\ge 1$ and $k_p=[(p+1)/2]$ the integer part of
  $(p+1)/2$. Then for all $x,y\ge 0$
  \begin{equation*}
\sum_{k=0}^{k_p-1}\left(\begin{array}{ll} p
      \\
      k\end{array}\right)\left( x^ky^{p-k}+x^{p-k}y^k\right)\le
  (x+y)^p\le \sum_{k=0}^{k_p}\left(\begin{array}{ll} p
      \\
      k\end{array}\right)\left( x^ky^{p-k}+x^{p-k}y^k\right)
\end{equation*}
  where
$$\left(\begin{array}{ll} p
\\
k\end{array}\right)=\fr{p(p-1)\cdots (p-k+1)}{k!}\,,\quad k\ge
1\,;\quad\left(\begin{array}{ll} p
\\
0\end{array}\right)=1\,.$$
\end{lemma}

\begin{proof}[Proof of Lemma~\ref{lem3.1}] We refer to
  \cite[Lemma~2]{MR2096050} for the proof.
\end{proof}

Let $p\ge 1$ and $n\in\{1, 2,\dots, [p]\}$. Then using Taylor's
formula for the function $x\mapsto (1+x)^p$ one has
\[\sum_{k=0}^{n}\left(\begin{array}{ll} p
    \\
    k\end{array}\right)x^k\le (1+x)^p\,\quad \forall\, x\ge 0\,.
\]
In particular
\begin{equation}\label{(3.1)}
\sum_{k=0}^{n}\left(\begin{array}{ll} p
\\
k\end{array}\right)\le 2^p \,,\quad 1\le n\le p\,.
\end{equation}

Let $\Gm(x),{\rm B}(x,y)$ be the gamma and beta functions:
$$\Gm(x)=\int_{0}^{\infty} t^{x-1} e^{-t}{\rm d}t\,,\quad x>0\,;\quad
{\rm B}(x,y)=\int_{0}^{1} t^{x-1}(1-t)^{y-1}{\rm d}t\,,\quad
x,y>0\,.$$ It is well known that
\begin{equation}\label{(3.2)}
\Gm(x)\Gm(y)=\Gm(x+y){\rm B}(x,y)\,,\quad \forall \, x, y>0\,.
\end{equation}
Other relations that we shall also use are: For any
integer $k \ge 1$ and for any real number $p\ge k$ we have
\begin{equation}\label{(3.3)}
\left(\begin{array}{ll} p
\\
k\end{array}\right)=\fr{\Gm(p+1)}{\Gm(p-k+1)\Gm(k+1)}\,.
\end{equation}
And
\begin{equation}\label{(3.4)}
{\rm B}(x+1,y)+{\rm B}(x,y+1)={\rm B}(x,y)\,,\quad x,y>0\,.
\end{equation}

\begin{lemma}[A stationary phase result] \label{lem3.2} Let $0< \alpha, R<\infty$, $g\in
  C([0,R])$ and $S\in C^1([0,R])$ such that
$$
S(0)=0, \quad S'(x)<0 \qquad \forall\,  x \in [0,R)\,.
$$
Then for any $\lambda \ge 1$ we have
$$\int_{0}^{R}x^{\alpha-1}g(x)e^{\ld S(x)} {\rm d}x=\Gm(\alpha)
\left(\fr{1}{-\ld S'(0)}\right)^{\alpha}\Big(g(0)+o(1)\Big)$$ where $o(1)\to
0$ as $\ld\to\infty$.
\end{lemma}

\begin{proof}[Proof of Lemma~\ref{lem3.2}] This is
classical stationary phase type of analysis, we
  omit the proof for the sake of conciseness of this paper.
\end{proof}

\begin{lemma}[An estimate on the beta function]\label{lem3.3}
Let $p\ge 3$ and $k_p=[(p+1)/2]$. Then
\begin{equation}\label{(3.5)}
\sum_{k=1}^{k_p}\left(\begin{array}{ll} p
    \\
    k\end{array}\right){\rm B}(k,p-k)\le 4\log p\,.
\end{equation}
More generally for any
$a>1$ we have
\begin{equation}\label{(3.6)}
\sum_{k=1}^{k_p}\left(\begin{array}{ll} p
\\
k\end{array}\right){\rm B}(ak,a(p-k))\le C_a (ap)^{1-a}\,,
\end{equation}
\begin{equation}\label{(3.7)}
\sum_{k=0}^{k_p-1}\left(\begin{array}{ll} p
-2\\
\,\,\,\,k\end{array}\right){\rm B}(a(k+1),a(p-k-1)) \le C_a (ap)^{-a}
\end{equation}
where $0<C_a<\infty$ only depends on $a$.
\end{lemma}

\begin{proof}[Proof of Lemma~\ref{lem3.3}] Since $p\ge 3$ we have
  \begin{eqnarray*}&&\sum_{k=1}^{k_p}\left(\begin{array}{ll} p
      \\
      k\end{array}\right){\rm B}(k,p-k)
  =\sum_{k=1}^{k_p}\fr{p}{k(p-k)}=\sum_{k=1}^{k_p}\Big(
  \fr{1}{k}+\fr{1}{p-k}\Big) \le 2\sum_{k=1}^{k_p} \fr{1}{k}\le 4\log
  p \,.\end{eqnarray*} Now suppose $a>1$. Let \begin{eqnarray*}&&
  \sum_{k=1}^{k_p}\left(\begin{array}{ll} p
      \\
      k\end{array}\right){\rm B}(ak,a(p-k)) =I_a(p)+I_a(p, k_p)\end{eqnarray*}
  where
$$ I_a(p)=\sum_{k=1}^{k_p-1}\left(\begin{array}{ll} p
      \\
      k\end{array}\right){\rm B}(ak,a(p-k)),\quad
I_a(p, k_p)=\left(\begin{array}{ll} p
    \\
    k_p\end{array}\right){\rm B}(ak_p,a(p-k_p)).$$
For the first term $I_a(p)$ we use the symmetry (w.r.t $x=1/2$) and Lemma
\ref{lem3.1} to get \begin{eqnarray*}&&
  I_a(p)= \fr{1}{2}
  \int_{0}^{1}\fr{1}{x(1-x)}\Big\{\sum_{k=1}^{k_p-1}\left(\begin{array}{ll}
      p
      \\
      k\end{array}\right)\Big(x^{ak}(1-x)^{a(p-k)}+
  x^{a(p-k)}(1-x)^{ak}\Big)\Big\} {\rm d}x \\
  &&\qquad \,\,\,\le \fr{1}{2} \int_{0}^{1}\fr{1}{x(1-x)}\Big\{
  \Big(x^a+(1-x)^a\Big)^{p}-x^{ap}-(1-x)^{ap}\Big\} {\rm d}x
  \\
  &&\qquad \,\,\, = \int_{0}^{1/2}\fr{1}{x(1-x)}\Big\{
  \Big(x^a+(1-x)^a\Big)^{p}-x^{ap}-(1-x)^{ap}\Big\} {\rm
    d}x\,. \end{eqnarray*} Omitting the negative term $-x^{ap}$ we
have
$$( x^a+(1-x)^a)^{p}-x^{ap}-(1-x)^{ap} \le
p( x^a+(1-x)^a)^{p-1} x^{a}$$ so that
$$ I_a(p)\le p\int_{0}^{1/2}x^{a-1} g_1(x) e^{pS(x)}{\rm d}x $$ where
$g_1(x)=(1-x)^{-1}( x^a+(1-x)^a)^{-1}$ and $S(x)=\log(x^a+(1-x)^a)$,
$x\in[0,1/2]$. Since $g_1(0)=1, S(0)=0$ and
$$ S'(0)=-a\,,\quad S'(x)=\fr{a\Big(x^{a-1}-(1-x)^{a-1}\Big)}{ x^a+(1-x)^a}<0
\qquad \forall\, x\in[0,1/2) $$  (because $a>1$) it follows from Lemma
\ref{lem3.2} that for all $p\ge 3$ \begin{eqnarray*}&& I_a(p)\le
C_ap\Gm(a)\left(\fr{1}{pa}\right)^{a}=
 C_a{(ap)}^{1-a}\,.\end{eqnarray*}
For the second term $I_a(p, k_p)$ we use Stirling's formula
$$\Gm(x)=\left(\fr{x}{e}\right)^{x}
\sqrt{\fr{2\pi}{x}}\, e^{\fr{\theta_x}{12 x}} \,,\quad
\Gm(x+1)=x\Gm(x)=\left(\fr{x}{e}\right)^{x} \sqrt{2\pi x}\,
e^{\fr{\theta_x}{12 x}}\,,\quad x\ge 1$$ ($0<\theta_x<1$) to compute
\begin{eqnarray}\label{(3.8)} && \quad I_a(p, k_p)=
\fr{\Gm(p+1)}{\Gm(k_p+1)\Gm(p-k_p+1)}\cdot\fr{\Gm(ak_p
)\Gm(a(p-k_p))}{\Gm(ap)}
\\ \nonumber &&
\le e^{1/4}\fr{\sqrt{a}}{ap}\left(\fr{k_p}{p}\right)^{(a-1)k_p}
\left(\fr{p-k_p}{p}\right)^{(a-1)(p-k_p)}\,
\left(\fr{p}{k_p}\right)\left(\fr{p}{p-k_p}\right) \le C_a
\fr{1}{ap}\left(\fr{1}{2}\right)^{(a-1)p} \,.
\end{eqnarray}
Here in the last inequality we used the simple estimates
\[
\fr{p-1}{2}\le p-k_p \le \fr{p+1}{2}
\]
for $p\ge 3$.  This proves \eqref{(3.6)} because $a>1$.

In order to prove \eqref{(3.7)} we consider again a decomposition
\begin{eqnarray*}
&& \sum_{k=0}^{k_p-1}\left(\begin{array}{ll} p
    -2\\
    \,\,\,\,k\end{array}\right){\rm B}(a(k+1),a(p-k-1)) =J_a(p)+J_a(p,
k_p)
\end{eqnarray*}
where for the first term $J_a(p)$ we use that
$k_p-2=[(p-1)/2]-1=k_{p-2}-1$ and Lemma \ref{lem3.1} to get
\begin{eqnarray*}
&& J_a(p):=\sum_{k=0}^{k_{p}-2}\left(\begin{array}{ll} p
    -2\\
    \,\,\,\,k\end{array}\right){\rm B}(a(k+1),a(p-k-1)) \\
&&= \fr{1}{2}\int_{0}^{1}x^{a-1}(1-x)^{a-1}
\sum_{k=0}^{k_{p-2}-1}\left(\begin{array}{ll} p
    -2 \\
    \,\,\,\,k\end{array}\right)\Big(
x^{ak}(1-x)^{a(p-2-k)}+x^{a(p-2-k)}(1-x)^{ak} \Big){\rm d}x\\
&& \le\fr{1}{2}\int_{0}^{1}x^{a-1}(1-x)^{a-1}\Big(x^a +(1-x)^a\Big)^{p-2}{\rm
d}x =\int_{0}^{1/2} x^{a-1} g_2(x) e^{p
  S(x)}{\rm d}x
\end{eqnarray*}
with $g_2(x)=(1-x)^{a-1}( x^a+(1-x)^a)^{-2}.$ Since $a>1$, it follows
from Lemma \ref{lem3.2} that
$$J_a(p)\le C_a\left(\fr{1}{
ap}\right)^{a}\,.$$ For the second term $J_a(p, k_p)$ we use \eqref{(3.8)} to
get
\begin{eqnarray*}
&& J_a(p,k_p):= \left(\begin{array}{ll} p
-2\\
k_p-1\end{array}\right){\rm B}(ak_p,a(p-k_p)) = \fr{(p-k_p)k_p}{p(p-1)}I_a(p,
k_p) \le C_a \fr{1}{ap} \left(\fr{1}{2}\right)^{(a-1)p}\,.
\end{eqnarray*}
Since $a>1$,
this proves the lemma.
\end{proof}

\subsection{An estimate of the angular cutoff
  reminder}
\label{sec:angular cutoff reminder}

\begin{lemma}\label{lem3.4}
  Suppose $b(\cdot)$ satisfies the assumption {\bf (H0)}.  For all
  $p\ge 3$ we define
\begin{equation}\label{(3.9)}
\vep_p:=\fr{2}{A_2} \left|\mathbb{S}^{N-2}\right|\int_{0}^{\pi}
\left\{
 \int_{0}^{1} t\left(1-\fr{\sin^2\theta}{2}
t\right)^{p-2} {\rm d}t\right\} b(\cos\theta)\sin^{N}\theta\, {\rm d}\theta\,\quad (\le 1)\,.
\end{equation}
Then $\,\vep_p\to 0\,\, (p\to\infty).\, $ Furthermore, if either $0<\gm\le 1$
or {\bf(H2)} is satisfied, then
\begin{equation}\label{(3.10)}
p^{2-2/\gamma}\vep_p \to 0\qquad (p\to\infty)\,.
\end{equation}
\end{lemma}

\begin{proof}[Proof of Lemma~\ref{lem3.4}]
  Under the assumption {\bf (H0)}, the convergence $\vep_p\to 0\,
  (p\to\infty)$ is obvious and hence \eqref{(3.10)} holds for $0< \gm\le 1$. Suppose {\bf (H2)} is satisfied, which means that $\nu=2-2/\gm\in (0,1)$ and $\theta\mapsto b(\cos\theta)\sin^{N-2\nu}\theta $ is integrable on $[0,\pi]$. For all $p\ge 3$ we have
  \begin{equation}\label{(3.11)}
   p^{\nu}\vep_p
  \le C\int_{0}^{\pi}\left\{
  \int_{0}^{1}\left((p-2)\fr{\sin^2\theta}{2} t
  \right)^{\nu}\left(1-\fr{\sin^2\theta}{2}t\right)^{p-2} {\rm d}t \right\}
  b(\cos\theta)\sin^{N-2\nu}\theta\,{\rm
    d}\theta
\end{equation}
where $C$ depends only on $N,A_2$ and $\nu$.  Applying elementary estimates
$$0\le (\ld x)^{\nu}(1-x)^{\ld}<1\,,\quad (\ld x)^{\nu}
(1-x)^{\ld}\to 0\quad (\ld\to\infty)\quad \forall\, x\in[0,1]$$
to $\ld=p-2$ and $
x= \fr{\sin^2\theta}{2} t$  we conclude from \eqref{(3.11)} and the dominated
convergence theorem that $p^{\nu}\vep_p\to 0\,\,( p\to\infty).$
\end{proof}

\begin{remark} It is easily calculated that if the
assumption {\bf (H4)} is satisfied, i.e. if $A_0<\infty$, then $\vep_p\le
\frac{16A_0}{A_2}\fr{1}{p}$ for all $p\ge 3$, so that in case
  $0<\gm <2$ we have $p^{2-2/\gm}\vep_p\le \fr{16A_0}{A_2} p^{1-2/\gm}\,.$
\end{remark}

\subsection{Moment estimates on the kernel $L_B$}
\label{sec:moment-estim-kern}

In this subsection we shall prove moment estimates on the kernel $L_B$
as defined in \eqref{(1.6)}.

\begin{lemma}\label{lem3.5} Let $B(z,\sg)=|z|^{\gm}b(\cos\theta)$.

\begin{itemize}
\item[(I)] Under the assumption {\bf (H0)} we have
for all $p\ge 3$
\begin{eqnarray}\label{(3.12)} &&\qquad
L_B \left[\Dt\langle \cdot\rangle^{2p}\right](v,v_*)\\ \nonumber && \quad \le
-\fr{A_2}{4}\Big(\langle v\rangle^{2p+\gm} +\langle v_*\rangle^{2p+\gm}\Big)
+\fr{A_2}{2}\Big(\langle v\rangle^{2p}\langle v_*\rangle^{\gm}+\langle
v_*\rangle^{2p}\langle v\rangle^{\gm} \Big)
\\ \nonumber
&&\quad +A_2\sum_{k=1}^{k_p}\left(\begin{array}{ll} p
\\
k\end{array}\right)\Big(\langle v\rangle^{2k+\gm}\langle v_*\rangle^{2(p-k)}+
\langle v\rangle^{2(p-k)+\gm}\langle v_*\rangle^{2k} \Big)\\ \nonumber &&\quad
+A_2\sum_{k=1}^{k_p}\left(\begin{array}{ll} p
\\ \nonumber
k\end{array}\right)\Big(\langle v\rangle^{2k}\langle v_*\rangle^{2(p-k)+\gm}+
\langle
v\rangle^{2(p-k)}\langle v_*\rangle^{2k+\gm} \Big)\\
&&\quad +2 p(p-1)A_2 \vep_p \sum_{k=0}^{k_p-1}\left(\begin{array}{ll} p-2
\\ \nonumber
\,\,\,\,k\end{array}\right)\Big(\langle v\rangle^{2(k+1)+\gm}\langle
v_*\rangle^{2(p-1-k)}+ \langle v\rangle^{2(p-1-k)+\gm}\langle v_*\rangle^{2(k+1)} \Big)\\
\nonumber && \quad+2 p(p-1)A_2 \vep_p
\sum_{k=0}^{k_p-1}\left(\begin{array}{ll} p-2
\\
\,\,\,\,k\end{array}\right)\Big(\langle v\rangle^{2(k+1)}\langle
v_*\rangle^{2(p-1-k)+\gm}+
\langle v\rangle^{2(p-1-k)}\langle v_*\rangle^{2(k+1)+\gm} \Big)\,.
\end{eqnarray}

\item[(II)]
Under the assumption {\bf (H3)} which is rewritten in the form
\begin{equation}\label{H3bis}
\gm =2,\quad 1<p_1<\infty\,,\quad A_{p_1}^*:=\left|\mathbb{S}^{N-2}\right|
\left(\int_{0}^{\pi}[b(\cos\theta)]^{p_1}\sin^{N-2}\theta\,{\rm d}\theta
\right)^{1/{p_1}}<\infty
\end{equation}
and let
\begin{equation}\label{H3bis-1}
q_1=\fr{p_1}{p_1-1},\quad \eta=\fr{1}{2q_1}.
\end{equation}
Then
\begin{eqnarray}\label{(3.13)} &&
L_B\left[\Dt\langle \cdot\rangle^{2p}\right](v,v_*) \\
 \nonumber
 &&\le
 \fr{12A^*_{p_1}}{p^{\eta}}
\sum_{k=1}^{k_p}\left(\begin{array}{ll} p
\\
k\end{array}\right)\Big(\langle v\rangle^{2(k+1)}\langle v_*\rangle^{2(p-k)}+
\langle v\rangle^{2(p-k+1)}\langle v_*\rangle^{2k} \Big)\\ \nonumber
&&+\fr{12A^*_{p_1}}{p^{\eta}} \sum_{k=1}^{k_p}\left(\begin{array}{ll} p
\\
k\end{array}\right)\Big(\langle v\rangle^{2k}\langle v_*\rangle^{2(p-k+1)}+
\langle v\rangle^{2(p-k)}\langle v_*\rangle^{2(k+1)} \Big)\\ \nonumber
&&+\fr{A_0}{2}\langle v\rangle^{2p}\langle v_*\rangle^{2} + \fr{A_0}{2}\langle
v_*\rangle^{2p}\langle v\rangle^{2} -\fr{A_0}{4}\langle
v\rangle^{2(p+1)}-\fr{A_0}{4}\langle v_*\rangle^{2(p+1)}
\end{eqnarray}
for all $p\ge \left(12A^*_{p_1}/A_0\right)^{2q_1}$.
\end{itemize}
\end{lemma}

\begin{proof}[Proof of Lemma~\ref{lem3.5}]
\mbox{ } \medskip

\noindent
 {\bf Part (I)} Let us write
$$
L_B\left[\Dt\langle \cdot\rangle^{2p}\right](v,v_*)=|v-v_*|^{\gm}
\left|\mathbb{S}^{N-2}\right| \int_{0}^{\pi}b(\cos\theta)\sin^{N}\theta\,
L_p(v,v_*,\theta)\,{\rm d}\theta
$$
with $$L_p(v,v_*,\theta):=\fr{1}{\sin^2\theta\,|\mathbb{S}^{N-2}|}
\int_{\mathbb{S}^{N-2}({\bf k})}\left(\langle v'\rangle^{2p}+ \langle
v_*'\rangle^{2p}-\langle v\rangle^{2p}-\langle v_*\rangle^{2p} \right){\rm
d}\og .$$

We first prove that
\begin{eqnarray}\label{(3.14)}
&& L_p(v,v_*,\theta)\le -\fr{1}{2}\Big(\langle v\rangle^{2p}+\langle
 v_*\rangle^{2p}\Big)\\ \nonumber
 &&+
 \fr{1}{2}\sum_{k=1}^{k_p}\left(\begin{array}{ll} p
\\ \nonumber
k\end{array}\right)\Big(\langle v\rangle^{2k} \langle
 v_*\rangle^{2p-2k} +\langle v\rangle^{2p-2k} \langle
 v_*\rangle^{2k}\Big)\\ \nonumber
 &&+2p(p-1)\int_{0}^{1} t\Big(1-\fr{\sin^2\theta}{2}
t\Big)^{p-2} dt\\ \nonumber
&&\times \sum_{k=0}^{k_p-1}\left(\begin{array}{ll}
p-2
\\
\,\,\,\,k\end{array}\right)\Big(\langle v\rangle^{2(k+1)}\langle
v_*\rangle^{2(p-1-k)}+ \langle v\rangle^{2(p-1-k)}\langle v_*\rangle^{2(k+1)}
\Big)\,.
\end{eqnarray}
To do this we denote the shorthand
\begin{eqnarray*}
&& E(\theta)=\langle v\rangle^2\cos^2\theta/2+\langle
v_*\rangle^2\sin^2\theta/2\,,\quad h=\sqrt{|v|^2|v_*|^2-\langle
v,v_*\rangle^2}\,.
\end{eqnarray*}
Then by \eqref{(1.25)}
\begin{eqnarray*}
&&\langle v'\rangle^2= E(\theta) +h\,\sin\theta\, ({\bf
j}\cdot\og)\,,\qquad\langle v_*'\rangle^2=E(\pi-\theta) -h\,\sin\theta\,({\bf
j}\cdot\og)\,.
\end{eqnarray*}
By Taylor's formula we have
\begin{eqnarray*}&&\Big(E(\theta)\pm h\,\sin\theta\, ({\bf j}\cdot\og)\Big)^p
= \Big(E(\theta)\Big)^p \pm  q\Big(E(\theta)\Big)^{p-1}
h\,\sin\theta\, ({\bf j}\cdot\og)\\
&& + p(p-1) \int_{0}^1(1-t)\Big(E(\theta)\pm  t h\,\sin\theta\, \langle {\bf
j},\og\rangle\Big)^{p-2} {\rm d}t (h\,\sin\theta\, \langle {\bf
j},\og\rangle)^2\,.\end{eqnarray*} Look at the last term: We have for all
$\theta\in(0,\pi), t\in[0,1]$
\begin{eqnarray*}&& E(\theta)+t h\,\sin\theta\,
|({\bf j}\cdot\og)| \le E(\theta) +
\Big(E(\pi-\theta)\Big)\,t\\
&&=\langle v\rangle^2+\langle v_*\rangle^2 -
 \Big(E(\pi-\theta)\Big)
(1-t)\\
&&\le \Big(\langle v\rangle^2+\langle
v_*\rangle^2\Big)\Big(1-\fr{1-t}{2}\sin^2\theta \Big)\end{eqnarray*} where we
used
$$E(\pi-\theta)\ge(\langle v\rangle^2+\langle v_*\rangle^2)\min\{\cos^2\theta/2\,,\,\sin^2\theta/2\}
\ge(\langle v\rangle^2+\langle v_*\rangle^2)\fr{\sin^2\theta}{2}\,.$$ Since
$$
\int_{\mathbb{S}^{N-2}({\bf n})}({\bf j}\cdot\og) {\rm d}\og=0
$$
it follows that
\begin{eqnarray}\label{(3.15)}
&& L_p(v,v_*,\theta)\le
\fr{1}{\sin^2\theta}\left((E(\theta))^p+(E(\pi-\theta))^p-\langle
v\rangle^{2p}- \langle v_*\rangle^{2p}\right)
\\ \nonumber
&&+2p(p-1)\left(\langle v\rangle^2+\langle v_*\rangle^2\right)^{p-2}h^2
\int_{0}^{1} t\left(1-\fr{\sin^2\theta}{2} t\right)^{p-2}{\rm d}t\,.
\end{eqnarray}
We need to prove that for $p\ge 3$ and $k_p=[(p+1)/2]$
\begin{eqnarray}\label{(3.16)} &&
\fr{1}{\sin^2\theta}\Big((E(\theta))^{p}+ (E(\pi-\theta))^{p}-\langle
v\rangle^{2p}-\langle
 v_*\rangle^{2p}\Big)\\ \nonumber
 &&
\le -\fr{1}{2}\Big(\langle v\rangle^{2p}+\langle
 v_*\rangle^{2p}\Big)+\fr{1}{2}
\sum_{k=1}^{k_p}\left(\begin{array}{ll} p
\\
k\end{array}\right)\Big(\langle v\rangle^{2k} \langle
 v_*\rangle^{2p-2k} +\langle v\rangle^{2p-2k} \langle
 v_*\rangle^{2k}\Big)\,.
\end{eqnarray}
In fact using Lemma \ref{lem3.1} we have
\begin{eqnarray*}
  && (E(\theta))^p+(E(\pi-\theta))^p\\
  &&\le \sum_{k=0}^{k_p}\left(\begin{array}{ll} p
      \\
      k\end{array}\right)\Bigg(\left[\langle v\rangle^2
    \cos^2(\theta/2)\right]^{k} \left[ \langle
    v_*\rangle^2\sin^2(\theta/2)\right]^{p-k} \\
  && \qquad \qquad \qquad \qquad +\left[\langle v\rangle^2 \cos^2(\theta/2)\right]^{p-k}
  \left[ \langle
    v_*\rangle^2\sin^2(\theta/2)\right]^{k}\Bigg)
  \\
  &&+ \sum_{k=0}^{k_p}\left(\begin{array}{ll} p
      \\
      k\end{array}\right)\Big(\left[\langle v\rangle^2
      \sin^2(\theta/2)\right]^{k} \left[ \langle
      v_*\rangle^2\cos^2(\theta/2)\right]^{p-k}\\
&& \qquad \qquad \qquad \qquad +\left[\langle v\rangle^2
\sin^2(\theta/2)\right]^{p-k}
    \left[ \langle
      v_*\rangle^2\cos^2(\theta/2)\right]^{k}\Bigg)
  \\
  &&\le \fr{\sin^2\theta}{2}\sum_{k=1}^{k_p}\left(\begin{array}{ll} p
      \\
      k\end{array}\right)\Big(\langle v\rangle^{2k} \langle
  v_*\rangle^{2p-2k}+
  \langle v\rangle^{2p-2k} \langle
  v_*\rangle^{2k}
  \Big)\\
  &&+\Big(\langle v\rangle^{2p}+\langle v_*\rangle^{2p}\Big)
  \Big(\cos^{2p}(\theta/2)+\sin^{2p}(\theta/2)\Big)
\end{eqnarray*}
where we used the fact that $p\ge 3 \Longrightarrow
p-k_p\ge 1$  so that
\begin{eqnarray*}
&&\cos^{2k}(\theta/2)\sin^{2p-2k}(\theta/2),\,
  \sin^{2k}(\theta/2)\cos^{2p-2k}(\theta/2) \le \fr{1}{4}\sin^2\theta
\end{eqnarray*}
for all  $k\in\{1,2,...,k_p\}$. Since $p\ge 3$ implies
\begin{eqnarray*}
&&\cos^{2p}(\theta/2)+\sin^{2p}(\theta/2)\le
  \cos^4(\theta/2)+\sin^4(\theta/2) =1-\fr{1}{2}\sin^2(\theta)
\end{eqnarray*}
this gives \eqref{(3.16)}.

Note that $h^2\le\langle v\rangle^2\langle v_*\rangle^2$. Then using Lemma
\ref{lem3.1} again and recalling $k_p-1=k_{p-2}=[(p-1)/2]$ we have
 \begin{eqnarray*}&&
 \Big(\langle v\rangle^2+\langle
v_*\rangle^2\Big)^{p-2}h^2\le \sum_{k=0}^{k_p-1}\left(\begin{array}{ll} p-2
\\
\,\,\,\,k\end{array}\right)\Big(\langle v\rangle^{2(k+1)}\langle
v_*\rangle^{2(p-1-k)}+ \langle v\rangle^{2(p-1-k)}\langle v_*\rangle^{2(k+1)}
\Big).
\end{eqnarray*}
This together with \eqref{(3.15)}-\eqref{(3.16)} concludes the proof of
\eqref{(3.14)}.

Now using \eqref{(3.14)} and the definitions of $L_B[\Dt\vp]$, $A_2$ and
$\vep_p$ we obtain
\begin{eqnarray} \label{(3.18)}
&& \qquad L_B\left[\Dt\langle \cdot\rangle^{2p}\right](v,v_*) \le
-\fr{A_2}{2}\Big(\langle v\rangle^{2p}+\langle
 v_*\rangle^{2p}\Big)|v-v_*|^{\gm}\\ \nonumber
 &&+\fr{A_2}{2}
 \sum_{k=1}^{k_p}\left(\begin{array}{ll} p
\\
k\end{array}\right)\Big(\langle v\rangle^{2k} \langle
 v_*\rangle^{2p-2k} +\langle v\rangle^{2p-2k} \langle
 v_*\rangle^{2k}\Big)|v-v_*|^{\gm} \\ \nonumber
 &&+p(p-1)A_2\vep_p
\sum_{k=0}^{k_p-1}\left(\begin{array}{ll} p-2
\\
\,\,\,\,k\end{array}\right) \Big(\langle v\rangle^{2(k+1)}\langle
v_*\rangle^{2(p-1-k)}+ \langle
v\rangle^{2(p-1-k)}\langle v_*\rangle^{2(k+1)} \Big)|v-v_*|^{\gm}\,.
\end{eqnarray}
Next by $0<\gm\le 2$ we have
\begin{equation}\label{(3.19)}
|v-v_*|^{\gm}\ge\fr{1}{2}\langle v\rangle^{\gm}- \langle v_*\rangle^{\gm}
\,,\quad |v-v_*|^{\gm}\ge \fr{1}{2}\langle v_*\rangle^{\gm}- \langle
v\rangle^{\gm} \,.
\end{equation}
Thus
\begin{eqnarray*}
&&\left(\langle v\rangle^{2p}+
 \langle v_*\rangle^{2p}\right)|v-v_*|^{\gm}=\langle v\rangle^{2p}|v-v_*|^{\gm}+
 \langle v_*\rangle^{2p}|v-v_*|^{\gm}\\
&&\ge\langle v\rangle^{2p}\Big(\fr{1}{2}\langle v\rangle^{\gm}- \langle
v_*\rangle^{\gm} \Big) +
 \langle v_*\rangle^{2p}\Big(\fr{1}{2}\langle v_*\rangle^{\gm}- \langle v\rangle^{\gm} \Big)
\\
&&=\fr{1}{2}\langle v\rangle^{2p+\gm} +\fr{1}{2}\langle v_*\rangle^{2p+\gm} -
\langle v\rangle^{2p}\langle v_*\rangle^{\gm}-\langle v_*\rangle^{2p}\langle
v\rangle^{\gm}\,.
\end{eqnarray*}
Since
\begin{equation}\label{(3.20)}
|v-v_*|^{\gm}
\le 2(\langle v\rangle^{\gm}+\langle v_*\rangle^{\gm})
\end{equation}
it follows that
\begin{eqnarray*}
&&\left(\langle v\rangle^{2k}\langle v_*\rangle^{2(p-k)}+ \langle
v\rangle^{2(p-k)}\langle
v_*\rangle^{2k} \right)|v-v_*|^{\gm}\\
&&\le 2\left(\langle v\rangle^{2k}\langle v_*\rangle^{2(p-k)}+ \langle
v\rangle^{2(p-k)}\langle
v_*\rangle^{2k} \right)(\langle v\rangle^{\gm}+\langle v_*\rangle^{\gm})\\
&&=2\left(\langle v\rangle^{2k+\gm}\langle v_*\rangle^{2(p-k)}+ \langle
v\rangle^{2(p-k)+\gm}\langle
v_*\rangle^{2k} \right)\\
&&+2\left(\langle v\rangle^{2k}\langle v_*\rangle^{2(p-k)+\gm}+ \langle
v\rangle^{2(p-k)}\langle v_*\rangle^{2k+\gm} \right)\,.
\end{eqnarray*}
And similarly
\begin{eqnarray*}
&&\Big(\langle v\rangle^{2(k+1)}\langle v_*\rangle^{2(p-1-k)}+ \langle
v\rangle^{2(p-1-k)}\langle
v_*\rangle^{2(k+1)} \Big)|v-v_*|^{\gm}\\
&&\le 2\Big(\langle v\rangle^{2(k+1)}\langle v_*\rangle^{2(p-1-k)}+ \langle
v\rangle^{2(p-1-k)}\langle
v_*\rangle^{2(k+1)} \Big)(\langle v\rangle^{\gm}+\langle v_*\rangle^{\gm})\\
&&=2\Big(\langle v\rangle^{2(k+1)+\gm}\langle v_*\rangle^{2(p-1-k)}+ \langle
v\rangle^{2(p-1-k)+\gm}\langle
v_*\rangle^{2(k+1)} \Big)\\
&&+2\Big(\langle v\rangle^{2(k+1)}\langle v_*\rangle^{2(p-1-k)+\gm}+ \langle
v\rangle^{2(p-1-k)+\gm}\langle v_*\rangle^{2(k+1)+\gm} \Big) \,.
\end{eqnarray*}
These together with \eqref{(3.18)} yield the estimate \eqref{(3.12)}.  \medskip

\noindent {\bf Part (II)} For any $p\ge 1$ we have
\begin{eqnarray*}
&& |v-v_*|^{-2}L_B[\Dt\langle \cdot\rangle^{2p}] (v,v_*)
=2\int_{\mathbb{S}^{N-1}}b(\cos\theta)\langle v'\rangle^{2p} {\rm
d}\sg-A_0\Big(
\langle v\rangle^{2p}+\langle v_*\rangle^{2p}\Big)\\
&&\le 2A^*_{p_1}\left(\fr{1}{|\mathbb{S}^{N-2}|}
\int_{\mathbb{S}^{N-1}}\langle
  v'\rangle^{2pq_1} {\rm d}\sg\right)^{1/{q_1}} -A_0\Big( \langle
v\rangle^{2p}+\langle v_*\rangle^{2p}\Big)
\end{eqnarray*}
where we used H\"{o}lder's inequality. We have to prove that
\begin{equation}\label{(3.21)}
\left(\fr{1}{|\mathbb{S}^{N-2}|}\int_{\mathbb{S}^{N-1}}\langle
v'\rangle^{2pq_1} {\rm d}\sg\right)^{1/{q_1}} \le \fr{3}{p^{\eta}}\Big(\langle
v\rangle^2+\langle v_*\rangle^2\Big)^{p}\,.
\end{equation}
To do this we denote $\ld=pq_1\,( >1)$. Then using elementary inequalities
$$\left(\fr{1+x}{2}\right)^{\ld}+\left(\fr{1-x}{2}\right)^{\ld} 
\le \left(\fr{1+y}{2}\right)^{\ld}+\left(\fr{1-y}{2}\right)^{\ld}
,\quad x,y\in[-1,1],\quad |x|\le |y|\,;$$
$$
|v+v_*||v-v_*|\le \langle v\rangle^2+\langle v_*\rangle^2\,,$$
and the formula \eqref{(1.24)} 
 we compute (recall that $N\ge 2$)
\begin{eqnarray*}&& \fr{1}{|\mathbb{S}^{N-2}|}\int_{\mathbb{S}^{N-1}}\langle
  v'\rangle^{2\ld} {\rm d}\sg
  \\
  && =\left(\langle v\rangle^2+\langle
    v_*\rangle^2\right)^{\ld}\int_{0}^{\pi}\sin^{N-2}\theta\left(\fr{1}{2}+\fr{|v+v_*||v-v_*|}
    {2(\langle v\rangle^2+\langle v_*\rangle^2)}\cos\theta \right)^{\ld}
  {\rm d}\theta\\
  &&\le \left(\langle v\rangle^2+\langle v_*\rangle^2\right)^{\ld}
  \int_{0}^{\pi}\left(\fr{1-\cos\theta}{2} \right)^{\ld}{\rm d}\theta
  \le\Big(\langle v\rangle^2+\langle v_*\rangle^2\Big)^{\ld}
  \sqrt{\fr{2\pi}{\ld}}\end{eqnarray*} where we used the well-known
inequality
\[
\int_{0}^{\pi/2}\sin^n\theta\,{\rm d}\theta<\sqrt{\fr{\pi}{2n}}
\]
with $n=2[\ld]$. This yields \eqref{(3.21)}.

From this and using Lemma \ref{lem3.1} we obtain that for all $p\ge  3$
\begin{eqnarray*}&& |v-v_*|^{-2} L_B\left[\Dt\langle
    \cdot\rangle^{2p}\right](v,v_*)\le \fr{6A^*_{p_1}}{p^{\eta}} \left(
    \langle v\rangle^2+\langle v_*\rangle^2\right)^{p}-A_0\left( \langle v\rangle^{2p}
    +\langle
    v_*\rangle^{2p}\right)\\
  &&\le\fr{6A^*_{p_1}}{p^{\eta}}\sum_{k=1}^{k_p}
  \left(\begin{array}{ll} p
      \\
      k\end{array}\right)\Big(\langle v\rangle^{2k} \langle v_*\rangle^{2(p-k)}+ \langle
  v\rangle^{2(p-k)} \langle v_*\rangle^{2k}
  \Big) -\Big(A_0-\fr{6A^*_{p_1}}{p^{\eta}}\Big) \Big( \langle v\rangle^{2p}
  +\langle
  v_*\rangle^{2p}\Big) \,.
\end{eqnarray*}
Since
$p\ge
\left(12A^*_{p_1}/A_0\right)^{2q_1}\Longleftrightarrow
6A^*_{p_1}/p^{\eta}\le A_0/2$,
it follows that
\begin{eqnarray*}
&& L_B\left[\Dt\langle \cdot\rangle^{2p}\right](v,v_*)
\\
&&
\le
\fr{6A^*_{p_1}}{p^{\eta}}\sum_{k=1}^{k_p}
\left(\begin{array}{ll} p
\\
k\end{array}\right)\left(\langle v\rangle^{2k} \langle
 v_*\rangle^{2(p-k)}+
 \langle v\rangle^{2(p-k)} \langle
 v_*\rangle^{2k}
\right)|v-v_*|^{2}\\
&& -\fr{A_0}{2}\left( \langle v\rangle^{2p}+\langle
  v_*\rangle^{2p}\right)|v-v_*|^{2}\qquad \forall\,p\ge \left(12A^*_{p_1}/A_0\right)^{2q_1}.
\end{eqnarray*}
Therefore as shown in the above using \eqref{(3.19)}-\eqref{(3.20)}
with $\gm=2$ we obtain \eqref{(3.13)}.
\end{proof}

\subsection{Moment estimates on the collision operator}
\label{sec:moment-estim-coll}

We shall now deduce from the moment estimates on $L_B$ in the previous
Lemma~\ref{lem3.5} some moment estimates on the collision operator.

\begin{lemma}\label{lem3.6}
Let $B(z,\sg)=|z|^{\gm}b(\cos\theta)$, $\mu\in {\mathcal
    B}_s^+(\mathbb{R}^N) $ with $\|\mu\|_0\neq 0$, $s\ge \gm +2p$, $0<\gm\le 2$,
and $p\ge 3$.

\begin{itemize}
\item[(I)] If $b(\cos\theta)$ satisfies
  the assumption {\bf (H0)}, then
\begin{eqnarray}\label{(3.22)} &&
\left\langle Q(\mu,\mu),\,\left\langle\cdot\right\rangle^{2p}\right\rangle\le
2^{2p+1} A_2\|\mu\|_2\|\mu\|_{2p}-\fr{1}{4}A_2\|\mu\|_0\|\mu\|_{2p+\gm}\,.
\end{eqnarray}
Furthermore if $0<\gm<2$,  then
\begin{eqnarray} &&
 \fr{\left\langle
      Q(\mu,\mu),\,\langle\cdot\rangle^{2p}\right\rangle
  }{\Gm(q)\|\mu\|_0} \\ \nonumber &&
  \le \Big(C_a q^{2-a} + C_a
  q^{3-a}\vep_p \Big)A_2\|\mu\|_0 Z_p^*+\fr{1}{2}\|\mu\|_2A_2Z_{q}-
  \fr{q}{16}A_2\|\mu\|_0Z_q^{1+\fr{1}{q}}
\label{(3.23)} \end{eqnarray} where $q=ap\,,\, a= 2/\gm$,
\begin{equation}
Z_q=\fr{\|\mu\|_{\gm q}}{\Gm(q)\|\mu\|_0},\qquad Z_p^*=\max_{k\in\{1,2, \dots,
k_p\}} \{Z_{ak+1}Z_{a(p-k)}\,,\,Z_{ak}Z_{a(p-k)+1}\} \label{(3.25)}
\end{equation} and the constant $0<C_a<\infty$ only depends on $a$.
\medskip

\item[(II)] If $\gm=2$ and $b(\cos\theta)$
  satisfies {\bf (H3)} which is rewritten as in \eqref{H3bis}, and let
 $p_1,q_1,\eta$ be given in \eqref{H3bis}-\eqref{H3bis-1}, then
 \begin{multline}
  \fr{\left\langle
      Q(\mu,\mu),\,\langle\cdot\rangle^{2p}\right\rangle}{\Gm(p)\|\mu\|_0}
  \\
  \le 48A^*_{p_1}p^{1-\eta}(\log p)\|\mu\|_0 \wt{Z}_p^*+
  \Big(12A^*_{p_1}p^{1-\eta}+\fr{A_0}{4}\Big)\|\mu\|_{2}Z_p
  -\fr{p}{16}A_0\|\mu\|_0Z_{p}^{1+\fr{1}{p}}
\label{(3.26)}
\end{multline}
 for all $p\ge \left(12A^*_{p_1}/A_0\right)^{2q_1}$, where
\begin{equation} Z_p=\fr{\|\mu\|_{2p}}{\Gm(p)\|\mu\|_0},\qquad
\wt{Z}_p^*=\max_{k\in\{1, \, 2, \dots,\,
  k_p\}}Z_{k+1}Z_{p-k}\,.
\label{(3.27)}
\end{equation}
\end{itemize}
\end{lemma}

\begin{proof}[Proof of Lemma~\ref{lem3.6}] By replacing $\mu$ with
$\mu/\|\mu\|_0$ we can assume that $\|\mu\|_0=1$.

\noindent {\bf Part (I).}  By part (I) of Lemma \ref{lem3.5} we have
\begin{eqnarray}\label{(3.Q)}
  &&\quad \left\langle Q(\mu,\mu),\,\langle\cdot\rangle^{2p}\right\rangle =
  \fr{1}{2} \intt_{\mathbb{R}^N \times \mathbb{R}^N}L_B\left[\Dt\langle
  \cdot\rangle^{2p}\right] (v,v_*)
  {\rm d}\mu(v){\rm d}\mu(v_*) \\
&&\quad \le A_2\sum_{k=1}^{k_p}\left(\begin{array}{ll} p
       \nonumber \\
      k\end{array}\right)\left(\|\mu\|_{2k+\gm}\|\mu\|_{2(p-k)}+
  \|\mu\|_{2k}\|\mu\|_{2(p-k)+\gm}\right)  \nonumber \\
  &&\quad +2 p(p-1)A_2 \vep_p \sum_{k=0}^{k_p-1}
  \left(\begin{array}{ll} p-2
  \\  \,\,\,\,k\end{array}\right) \left(\|\mu\|_{2(k+1)+\gm}
  \|\mu\|_{2(p-1-k)}+\|\mu\|_{2(k+1)}
  \|\mu\|_{2(p-1-k)+\gm}\right)  \nonumber\\
&&\quad
  +\fr{A_2}{2}\|\mu\|_{2p}\|\mu\|_{\gm} -\fr{A_2}{4}\|\mu\|_{2p+\gm}  \nonumber
  \,.\nonumber
\end{eqnarray}
Using H\"{o}lder's inequality we have (for $ s>2$)
\begin{equation}
\|\mu\|_{r}\le \|\mu\|_{2}^{\fr{s-r}{s-2}}\|\mu\|_s
^{\fr{r-2}{s-2}}\,, \quad  2\le r\le s
\label{(3.28)}
\end{equation}
from which we obtain for all $s_1, s_2\ge 2$ satisfying $s_1+s_2\le
2p+2$
\begin{eqnarray*}
&& \|\mu\|_{s_1}\|\mu\|_{s_2}\le
  \|\mu\|_{2}^{\fr{2p-s_1+2p-s_2}{2p-2}} \|\mu\|_{2p}^{\fr{s_1+s_2-4}{2p-2}}
  \le \|\mu\|_2\|\mu\|_{2p}
\end{eqnarray*}
where we used
$\|\mu\|_{2}\le \|\mu\|_{2p}$.  Thus
\begin{eqnarray*}
  \left\langle
    Q(\mu,\mu),\,\langle\cdot\rangle^{2p}\right\rangle &\le&
  4A_2\left\{\sum_{k=1}^{k_p}\left(\begin{array}{ll} p
 \\ k\end{array}\right)+2p(p-1)
  \sum_{k=0}^{k_p-1}\left(\begin{array}{ll} p-2
 \\ \,\,\,\,k\end{array}\right)\right\}\|\mu\|_2\|\mu\|_{2p}
 \\
  &+& \fr{A_2}{2}\|\mu\|_2\|\mu\|_{2p}-\fr{A_2}{4}\|\mu\|_{2p+\gm}
 \\
  &\le& 4A_2\left( 2^p-1+p(p-1)2^{p-1} \right) \|\mu\|_2\|\mu\|_{2p}
  +\fr{A_2}{2}\|\mu\|_2\|\mu\|_{2p}- \fr{A_2}{4}\|\mu\|_{2p+\gm}
\\
  &\le&2^{2p+1}
  A_2\|\mu\|_2\|\mu\|_{2p}-\fr{A_2}{4}\|\mu\|_{2p+\gm}\end{eqnarray*}
which proves \eqref{(3.22)} for $\|\mu\|_0=1$, where we used the inequality
\eqref{(3.1)} and
$$2^{p}+p(p-1)2^{p-1}\le  2^{2p-1}\,,\quad  p\ge 3\,.$$

Now suppose that $0<\gm<2$. This implies $a=2/\gm>1$. Recall definitions of
$Z_q$ and $Z_p^*$ in \eqref{(3.25)}. Then applying (\ref{(3.2)}) and
(\ref{(3.4)}) we compute for all $k\in\{1, 2,..., k_p\}$
\begin{eqnarray*}
&&\|\mu\|_{2k+\gm}\|\mu\|_{2(p-k)}+ \|\mu\|_{2k}\|\mu\|_{2(p-k)+\gm}
\\
&&=\|\mu\|_{\gm(ak+1)}\|\mu\|_{\gm a(p-k)}+ \|\mu\|_{\gm ak}\|\mu\|_{\gm
(a(p-k)+1)}
\\
&&=Z_{ak+1} Z_{a(p-k)}\Gm(ak+1)\Gm(a(p-k)) +
Z_{ak}Z_{a(p-k)+1}\Gm(ak)\Gm(a(p-k)+1)
\\
&&\le Z_p^*\Gm(ap+1)\Big({\rm B}(ak+1, a(p-k))+{\rm B}(ak, a(p-k)+1)\Big)
\\
&&=Z_p^*\Gm(q+1){\rm B}(ak, a(p-k))\,,
\end{eqnarray*}
 and for all $k\in\{0,1,...,
k_p-1\}$
\begin{eqnarray*}
&&\|\mu\|_{2(k+1)+\gm} \|\mu\|_{2(p-1-k)}+\|\mu\|_{2(k+1)}
\|\mu\|_{2(p-1-k)+\gm}
\\
&&=Z_{a(k+1)+1}Z_{a(p-k-1)}\Gm(a(k+1)+1)\Gm(a(p-1-k))
\\
&&+Z_{a(k+1)} Z_{a(p-1-k)+1}\Gm(a(k+1))\Gm(a(p-1-k)+1)
\\
&&\le Z_p^*\Gm(q+1){\rm B}(a(k+1), a(p-1-k)) \,.
\end{eqnarray*}
This together with $\Gm(q+1)/\Gm(q)=q$ and Lemma \ref{lem3.3} gives from
\eqref{(3.Q)} that
\begin{eqnarray}  \label{(3.29)}
&& \fr{\left\langle
Q(\mu,\mu),\,\langle\cdot\rangle^{2p}\right\rangle}{\Gm(q)} \le Z_p^* q
A_2\sum_{k=1}^{k_p}\left(\begin{array}{ll} p
\\
k\end{array}\right){\rm B}(ak,a(p-k))\\ \nonumber &&+ Z_p^* 2q p(p-1)A_2\vep_p
\sum_{k=0}^{k_p-1}\left(\begin{array}{ll} p-2
\\ \nonumber
\,\,\,\,k\end{array}\right){\rm B}(a(k+1), a(p-1-k))\\ \nonumber
&&+\fr{A_2\|\mu\|_2}{2}Z_q-\fr{A_2}{4}\fr{\|\mu\|_{2p+\gm}}{\Gm(q)}
\\ \nonumber
&& \le  Z_p^* A_2 C_a q^{2-a} +  Z_p^* A_2 C_a q^{3-a}\vep_p
+\fr{A_2\|\mu\|_2}{2}Z_q-\fr{A_2}{4}\fr{\|\mu\|_{2p+\gm}}{\Gm(q)} \,.
\end{eqnarray}
For the negative term we use H\"{o}lder's inequality, $\|\mu\|_0=1$,
and $q=ap=\fr{2p}{\gm }$ to get
$$\|\mu\|_{2p+\gm}\ge
\|\mu\|_{2p} ^{1+\fr{\gm}{2p}}=\|\mu\|_{\gm q} ^{1+\fr{1}{q}}
$$ and so
\begin{eqnarray}  \label{(3.30)}
\fr{\|\mu\|_{2p+\gm}}{\Gm(q)} \ge \Gm(q)^{\fr{1}{q}}\left(\fr{\|\mu\|_{\gm
q}}{\Gm(q)}\right)^{1+\fr{1}{q}}=\Gm(q)^{\fr{1}{q}}Z_q^{1+\fr{1}{q}}\ge
\fr{q}{4}Z_q^{1+\fr{1}{q}}
\end{eqnarray} where we have used the inequality $\Gm(q)^{\fr{1}{q}}\ge q/4$.
Thus \eqref{(3.23)} (with $\|\mu\|_0=1$) follows from \eqref{(3.29)}.

\noindent {\bf Part (II).} In this case we have $\gm=2$, i.e. $a=1$ so
that $q=p$ and hence (\ref{(3.30)}) becomes
$$\fr{\|\mu\|_{2(p+1)}}{\Gm(p)} \ge \fr{p}{4}
Z_p^{1+\fr{1}{p}}.$$ By part (II) of Lemma \ref{lem3.5} we have, as shown
above, that (the special term $\|\mu\|_{2k}\|\mu\|_{2(p-k+1)}$ for $k=1$ in
the sum should be treated separately)
\begin{eqnarray*}
  && \fr{\left\langle
      Q(\mu,\mu),\,\langle\cdot\rangle^{2p}\right\rangle}{\Gm(p)}
  =\fr{1}{2\Gm(p)}\intt_{\mathbb{R}^N \times \mathbb{R}^N}
  L_B\left[\Dt\langle \cdot\rangle^{2p}\right](v,v_*){\rm d}\mu(v){\rm d}\mu(v_*)\\
  &&\le \fr{1}{\Gm(p)}\cdot\fr{12A^*_{p_1}}{p^{\eta}}
  \sum_{k=1}^{k_p}\left(\begin{array}{ll} p
      \\
      k\end{array}\right) \left(\|\mu\|_{2(k+1)}\|\mu\|_{2(p-k)}+
    \|\mu\|_{2k}\|\mu\|_{2(p-k+1)} \right)\\
  &&+\fr{1}{4\Gm(p)}A_0\|\mu\|_2\|\mu\|_{2p}-
  \fr{A_0}{4}\fr{\|\mu\|_{2(p+1)}}{\Gm(p)}
  \\
  &&= \fr{1}{\Gm(p)}\cdot\fr{12A^*_{p_1}}{p^{\eta}}
  \sum_{k=2}^{k_p}\left(\begin{array}{ll} p
      \\
      k\end{array}\right) \left(\|\mu\|_{2(k+1)}\|\mu\|_{2(p-k)}+
    \|\mu\|_{2k}\|\mu\|_{2(p-k+1)} \right)
  \\
  &&+\fr{1}{\Gm(p)}\cdot\fr{12A^*_{p_1}}{p^{\eta}}\left(\begin{array}{ll}
      p
      \\
      1\end{array}\right)
  \|\mu\|_{4}\|\mu\|_{2(p-1)}+\fr{1}{\Gm(p)}\cdot\fr{12A^*_{p_1}}{p^{\eta}}
  \left(\begin{array}{ll} p
      \\
      1\end{array}\right)\|\mu\|_{2}\|\mu\|_{2p}
  \\
  &&+\fr{A_0}{4}\|\mu\|_2\fr{\|\mu\|_{2p}}{\Gm(p)}
  -\fr{A_0}{4}\fr{\|\mu\|_{2(p+1)}}{\Gm(p)}
  \\
  &&\le \wt{Z}_p^*\fr{12A^*_{p_1}}{p^{\eta}}\cdot
  p\sum_{k=2}^{k_p}\left(\begin{array}{ll} p
      \\
      k\end{array}\right){\rm B}(k,p-k)+\wt{Z}_p^*\cdot\fr{12A^*_{p_1}}
  {p^{\eta}}p\left(\begin{array}{ll} p
      \\
      1\end{array}\right){\rm B}(2,p-1)
  \\
  &&+\Big(12A^*_{p_1}p^{1-\eta}+\fr{A_0}{4}\Big)\|\mu\|_{2}Z_p
  -\fr{p}{16}A_0Z_{p}^{1+\fr{1}{p}}
  \\
  &&\le \wt{Z}_p^*\fr{12A^*_{p_1}}{p^{\eta}}\cdot
  p\sum_{k=1}^{k_p}\left(\begin{array}{ll} p
      \\
      k\end{array}\right){\rm
    B}(k,p-k)+\Big(12A^*_{p_1}p^{1-\eta}+\fr{A_0}{4}\Big)\|\mu\|_{2}Z_p
  -\fr{p}{16}A_0Z_p^{1+\fr{1}{p}}
  \\
  &&\le 48A^*_{p_1}p^{1-\eta}(\log p) \wt{Z}_p^*+
  \Big(12A^*_{p_1}p^{1-\eta}+\fr{A_0}{4}\Big)\|\mu\|_{2}Z_p
  -\fr{p}{16}A_0Z_{p}^{1+\fr{1}{p}}
\end{eqnarray*}
where in the last inequality we used Lemma \ref{lem3.3}. This proves
\eqref{(3.26)} for $\|\mu\|_0=1$.
\end{proof}

\subsection{An ODE comparison inequality}
\label{sec:an-ode-comparison-inequality}

Finally we shall conclude this section by proving an ODE comparison inequality which
will be useful for proving moment production estimates.

\begin{lemma} \label{lem3.7}
Given any $A>0, B>0, \vep>0$, we have:
\begin{itemize}
\item[(I)] The function
$$Y(t)=\left(\fr{A}{B(1-e^{-\vep A t})}\right)^{1/\vep}\,,\quad
t>0$$ is the unique positive $C^1$-solution of the equation
$$\fr{{\rm d}}{{\rm d}t}Y(t)=AY(t)-B Y(t) ^{1+\vep}\,,\quad t> 0\,;\quad Y(0+)=\infty\,. $$

\item[(II)]  Let $u(t)$ be a non-negative function in $(0,\infty)$ with the
    properties
that $u$ is absolutely continuous on every bounded closed subinterval of
$(0,\infty)$ and
$$ \left(\fr{{\rm d}
}{{\rm d}t}u(t)\right)1_{\{u(t)>Y(t)\}}\le \left(Au(t)-B u(t) ^{1+\vep}\right)1_{\{u(t)>Y(t)\}}\quad {\rm a.e.}\quad t\in(0,\infty)\,. $$
 Then $\,u(t)\le Y(t)$ for all $t\in(0,\infty)$.
\end{itemize}
\end{lemma}

\begin{proof}[Proof of  Lemma~\ref{lem3.7}] Part (I) is obvious. To prove part (II) we use the
  assumption on $u$ and notice that the function $x\mapsto
  Bx^{1+\vep}-A x$ is increasing in $((A/B)^{1/\vep},\infty)$ and
  $Y(t)>(A/B)^{1/\vep}$. Then it follows from the assumption of the lemma that \begin{eqnarray*}&&
  \Big(\fr{{\rm d}}{{\rm d}t}u (t)-\fr{{\rm d}}{{\rm d}t}Y(t)\Big)1_{\{u(t)>Y(t)\}}
\\
&&\le \Big(BY(t)^{1+\vep}-AY(t)-Bu(t)^{1+\vep}+Au(t)\Big)1_{\{u(t)>Y(t)\}}
  \le 0 \quad {\rm a.e.}\quad t\in(0,\infty).\end{eqnarray*}
Thus by the absolute
  continuity of $u$ we have for any $t>t_*>0$
\begin{eqnarray*}&&(u(t)-Y(t))^{+}\\
 && =(u(t_*)-Y(t_*))^{+}+\int_{t_*}^t \Big( \fr{{\rm d}}{{\rm
      d}\tau}u(\tau)-\fr{{\rm d}}{{\rm
      d}\tau}Y(\tau)\Big)1_{\{u(\tau)>Y(\tau)\}} {\rm d}\tau \le
  (u(t_*)-Y(t_*))^{+}\,.
  \end{eqnarray*}
 From this we see it is enough to prove that
for any $t>0$ there is $t_*\in (0,t)$ such that $u(t_*)\le
Y(t_*)$. Otherwise there were $t_0>0$ such that $u(t)>Y(t)$ for all
$t\in (0,t_0)$. By assumption on $u$, this implies
$$\fr{{\rm d}}{{\rm d}t}u (t)\le Au(t)-Bu(t)^{1+\vep} \quad {\rm a.e.}\quad t\in (0,t_0)\,.$$
On the other hand, from the lower bound $Y(t)>(A/B)^{1/\vep}$ we see that the
function $t\mapsto u^{-\vep}(t)$ is absolutely continuous on every
closed subinterval of $(0,t_0]$. We then compute for a.e. $t\in(0,t_0)$
\begin{eqnarray*}&& \fr{\rm d}{{\rm d}t}(u^{-\vep}(t)) \ge -\vep Au^{-\vep}(t)+ \vep
B\end{eqnarray*} and hence for any $0<\tau<t_0$ we have by the absolute continuity of
$t\mapsto u^{-\vep}(t)e^{\vep A t}$ on $[\tau, t_0]$ that
$$u^{-\vep}(t)e^{\vep A t}
\ge u^{-\vep}(\tau)e^{\vep A \tau }+\fr{B(e^{\vep A t}-e^{\vep A \tau} )}{A} \,,\quad
\forall\,t\in [\tau, t_0]\,.$$ Omitting the positive term $u^{-\vep}(\tau)e^{\vep A \tau
}$ and letting $\tau\to 0+$ leads to
$$u^{-\vep}(t)e^{\vep A t}
\ge \fr{B(e^{\vep A t}-1 )}{A}\,,\quad \forall\, t\in (0, t_0]$$ i.e.
$$u(t)\le\left(\fr{A}{B(1-e^{-\vep A t})}\right)^{1/\vep}=Y(t)\quad
\forall\, t\in (0, t_0]$$ which contradicts the assertion
``$u(t)>Y(t)$ for all $t\in (0,t_0)$''. This prove the existence of
$t_*\in(0,t)$ for all $t>0$ and therefore concludes the proof of the lemma.
\end{proof}

\section{Construction of weak measure solutions: Proof of Theorem
  \ref{theo1}}
\label{sec4}

For notation convenience we denote
$$\int_{\mathbb{R}^N}\vp {\rm d}F_t=
\int_{\mathbb{R}^N}\vp(v){\rm d}F_t(v),\quad {\rm etc.}$$
And note that if $F_t$ is a measure weak solution of Eq.~\eqref{(B)},
then for any $\vp\in C^2_b(\mathbb{R}^N)$ we have
\begin{equation}
\int_{\mathbb{R}^N}\vp{\rm d}F_{t}=\int_{\mathbb{R}^N}\vp{\rm
  d}F_{t_0} +\int_{t_0}^{t} \left\langle Q(F_\tau,F_{\tau}),\,\vp \right\rangle {\rm
  d}\tau \quad \forall\, t>t_0> 0\,.
\label{(4.1)}
\end{equation}

Our proofs of the parts (a)-(b)-(c)-(d) of Theorem \ref{theo1} are
contained in the following three steps.  \medskip

\subsection*{Step 1. A priori estimates for measure weak solutions}

We first prove part (b) and moreover we prove that the solution $F_t$
in part (b) satisfies that for any $s\ge 0$ and any $\vp\in
L^{\infty}_{-s}\cap C^2 (\mathbb{R}^N)$,
\begin{equation}
t\mapsto \langle Q(F_t,F_t),\,\vp\rangle
\quad {\rm is \,\,\,continuous\,\,\, in}\quad  (0,\infty)
\label{(4.4)}
\end{equation}
and
\begin{equation}
\fr{\rm d}{{\rm d}t}\int_{\mathbb{R}^N}\vp{\rm d}F_t=\langle
Q(F_t,F_t),\,\vp\rangle
\quad \forall\, t>0\,.
\label{(4.5)}
\end{equation}
And these integrals are absolutely convergent for any $t>0$. Then we prove
that  $F_t$ satisfies the moment production estimates in parts (c) and (d) of
Theorem \ref{theo1}.

Now let $F_t$ satisfy the assumptions in part (b). Recall that $F_t$ already
conserves the mass as mentioned in Definition~\ref{def:measure-weak}.
Therefore the assumption $\|F_t\|_2\le \|F_0\|_2\,(\forall\, t>0)$ is
equivalent to the energy inequality
\begin{equation}\int_{\mathbb{R}^N}|v|^2{\rm d}F_t(v)\le
\int_{\mathbb{R}^N}|v|^2{\rm d}F_0(v)\quad \forall\, t>0. \label{(4.E)}
\end{equation}

Since our test function space for defining measure weak solutions is only
$C^2_b(\mathbb{R}^N)$, we need a truncation-mollification approximation. Let
$\chi\in C^{\infty}_c(\mathbb{R}^N)$ satisfy $0\le \chi\le 1$ on
$\mathbb{R}^N$ and $\chi(v)=1$ for $|v|\le 1$,\, $\chi(v)=0$ for $|v|\ge 2$.
Given any $s\ge 0$ and any $\vp\in L^{\infty}_{-s}\cap C^2(\mathbb{R}^N)$, let
$\vp_n(v):=\vp(v)\chi(v/n)$. It is easily seen that $\vp_n\in
C^2_c(\mathbb{R}^N)\subset C^2_b(\mathbb{R}^N)$ and their Hessian matrices
satisfy
\[
\sup\limits_{n\ge 1}\left|H_{\vp_n}(v)\right|\le C_{\vp}\langle v\rangle^s\,.
\]
Thus by \eqref{(1.8)} we have for any $s_1>s+2+\gm$
\[
\sup_{n\ge 1}\fr{\left|L_B\left[\Dt\vp_n\right](v,v_*)\right|}{\langle
    v\rangle^{s_1}+\langle v_*\rangle^{s_1}}\le C_{\vp} A_{2}
    \fr{(\langle v\rangle^s+\langle
    v_*\rangle^{s})|v-v_*|^{2+\gm}}{\langle v\rangle^{s_1}+\langle v_*\rangle^{s_1}}\to
  0
\]
as $|v|^2+|v_*|^2\to\infty$, and by part (II) of Proposition
\ref{prop2.1}, we deduce
\[
\lim_{n\to\infty}L_{B}\left[\Dt
  \vp_n\right](v,v_*)=L_{B}\left[\Dt \vp\right](v,v_*)\qquad \forall\,
  (v,v_*)\in\mathbb{R}^N \times \mathbb{R}^N\,.
\]
Thus by \eqref{(4.1)}, the assumption \eqref{(1.E)} and the dominated
convergence theorem we obtain
$$\lim_{n\to\infty}\int_{t_0}^t \left\langle
  Q(F_{\tau},F_{\tau}),\vp_n\right\rangle {\rm d}\tau =\int_{t_0}^t
\left\langle Q(F_{\tau},F_{\tau}),\vp \right\rangle {\rm d}\tau\quad \forall\,
t>t_0>0$$ and thus \eqref{(4.1)} holds for all $\vp\in \bigcup_{s\ge
  0}L^{\infty}_{-s}\cap C^2(\mathbb{R}^N)$.

Since $\psi_j(v)=v_j$, $j=1, \dots, N$, and $\psi(v)=|v|^2$ belong to
$L^{\infty}_{-2}\cap C^2(\mathbb{R}^N)$ and $\Dt\psi_j=\Dt\psi=0$, it follows
from (\ref{(4.1)}) that $F_t$ conserves the momentum and energy in the
\emph{open} interval $(0,\infty)$. Therefore in order to prove the
conservation of momentum and energy in the \emph{closed} interval
$[0,\infty)$, we only have to prove that
\begin{equation}\label{(4.EE)}\lim_{t\to
    0^+}\int_{\mathbb{R}^N}v_j{\rm d}F_t(v)=
  \int_{\mathbb{R}^N}v_j{\rm d}F_0(v),\quad\lim_{t\to
    0^+}\int_{\mathbb{R}^N}|v|^2{\rm d}F_t(v)=
  \int_{\mathbb{R}^N}|v|^2{\rm d}F_0(v)\end{equation} for $j= 1,
2,\dots, N$.

Let $\chi(v)$ be given above and let $\vep>0$. Then $v\mapsto
v_j\chi(\vep v), v\mapsto |v|^2\chi(\vep v)$ belong to
$C^2_c(\mathbb{R}^N)\subset C^2_b(\mathbb{R}^N)$ so that, by
definition of measure weak solutions, the functions
\[
t\mapsto \int_{\mathbb{R}^N}v_j\chi(\vep v){\rm d}F_t(v) \quad \mbox{
  and } \quad t\mapsto \int_{\mathbb{R}^N}|v|^2\chi(\vep v){\rm
  d}F_t(v)
\]
are all continuous on $[0,\infty)$.  Since
\[
|v_j-v_j\chi(\vep v)|\le |v|1_{\{|v|\ge 1/\vep\}}\le \vep |v|^2
\] and
\[
C:=\sup_{t\ge
  0}\int_{\mathbb{R}^N}|v|^2{\rm d}F_t(v)\le
\int_{\mathbb{R}^N}|v|^2{\rm d}F_0(v)<\infty,
\]
it follows that
$$\int_{\mathbb{R}^N}v_j{\rm
  d}F_t(v)=\int_{\mathbb{R}^N}v_j\chi(\vep v){\rm
  d}F_t(v)+O(\vep)\qquad \forall\,t\ge 0$$ where $|O(\vep)|\le C\vep.$
Thus letting $t\to 0^+$ gives $$ \lim_{t\to
  0^+}\int_{\mathbb{R}^N}v_j{\rm d}F_t(v)
=\int_{\mathbb{R}^N}v_j\chi(\vep v){\rm d}F_0(v) +O(\vep).$$ Then
letting $\vep\to 0^+$ leads to the first equality in (\ref{(4.EE)})
for $j=1,2,\dots ,N$. Next using $|v|^2\ge |v|^2\chi(\vep v)$ and the
inequality ({\ref{(4.E)}) we have
$$\int_{\mathbb{R}^N}|v|^2{\rm
    d}F_0\ge\lim_{t\to 0^+}
\int_{\mathbb{R}^N}|v|^2{\rm
    d}F_t\ge \lim_{t\to 0^+}\int_{\mathbb{R}^N}|v|^2\chi(\vep v){\rm d}F_t(v)=
 \int_{\mathbb{R}^N}|v|^2\chi(\vep v){\rm d}F_0(v)$$
which leads to the second equality in (\ref{(4.EE)}) by letting $\vep \to
0^+$.

Next let's prove \eqref{(4.4)} and \eqref{(4.5)}. Given any $s\ge 0$ and
$\vp\in L^{\infty}_{-s}\cap C^2(\mathbb{R}^N)$.  For any $0<\dt<T<\infty$, by
denoting
\[
C_{\dt,T,s}=\sup_{\dt\le t\le T}\|F_t\|_{s}^2<\infty
\]
and using \eqref{(1.8)} we have
\begin{eqnarray*}
&&\left| \int_{\mathbb{R}^N}\vp{\rm
d}F_{t_1}-\int_{\mathbb{R}^N}\vp{\rm d}F_{t_2}\right| \le
C_{\vp}A_{2}C_{\dt,T,s}|t_1-t_2|\qquad \forall\, t_1, t_2\in [\dt, T]\,.\end{eqnarray*}
So
\begin{equation}
t\mapsto\int_{\mathbb{R}^N}\vp{\rm d}F_{t}
\quad {\rm is \,\,\,continuous\,\,\, in}\,\,\, t\in
(0,\infty)\,.
\label{(4.6)}
\end{equation}
In order to prove \eqref{(4.4)}, we need
only to show that for any fixed $t>0$ and any sequence $\{t_n\}\subset
[t/2, \,3t/2]$ satisfying $t_n\to t\ (n\to\infty)$ we have
\begin{equation}\label{(4.7)}
\lim_{n\to\infty} \left\langle Q(F_{t_n}, F_{t_n}),\vp \right\rangle=
\left\langle Q(F_{t},
F_{t}),\vp \right\rangle \,.
\end{equation}
This is an application of Proposition \ref{prop2.2}. In fact by
Proposition \ref{prop2.1} we know that $(v,v_*)\mapsto
L_B[\Dt\vp](v,v_*)$ is continuous on $\mathbb{R}^N \times
\mathbb{R}^N$, and as shown above
\begin{eqnarray*}
  &&  \fr{\left|L_B\left[\Dt\vp\right](v,v_*)\right|}{\langle v\rangle^{s_1}+\langle v_*\rangle^{s_1}}\le
  C_{\vp} A_{2} \fr{(\langle v\rangle^s+\langle v_*\rangle^{s})
  |v-v_*|^{2+\gm}}{\langle
    v\rangle^{s_1}+\langle v_*\rangle^{s_1}}\to 0
\end{eqnarray*}
for all $s_1>s+2+\gm$ as $|v|^2+|v_*|^2\to\infty$.  Since
\[
\sup\limits_{t/2\le \tau\le
  3t/2}\|F_{\tau}\|_{s_1}<\infty,
\]
it follows from Proposition \ref{prop2.2} and the weak-star convergence
$F_{t_n}\rightharpoonup F_{t}$ $(n\to\infty)$ (see \eqref{(4.6)}) that
\eqref{(4.7)} and therefore \eqref{(4.4)} hold true.

The differential equation \eqref{(4.5)} follows from the continuity property
\eqref{(4.4)} and from the equation \eqref{(4.1)} which has been proven to
hold for all $\vp\in L^{\infty}_{-s}\cap C^2(\mathbb{R}^N)$.

Now for any $s\ge 6$,  applying \eqref{(4.5)} to
 $\vp(v)=\langle v\rangle^{s}$, which belongs to $L^{\infty}_{-s}\cap
C^2(\mathbb{R}^N)$, and applying Lemma \ref{lem3.6} with $p=s/2$  we have for
any $t>0$
\begin{eqnarray*}&&\fr{\rm d}{{\rm d}t}\|F_t\|_{s}=\langle Q(F_t,F_t),\,
\langle\cdot\rangle^{s}\rangle\le 2^{s+1}
A_2\|F_0\|_2\|F_t\|_{s}-\fr{1}{4}A_2\|F_0\|_0\|F_t\|_{s+\gm}.
\end{eqnarray*}
Since, by using the inequality \eqref{(3.28)},
\begin{eqnarray*}&& \|F_t\|_{s+\gm}\ge (\|F_0\|_2)^{-\fr{\gm}{s-2}}
\left(\|F_t\|_{s}\right)^{1+\fr{ \gm}{s -2}}
\end{eqnarray*}
it follows that
\begin{eqnarray*}
&& \fr{\rm d}{{\rm d}t}\|F_t\|_{s}\le
  2^{s+1}A_2\|F_0\|_2\|F_t\|_{s}-\fr{1}{4}A_2\|F_0\|_0(\|F_0\|_2)^{-\fr{\gm}{s-2}}
  \left(\|F_t\|_{s}\right)^{1+\fr{ \gm}{s -2}}\quad
  \forall\,t>0\,.
\end{eqnarray*}
Thus using Lemma \ref{lem3.7} we obtain
\begin{eqnarray*}
&& \|F_t\|_{s}\le
  \left(\fr{2^{s+1}A_2\|F_0\|_2}{\fr{1}{4}A_2\|F_0\|_0(\|F_0\|_2)^{-\fr{\gm}{s-2}}
      \Big(1-\exp(-\fr{\gm}{s-2}2^{s+1}A_2\|F_0\|_2
      t)\Big)}\right)^{\fr{s-2}{\gm}}\quad \forall\,
  t>0\,.
\end{eqnarray*}
Since $s\ge 6$ implies $2^{s}\ge 8(s-2)$, this gives
\begin{eqnarray*}
&&\fr{\gm}{s-2}2^{s+1}A_2\|F_0\|_2 \ge
  16A_2\|F_0\|_2\gm =:\beta
\end{eqnarray*}
and hence
\begin{eqnarray*}
&&
  \|F_t\|_{s}\le \|F_0\|_2\left(\fr{\|F_0\|_2}{\|F_0\|_0}\cdot\fr{2^{s+3}}{ 1-e^{-\beta
        t}}\right)^{\fr{s-2}{\gm}}\,,\quad t>0\,,\quad s\ge
  6\,.
\end{eqnarray*}
Applying this estimate to $s=6$ we also obtain that for any $2\le s<6$
\begin{eqnarray*}
&& \|F_t\|_{s}\le
  (\|F_0\|_2)^{\fr{6-s}{4}}(\|F_t\|_{6})^{\fr{s-2}{4}}
  \le(\|F_0\|_2)^{\fr{6-s}{4}}(\|F_0\|_2)^{\fr{s-2}{4}}
  \left(\fr{\|F_0\|_2}{\|F_0\|_0}\cdot\fr{2^{9}}{
      1-e^{-\beta t}}\right)^{\fr{4}{\gm}\times\fr{s-2}{4}}\\
  && =\|F_0\|_2 \left(\fr{\|F_0\|_2}{\|F_0\|_0}\cdot\fr{2^{9}}{ 1-e^{-\beta
        t}}\right)^{\fr{s-2}{\gm}} \,.
\end{eqnarray*}
Maximizing the two cases gives $\max\{2^{s+3}\,,\, 2^9 \}\le 2^{s+7}$ for all
$s\ge 2$  and thus
\begin{equation}
\|F_t\|_{s}\le \|F_0\|_2 \left(\fr{\|F_0\|_2}{\|F_0\|_0}\cdot\fr{2^{s+7}}{
1-e^{-\beta t}}\right)^{\fr{s-2}{\gm}} \qquad
\forall\, t>0\,,\quad \forall\,s\ge 2 .
\label{(4.8)}
\end{equation}

The estimate \eqref{(1.12)} now follows from \eqref{(4.8)} since by using the
inequality
\[
\fr{1}{1-e^{-\beta
    t}}\le \left(1+\frac{1}\beta\right)\left(1+\frac1t\right)
\]
we have
\begin{eqnarray*}
&& \|F_t\|_{s}\le
  \|F_0\|_2\left\{2^{s+7}\fr{\|F_0\|_2}{\|F_0\|_0}
    \left(1+\fr{1}{\beta}\right)\right\}^{\fr{s-2}{\gm}}
  \left(1+\frac1t\right)^{\fr{s-2}{\gm}} ={\mathcal K}_s(F_0)
  \left(1+\frac1t\right)^{\fr{s-2}{\gm}}.
\end{eqnarray*}
Note that from \eqref{(4.8)} and $0<\gm\le 2$ we also have
\begin{equation}
\|F_t\|_{s}\le \fr{\|F_0\|_0}{2^{s+7}}
\left(\fr{\|F_0\|_2}{\|F_0\|_0}\cdot\fr{2^{s+7}}{\left(1-e^{-\beta
t}\right)}\right)^{\fr{s}{\gm}} \qquad \forall\, t>0\,,\quad \forall\,s\ge 2
\label{(4.10)}
\end{equation}
which will be used below.

Now we are going to prove the exponential moment production estimate
\eqref{(1.13)}.
Let $p,q$ be
defined through the following relation (as used in Lemma \ref{lem3.6})
$$ q=ap\quad{\rm with}\quad  a=\fr{2}{\gm}\,.$$
Also recall that $F_t$ conserves the mass and energy, i.e.
$\|F_t\|_0=\|F_0\|_0,\,\|F_t\|_2=\|F_0\|_2$.  We consider two cases:
\smallskip

\noindent {\bf Case 1.} $0<\gm<2$. In this case we have $a>1$. By
Lemma \ref{lem3.6} we have for all $t >0$ and $q\ge 3 a$ (i.e. for all $p\ge 3$)
\begin{eqnarray*}
  &&\fr{{\rm d}}{{\rm d} t}Z_q(t)=\fr{\left\langle
      Q(F_t,F_t),\,\left\langle\cdot\right\rangle^{2p}\right\rangle}{\Gm(q)\|F_0\|_0}
  \\
  &&\le \left(C_a q^{2-a} + C_a q^{3-a}\vep_p \right)A_2\|F_0\|_0 Z_p^*(t)
 +\fr{1}{2}A_2\|F_0\|_2Z_{q}(t)-\fr{q}{16}A_2\|F_0\|_0Z_q(t)^{1+\fr{1}{q}}\,,
\end{eqnarray*}
where
$$
Z_q(t)=\fr{\|F_t\|_{\gm q}}{\Gm(q)\|F_0\|_0}\,,\quad Z_p^*(t)=\max_{k\in\{1,2,
\dots, k_p\}} \{Z_{ak+1}(t)Z_{a(p-k)}(t)\,,\,Z_{ak}(t)Z_{a(p-k)+1}(t)\}\,.
$$
Using $a=2/\gm>1$ and Lemma \ref{lem3.4} we have
\[
C_a q^{2-a} + C_a q^{3-a}\vep_p=o(1)q\quad (q\to\infty)
\]
so that there is a positive
integer $n_0$, depending only on $b(\cdot)$ and $\gm$, such that
$$
n_0\dt\ge 3a\quad {\rm and}\quad C_a q^{2-a} + C_a
q^{3-a}\vep_p\le \fr{q}{32}\qquad \forall\, q\ge n_0\dt,\quad {\rm where}\quad \dt=a-1\,.
$$
Since
\[
q\ge n_0\dt \,\,\Longrightarrow\,\, \fr{1}{2}A_2\|F_0\|_2
<16A_2\|F_0\|_2\gm q=\beta q,
\]
it follows that
\begin{equation}\label{(4.11)}
\fr{{\rm d}}{{\rm d}t}Z_q(t) \le \fr{A_2\|F_0\|_0q}{32}Z_p^*(t)+\beta q
Z_{q}(t)- \fr{q}{16}A_2\|F_0\|_0 Z_q(t)^{1+\fr{1}{q}} \qquad \forall\, q\ge
n_0\dt\,.
\end{equation}
Let
$$\Theta=2^{\gm n_0\dt+7}\fr{\|F_0\|_2}{\|F_0\|_0}\,,\quad
Y_q(t)=\left(\fr{\Theta}{1-e^{-\beta t}}\right)^{q} \,, \quad\, t>0\,.$$ Then
$Y_q$ satisfies the equation
$$\fr{\rm d}{{\rm d}t}Y_q(t)=\beta qY_q(t)-\fr{\beta
q}{\Theta}(Y_q(t))^{1+\fr{1}{q}}\,,\quad t>0\,;\quad Y_q(0+)=\infty\,.$$
We now prove that
\begin{equation}
Z_q(t)\le Y_q(t)\qquad \forall\, t>0\,,\quad \forall\, q\ge
1\,.
\label{(4.12)}
\end{equation}
To do this, it suffices to show that
\begin{equation}\label{(4.13)}
Z_q(t)\le Y_q(t)\qquad \forall\, t>0\,,\quad \forall\, q\in[1,\, n\dt]\,,\quad
n=n_0, n_0+1, n_0+2, \dots\,.
\end{equation}
First of all it is easily seen that  \eqref{(4.13)} holds
for $n=n_0$.  In fact by definitions of $Z_q(t)$ and $Y_q(t)$ and using the
inequality $\Gm(q)>1/2$ ($\forall\,q\ge 1$) and \eqref{(4.10)} we have  for all $1\le q\le
n_0\dt$
\begin{eqnarray*}&&
Z_q(t)\le 2\fr{\|F_t\|_{\gm q}}{\|F_0\|_0} \le \left(\fr{\|F_0\|_2}{\|F_0\|_0}
\cdot \fr{2^{\gm q+7}}{1-e^{-\beta t}}\right)^{q}\le Y_q(t)\quad
\forall\,t>0.\end{eqnarray*} Suppose that \eqref{(4.13)} holds for an integer
$n\ge n_0$.  Take any $q\in [n\dt\,,\, (n+1)\dt]\,.$ Then $q\ge n\dt\ge
n_0\dt$ and so \eqref{(4.11)} holds for such $q$. Recall that $ap=q$.
Since for all integer $1\le k\le k_p=[(p+1)/2]$ there hold
\[
\left\{
\begin{array}{l}
 1<ak< ak+1
\le \fr{(n+1)\dt +a}{2}+1<n\dt,\vspace{0.2cm} \\
 1<a(p-k)< a(p-k) +1\le q-\dt\le n\dt
\end{array}
\right.
\]
it follows from the inductive hypothesis that
\[
\left\{
\begin{array}{l}
Z_{ak+1}(t)Z_{a(p-k)}(t)\le
  Y_{ak+1}(t)Y_{a(p-k)}(t)=Y_{q+1}(t) \,,\vspace{0.2cm} \\
  Z_{ak}(t)Z_{a(p-k)+1}(t)\le
  Y_{ak}(t)Y_{a(p-k)+1}(t)=Y_{q+1}(t) \,.
\end{array}
\right.
\]
Therefore by definitions of $Z_p^*(t), Y_q(t)$ we obtain
$$Z_p^*(t)\le Y_{q+1}(t)=Y_q(t)^{1+\fr{1}{q}}\,,\quad \forall\, t>0\,,\quad
\forall\, q\in[n\dt,\, (n+1)\dt]$$
and hence by \eqref{(4.11)}
 $$\fr{{\rm d}}{{\rm d} t}Z_q(t)\le \beta q Z_{q}(t)+\fr{A_2\|F_0\|_0}{32}q Y_q(t)^{1+\fr{1}{q}}
- \fr{A_2\|F_0\|_0}{16} qZ_q(t)^{1+\fr{1}{q}} \qquad
\forall\,t>0$$ for all $q\in[n\dt,\, (n+1)\dt]$. From this we obtain the
following inequality:
\begin{eqnarray*}
&&\left(\fr{{\rm d}
}{{\rm d} t}Z_q(t)\right)1_{\{Z_q(t)>Y_q(t)\}}\le\left(\beta q Z_{q}(t)- \fr{\beta q}{\Theta}
Z_q(t)^{1+\fr{1}{q}}\right)1_{\{Z_q(t)>Y_q(t)\}}\quad \forall\, t>0
\end{eqnarray*}
where we used the obvious fact that $$\fr{A_2\|F_0\|_0}{32}>
\fr{\beta}{\Theta}.$$ Thus applying Lemma \ref{lem3.7} we conclude $Z_q(t)\le Y_q(t)$ for all $t>0$. This together with the inductive hypotheses implies that $Z_q(t)\le
Y_q(t)$ for all $t>0$ and all $q\in [1\,,\, (n+1)\dt]\,.$ This proves
\eqref{(4.13)} and thus  \eqref{(4.12)} holds true.

Now let
\begin{equation}\label{(4.C)}
\alpha(t)=\fr{1-e^{-\beta t}}{2\Theta}\,,\quad t>0\,.
\end{equation}
Then by definitions of $Z_q(t), Y_q(t)$ and $Z_q(t)\le Y_q(t)$ we have for all
$t>0$
\begin{eqnarray*}
\fr{(\alpha(t))^q\|F_t\|_{\gm q}}{q!\|F_0\|_0}\le (\alpha(t))^qZ_q(t)
\le(\alpha(t))^qY_q(t)=\fr{1}{2^q},\quad q=1,2,\dots\end{eqnarray*} and thus
\begin{eqnarray*}
  \int_{\mathbb{R}^N} e^{\alpha(t)\langle v\rangle^{\gm}}
  {\rm d}F_t(v)=\|F_0\|_0+\sum_{q=1}^{\infty}\fr{(\alpha(t))^q}{q!}\|F_t\|_{\gm
    q}\le 2\|F_0\|_0\,.
\end{eqnarray*}
\smallskip

\noindent {\bf Case 2.} $\gm =2$. In this case we have $a=1$ hence $q=p$. From
part (II) of Lemma \ref{lem3.6} with  $p_1,q_1$ and $\eta$ given in \eqref{H3bis}-\eqref{H3bis-1}, we
have for all $p\ge (12 A^*_{p_1}/{A_0})^{2q_1}$ (which is larger
than $5$)
\begin{eqnarray*}&& \fr{{\rm d}}{{\rm d}t}Z_p(t)\le
 48A^*_{p_1}p^{1-\eta}(\log p)\|F_0\|_0 \wt{Z}_p^*(t) \\
 && \qquad \qquad +
\left(12A^*_{p_1}p^{1-\eta}  +\fr{A_0}{4}\right)\|F_0\|_2Z_p(t)
-\fr{A_0\|F_0\|_0}{16}pZ_p(t)^{1+\fr{1}{p}}
\end{eqnarray*}
where
$$Z_p(t)=\fr{\|F_t\|_{2p}}{\Gm(p)\|F_0\|_0}\,,\quad\wt{Z}_p^*(t)
=\max_{k\in\{1,2, \dots,\, k_p\}}Z_{k+1}(t) Z_{p-k}(t)\,,\quad t>0\,.$$ Let us
fix an integer $n_0\ge (12 A^*_{p_1}/{A_0})^{2q_1}$ such that
$$
48A^*_{p_1}p^{1-\eta}\log p\le \fr{A_2}{32}p\,,\quad
12A^*_{p_1}p^{1-\eta}+\frac{A_0}{4}\le 32A_2 p\qquad \forall\,
p\ge n_0\,.$$ Recalling $\beta=32 A_2\|F_0\|_2$ for $\gm=2$, this gives
\begin{equation}
\fr{{\rm d}}{{\rm d}t}Z_p(t)\le \fr{A_2\|F_0\|_0}{32}p \wt{Z}_p^*(t) +\beta p
Z_p(t)- \fr{A_2\|F_0\|_0}{16}pZ_p(t)^{1+\fr{1}{p}} \qquad \forall\,
p\ge n_0\,.
\label{(4.15)}
\end{equation} It will be clear that in the present case all $p$
can be chosen integers. Let
$$\Theta=2^{2n_0+7}\fr{\|F_0\|_2}{\|F_0\|_0}\,,\quad
Y_p(t)=\left(\fr{\Theta}{1-e^{-\beta t}}\right)^{p} \,, \quad\, t>0\,;\  p\ge
1\,.$$ Then $Y_p$ satisfies the equation
$$\fr{\rm d}{{\rm d}t}Y_p(t)=\beta pY_p(t)-\fr{\beta
p}{\Theta}Y_p(t)^{1+\fr{1}{p}}\,,\quad t>0\,;\quad Y_p(0+)=\infty\,.$$ We now
prove that
\begin{equation}
Z_p(t)\le Y_p(t)\qquad \forall\, t>0\,,\quad p=1,2,3, \dots
\label{(4.16)}
\end{equation}
As shown in the Case 1 one sees that \eqref{(4.16)} holds for all integer
$1\le p\le n_0$.  Suppose that \eqref{(4.16)} holds true for some integer $p-1
\ge n_0$. Let us check the case $p$. By $p-1\ge n_0>5$ we have $k_p+1\le
(p+1)/2+1\le p-1$ and so $ Z_{k+1}(t)Z_{p-k}(t) \le
Y_{k+1}(t)Y_{p-k}(t)=(Y_{p}(t))^{1+\fr{1}{p}}$ hold for all $k\in\{1, 2,..., k_p\}$.
So
\begin{eqnarray*}
&& \wt{Z}_{p}^*(t)=\max_{k\in\{1, \, 2,\dots,\,
    k_{p}\}}Z_{k+1}(t)Z_{p-k}(t)\le Y_{p}(t)^{1+\fr{1}{p}}\end{eqnarray*}
hence from \eqref{(4.15)} we obtain
$$\fr{{\rm d}}{{\rm d}t}Z_p(t) \le \beta p
Z_{p}(t)+\fr{A_2\|F_0\|_0}{32}pY_p(t)^{1+\fr{1}{p}} - \fr{A_2\|F_0\|_0}{16}pZ_p(t)^{1+\fr{1}{p}}
 \qquad \forall\, t>0$$
which together with
 $\fr{A_2\|F_0\|_0}{32}> \fr{\beta}{\Theta}$ implies the inequality
\begin{eqnarray*}
  &&\left(\fr{{\rm d}}{{\rm d}t} Z_{p}(t)\right)1_{\{Z_p(t)>Y_p(t)\}}\le
  \left( \beta p Z_{p}(t)- \fr{\beta
    p}{\Theta}Z_p(t)^{1+\fr{1}{p}}\right)1_{\{Z_p(t)>Y_p(t)\}}\quad \forall\, t>0.
\end{eqnarray*}
Applying Lemma \ref{lem3.7} we then conclude that $Z_{p}(t)\le
Y_p(t)\,\,\forall\, t>0\,.$ This proves \eqref{(4.16)}.

As shown above we obtain with the function $\alpha(t)$ defined in
\eqref{(4.C)} that
\begin{equation*}
\int_{\mathbb{R}^N} e^{\alpha(t)\langle v\rangle^{2}} {\rm d}F_t(v)\le
2\|F_0\|_0\qquad \forall\, t>0\,.
\end{equation*}
This completes Step 1.
\medskip

\noindent
\subsection*{Step 2. Construction of solutions for absolutely continuous measures} Suppose that $F_0$ is absolutely continuous with respect to the
Lebesgue measure, i.e. ${\rm d}F_0(v)=f_0(v) {\rm d}v$, and suppose
that (moment bounds and finite entropy)
\[
0\le
f_0\in \bigcap_{s\ge 0}L^1_s(\mathbb{R}^N) \quad \mbox{ and } \quad
0<\int_{\mathbb{R}^N}f_0(v)|\log f_0(v)|{\rm d}v<\infty\,.
\]
In this case we prove that there exists $\{f_t\}_{t\ge 0}\subset \bigcap_{s\ge
0}L^1_s(\mathbb{R}^N)$ such that the measure $F_t$ defined by ${\rm
d}F_t(v)=f_t(v){\rm d}v$ is a conservative measure weak solution of
Eq.~\eqref{(B)} associated with the initial datum $F_0$ and $F_t$ satisfies
the moment production estimates \eqref{(1.12)} and \eqref{(1.13)}.

To do this we consider some bounded truncations $B_n$ of the kernel $B$:
$$
B_n(z,\sg)=\min \{|z|^{\gm},\,n\}\min\{b(\cos\theta),\,n\}\,, \quad
n=1,2,\dots
$$
It is well known that for every $n\ge 1$ the Eq.~\eqref{(B)} with the bounded
kernel $B_n$ has a unique conservative solution $f^n_t(v)$ satisfying
$f^n_0(v)=f_0(v)$ and $f^n\in C^1([0,\infty); L^1_s(\mathbb{R}^N))\cap
L^{\infty}_{\rm loc} ([0,\infty); L^1_s(\mathbb{R}^N))$ for all $s\ge 0$, and
\begin{equation}\label{borne-unif}
\sup_{n\ge 1,\, t\ge 0}\int_{\mathbb{R}^N}f^n_t(v)\left(1+|v|^2+|\log f^n_t(v)|\right)
{\rm d}v<\infty\,.
\end{equation}

Let $Q_{B_n}(\cdot,\cdot)$ (collision operator) and $A_{n,2}$ (angular
momentum defined in {\bf (H0)}) correspond to the kernel $B_n$, and define
${\rm d}F^n_t(v)= f^n_t(v){\rm d}v$. Then $\|F^n_t\|_2=\|F_0^n\|_2=\|F_0\|_2$
and from the proof of Lemmas \ref{lem3.5}-\ref{lem3.6} we see that by omitting
the negative term in the proofs of the two lemmas and noting that $A_{n,2}\le
A_2$ we have for all $p\ge 3$
\begin{eqnarray*}
&&\fr{\rm d}{{\rm d}t}\|F^n_t\|_{2p} =\left\langle
  Q_{B_n}(F^n_{t},F^n_t),\left\langle\cdot\right\rangle^{2p}\right\rangle\le
  2^{2p+1}A_2\|F_0\|_2\|F^n_t\|_{2p} \,.
\end{eqnarray*}
Thus for all $s\ge 6$, letting $p=s/2$ and recalling
$\|f^n_t\|_{L^1_s}=\|F^n_t\|_s$ we obtain
\begin{eqnarray*}
\sup_{n\ge 1}\|f^n_t\|_{L^1_s}\le \|f_0\|_{L^1_s} \exp\left(2^{s+1} A_2\|F_0\|_2
t\right)\qquad \forall\,
t\ge 0\,.
\end{eqnarray*}

From this and the basic estimate \eqref{(1.8)} we get for any $\vp\in
C^2_b(\mathbb{R}^N)$ and any $T\in (0,\infty)$
\begin{eqnarray*}
&& \Big|\int_{\mathbb{R}^N}\vp(v) f^n_{t_1}(v){\rm d}v-\int_{\mathbb{R}^N}\vp
(v)f^n_{t_2}(v){\rm d}v\Big| \le C_{\vp, T} |t_1-t_2|\qquad \forall\, t_1,
t_2\in[0,T]\,.
\end{eqnarray*}
This together with \eqref{borne-unif} implies for any $\psi\in
L^{\infty}(\mathbb{R}^N)$ and any $T\in(0,\infty)$
\begin{equation}
\sup_{t_1,t_2\in[0,T],\,|t_1-t_2|\le \dt;\, n\ge 1}
\left|\int_{\mathbb{R}^N}\psi f^n_{t_1}{\rm d}v -\int_{\mathbb{R}^N}\psi f^n_{t_2}{\rm
d}v\right| \to 0\quad {\rm as}\quad \dt\to 0^+\,.
\label{cvg-Linfty}
\end{equation}

Since \eqref{borne-unif} implies that for every $t\ge 0$,
$\{f^n_t\}_{n=1}^{\infty}$ is $L^1$-weakly relatively compact, it follows from
diagonal argument and \eqref{cvg-Linfty} that there is a subsequence of
$\{n\}$ (independent of $t$), still denoted as $\{n\}$, and a nonnegative
measurable function $(t,v)\mapsto f_t(v)$ on $[0,\infty)\times\mathbb{R}^N$
satisfying $f_t\in L^1(\mathbb{R}^N)$ ($\forall\, t\ge 0$) such that for all
$\psi\in L^{\infty}(\mathbb{R}^N)$
\begin{equation}
\lim_{n\to\infty}\int_{\mathbb{R}^N}\psi f^n_t {\rm d}v
=\int_{\mathbb{R}^N}\psi f_t{\rm d}v\quad \forall\, t\ge
0\,.
\label{(4.19)}
\end{equation}
And consequently
\[
f_t\in \bigcap_{s\ge 0} L^1_s(\mathbb{R}^N) \quad \forall\,t\ge 0\,,
\]
and
\begin{equation}
\sup_{t\ge 0}\|f_t\|_{L^1_2}\le\|f_0\|_{L^1_2}\,,\quad \sup_{0\le t\le
T}\|f_t\|_{L^1_s}<\infty\quad \forall\, 0<T<\infty\,,\quad
\forall\, s\ge 0\,,
\label{(4.20)}
\end{equation}
and for any $s>0$ and any $\psi\in L^{\infty}(\mathbb{R}^N)$
\begin{equation}
t\mapsto \int_{\mathbb{R}^N}\psi f_t{\rm d}v\quad {\rm is \,\,\,continuous
\,\,\,on}\quad [0,\infty)
\,.
\label{(4.21)}
\end{equation}

Now we are going to show that $f_t$ (or equivalently the measure $F_t$ defined
by ${\rm d}F_t(v)=f_t(v){\rm d}v$) is a conservative weak solution of
Eq.~\eqref{(B)} with the kernel $B$. Given any $\vp\in C^2_b(\mathbb{R}^N)$,
we have by \eqref{(1.8)} and $B_n\le B$
\begin{eqnarray*}
  && \sup_{n\ge 1}
  \fr{\left|L_{B_n}\left[\Dt\vp\right](v,v_*)\right|}{\langle v\rangle^{s}
  +\langle
    v_*\rangle^{s}}\le A_2C_{\vp}\fr{ |v-v_*|^{2+\gm}}{\langle v\rangle^{s}
    +\langle v_*
    \rangle^{s}}\to 0\,\quad (|v|^2+|v_*|^2\to\infty)\end{eqnarray*} for
$s>2+\gm.$
Moreover by Proposition \ref{prop2.1}, $L_{B_n}[\Dt\vp](v,v_*),\,
L_B[\Dt\vp](v,v_*)$ are all
continuous on $(v,v_*)\in\mathbb{R}^N \times \mathbb{R}^N$, and
$$
\lim_{n\to\infty}\sup_{|v|+|v_*|\le R}
\left|L_{B_n}\left[\Dt\vp\right](v,v_*)-L_B\left[\Dt\vp\right](v,v_*)
\right|=0
\qquad \forall\,0<R<\infty\,.
$$
It follows from \eqref{(4.19)} and Proposition \ref{prop2.2} that
$$\sup_{0\le t\le T}\intt_{\mathbb{R}^N \times \mathbb{R}^N}
\left|L_{B}\left[\Dt\vp\right](v,v_*)\right| f_t(v) f_t(v_*)
{\rm d}v{\rm d}v_*<\infty\qquad \forall\, 0<T<\infty\,,
$$
$$
\left\langle Q_{B_n}(f^n_{t}, f^n_{t}), \vp\right\rangle \to \left\langle
Q_{B}(f_{t}, f_{t}), \vp \right\rangle \quad (n\to\infty)\quad \forall\, t\ge
0\,.
$$
Again using Proposition \ref{prop2.2} and \eqref{(4.21)} we conclude that
$$
t\mapsto \langle Q_{B}(f_{t}, f_{t}), \vp\rangle \quad {\rm is
  \,\,\,continuous \,\,\,on}\quad [0,\infty)\,.
$$
Finally using the dominated convergence theorem (in the $t$ variable) we
conclude that
\begin{eqnarray*}
&& \int_{\mathbb{R}^N}\vp f_t{\rm d}v
  =\int_{\mathbb{R}^N}\vp f_0{\rm d}v+\int_{0}^t \left\langle Q_{B}(f_{\tau},
  f_{\tau}), \vp \right\rangle {\rm d}\tau\qquad \forall\, t\ge
  0\,.
\end{eqnarray*}
Thus $f_t$ is a weak solution of Eq.~\eqref{(B)}. Let $F_t$ be defined by
${\rm d}F_t(v)=f_t(v){\rm d}v$. Then from $\|F_t\|_s=\|f_t\|_{L^1_s}$,
\eqref{(4.20)}, and Step 1 we conclude that $F_t$ is a conservative measure
weak solution of Eq.~\eqref{(B)} associated with the initial datum $F_0$ and
satisfies the moment production estimates \eqref{(1.12)} and \eqref{(1.13)}.
\medskip

\subsection*{Step 3. The approximation argument and conclusion}
Let $F_0$ be the given measure in ${\mathcal
  B}_2^{+}(\mathbb{R}^N)$ with $\|F_0\|_0\neq 0$. We shall prove the
existence of a measure weak solution $F_t$ that has all properties listed in
the theorem.

First if $F_0=c\dt_{v=v_0}$ ($c>0$) is a Dirac mass, then it is easily checked
that the measure $F_t\equiv c\dt_{v=v_0}$ is a measure weak solution of
Eq.\eqref{(B)} and apparently it conserves the mass, momentum and energy and
has finite moments of all orders. By Step 1 we conclude that $F_t$ satisfies
the moment production estimates \eqref{(1.12)}-\eqref{(1.13)}.

Suppose $F_0$ is not a Dirac mass. We shall use \textbf{Mehler
  transform}: Let
\begin{equation}
\rho=\|F_0\|_0\,,\quad v_0=\fr{1}{\rho}\int_{\mathbb{R}^N}v {\rm
d}F_0(v)\,,\quad T=\fr{1}{
N\rho}\int_{\mathbb{R}^N}|v-v_0|^2{\rm d}F_0(v)\,.
\label{(4.22)}
\end{equation}
Then $T>0$ so that the Maxwellian used in the Mehler transform can be
defined:
\begin{equation}
M(v)=\fr{e^{-|v|^2/2T}}{(2\pi T)^{N/2}}\,,\qquad v\in \mathbb{R}^N\,.
\label{(4.23)}
\end{equation}
The Mehler transform of $F_0$ is defined by
\begin{equation}
f^n_0(v)=e^{Nn}\int_{\mathbb{R}^N}M\left(e^{n}\Big(v-v_0-\sqrt{1-e^{-2n}}\,\,
(v_*-v_0)\Big)\right){\rm d}F_0(v_*)\,,\quad n\ge 1\,.
\label{(4.24)}
\end{equation}
It is well known that
\begin{eqnarray*}
\int_{\mathbb{R}^N} \left( \begin{array}{c} 1 \\ v \\
    |v|^2 \end{array} \right)  f^n_0(v){\rm d}v=\int_{\mathbb{R}^N} \left( \begin{array}{c} 1 \\ v \\
    |v|^2 \end{array} \right) {\rm d}F_0(v)
\end{eqnarray*}
and for all $\psi\in L^{\infty}_{-2}\cap C(\mathbb{R}^N)$
\begin{eqnarray*}
\lim_{n\to\infty} \int_{\mathbb{R}^N}\psi(v)f^n_0(v){\rm
d}v=\int_{\mathbb{R}^N}\psi(v){\rm
d}F_0(v)\,.
\end{eqnarray*}

For every $n$, choose $K_n>n $ such that
\begin{equation}
\int_{\mathbb{R}^N}\Big(f^n_0(v)-\min\{f^n_0(v),\, K_n\}e^{-\fr{|v|^2}{K_n}}
\Big)\langle v\rangle^2 {\rm d}v\le \fr{\|F_0\|_0}{2n}\,.
\label{(4.27)}
\end{equation}
Then let
$$
\wt{f}^n_0(v)=\min\{f^n_0(v),\, K_n\} e^{-|v|^2/n},\qquad
{\rm d}F_0^n(v)=\wt{f}^n_0(v){\rm d}v\,.
$$
We need to prove that
\begin{equation}
\lim_{n\to\infty} \int_{\mathbb{R}^N}\psi{\rm
d}F^n_0=\int_{\mathbb{R}^N}\psi{\rm d}F_0\qquad \forall\,
\psi\in L^{\infty}_{-2}C(\mathbb{R}^N)\,.
\label{(4.28)}
\end{equation}
Indeed we have \begin{eqnarray*}&& \Big| \int_{\mathbb{R}^N}\psi{\rm
d}F^n_0-\int_{\mathbb{R}^N}\psi{\rm
  d}F_0\Big|
\le\Big| \int_{\mathbb{R}^N}\psi(\wt{f}^n_0-f^n_0){\rm
d}v\Big|+\Big|\int_{\mathbb{R}^N}\psi f^n_0{\rm d}v
-\int_{\mathbb{R}^N}\psi{\rm d}F_0\Big|\,. \end{eqnarray*}
 The second term converges to
zero ($n\to\infty$). The first term also goes to zero: By \eqref{(4.27)} we
have
\begin{eqnarray*}
&&\Big| \int_{\mathbb{R}^N}\psi(\wt{f}^n_0-f^n_0){\rm d}v\Big| \le
C\int_{\mathbb{R}^N}\langle v\rangle^2|\wt{f}^n_0-f^n_0|{\rm d}v\le
\fr{C}{2n}\,.
\end{eqnarray*}

Since for every $n$, $\wt{f}^n_0$ satisfies the condition in the Step 2, there
is a conservative measure weak solution $F^n_t$ of Eq.~\eqref{(B)} with the
kernel $B$ and the initial data $F^n_0$, such that $F^n_t$ satisfies the
moment estimates
$$\|F^n_t\|_{s}\le {\mathcal K}_s(F_0^n)(1+1/t)^{\fr{s-2}{\gm}}\qquad \forall\,t>0
\,,\quad \forall\, s\ge 2. $$ Here recall that ${\mathcal K}_s(\cdot)$ is
defined in \eqref{(1.12*)}. By the convergence \eqref{(4.28)} we have
$$\lim_{n\to\infty}{\mathcal K}_s(F_0^n)={\mathcal K}_s(F_0)\qquad \forall\,
s\ge 2 \,.$$ Thus for any $s\ge 2$, $C_s^*:=\sup\limits_{n\ge 1}{\mathcal
K}_s(F_0^n) <\infty$ and hence
\begin{equation} \label{(4.29)}
\sup_{n\ge 1}\|F^n_t\|_{s}\le C_{s}^*\left(1+1/t\right)^{\fr{s-2}{\gm}} \qquad
\forall\, t>0\,,\quad \forall\,s\ge 2\,.
\end{equation}

Next we prove the equi-continuity of $\{F^n_t\}_{n=1}^{\infty}$ in
$t\in[0,\infty)$ (in particular in the neighborhood of $t=0$).  It is only in
this part that the logarithm $|\log(\sin\theta)|$ comes into play. Let
$$
\ld(\theta):= \fr{1}{1+|\log(\sin\theta)|}\,,\quad 0<\theta<\pi\,.
$$
By \eqref{(1.8)} and $0<\gm\ld(\theta)\le \gm\le 2$ we have for any $\vp\in
C_b^2(\mathbb{R}^N)$ \begin{eqnarray*}&&
  \left|\int_{\mathbb{S}^{N-2} ({\bf n})}\Dt\vp \,{\rm d}\og\right|\le
  C_{\vp} \left|\int_{\mathbb{S}^{N-2} ({\bf
        n})}\Dt\vp \,{\rm d}\og\right|^{\fr{2-\gm\ld(\theta)}{2}}
  \le
  C_{\vp}|v-v_*|^{2-\gm\ld(\theta)}(\sin\theta)^{2-\gm\ld(\theta)}
\end{eqnarray*} where here and below $C_{\vp}$ only depends on $\vp$ and $N$.  Then by
using
\[
|v-v_*|^{\gm+2-\gm\ld(\theta)}\le 8 \left(\langle
v\rangle^{\gm+2-\gm\ld(\theta)}+\langle
v_*\rangle^{\gm+2-\gm\ld(\theta)}\right)
\]
and $ (\sin\theta)^{-\gm\ld(\theta)}=e^{\gm(1-\ld(\theta))}\le e^2
$
and recalling \eqref{(1.6)} we obtain
$$
\left|L_{B}\left[\Dt\vp\right](v,v_*)\right|\le C_{\vp}\int_{0}^{\pi}
b(\cos\theta)\sin^N\theta\, \Big(\langle v\rangle^{\gm+2-\gm\ld(\theta)}
+\langle v_*\rangle^{\gm+2-\gm\ld(\theta)}\Big){\rm
  d}\theta \,.
$$
So for all $t>0$ (using Fubini's theorem and \eqref{(4.29)})
\begin{eqnarray}\label{(4.3*)}
&&\intt_{\mathbb{R}^N \times \mathbb{R}^N}
\left|L_{B}\left[\Dt\vp\right](v,v_*)\right| {\rm d}F_{t}^n(v){\rm
d}F_{t}^n(v_*)\\  \nonumber &&\le C_{\vp}\|F_0\|_0\int_{0}^{\pi}
b(\cos\theta)\sin^N\theta\,\|F^n_t\|_{\gm+2-\gm\ld(\theta)} {\rm d}\theta\\
 &&  \nonumber \le C_{\vp,F_0}\int_{0}^{\pi}
b(\cos\theta)\sin^N\theta\, \left(1+\frac1t\right)^{1-\ld(\theta)} {\rm
  d}\theta\,.
\end{eqnarray}
Thus for all $t_1, t_2\in[0,\infty)$ we compute (assuming $t_1<t_2$)
\begin{eqnarray} \label{(4.30)}
  && \int_{t_1}^{t_2}{\rm d}t\intt_{\mathbb{R}^N \times \mathbb{R}^N}
  \left|L_{B}\left[\Dt\vp\right](v,v_*)\right| {\rm
      d}F_{t}^n(v){\rm
      d}F_{t}^n(v_*) \\ \nonumber
&&\le C_{\vp,F_0}\int_{0}^{\pi}
    b(\cos\theta)\sin^N\theta\,{\rm d}\theta \left(1+t_2-t_1
    \right)^{1-\ld(\theta)}\int_{0}^{t_2-t_1} t^{\ld(\theta)-1} {\rm
      d}t \\ \nonumber
&&= C_{\vp,F_0}\int_{0}^{\pi}
    b(\cos\theta)\sin^N\theta\,(1+|\log(\sin\theta)|)\left(1+t_2-t_1
    \right)^{1-\ld(\theta)} (t_2-t_1)^{\ld(\theta)} {\rm
      d}\theta
    \\ \nonumber
    &&=:C_{\vp,F_0}\Og(t_2-t_1).
\end{eqnarray}

Since
\[
\left|\left\langle Q(F^n_t, F^n_t),\vp\right\rangle\right|\le
\intt_{\mathbb{R}^N \times
\mathbb{R}^N}\left|L_{B}\left[\Dt\vp\right](v,v_*)\right| {\rm
d}F_{t}^n(v){\rm d}F_{t}^n(v_*)\,,
\]
it follows that
\begin{eqnarray*}
\sup_{n\ge 1}\Big|\int_{\mathbb{R}^N}\vp{\rm
d}F_{t_2}^n-\int_{\mathbb{R}^N}\vp{\rm d}F_{t_1}^n \Big|&\le& \sup_{n\ge
1}\Big|\int_{t_1}^{t_2}|\langle Q(F^n_t, F^n_t),\vp\rangle| {\rm d}t\Big|\\
&\le& C_{\vp,F_0}\Og(|t_2-t_1|)\to 0
\end{eqnarray*}
as $|t_1-t_2|\to 0$.  We then deduce for any $\psi\in C_c(\mathbb{R}^N)$ that
\begin{equation}
\Ld_{\psi}(\dt):=\sup_{|t_1-t_2|\le \dt;\, n\ge 1}
\left|\int_{\mathbb{R}^N}\psi{\rm d}F^n_{t_1} -\int_{\mathbb{R}^N}\psi{\rm
d}F^n_{t_2}\right| \to 0\quad {\rm as}\quad \dt\to 0^{+}\,.
\label{(4.32)}
\end{equation}
Since $C_{c}(\mathbb{R}^N)$ is separated, it follows from a diagonal argument
that there is a subsequence of $\{n\}$ (independent of $t$), still denoted by
$\{n\}$, and a family $\{F_t\}_{t\ge 0}\subset {\mathcal
B}_2^{+}(\mathbb{R}^N)$, such that
\begin{equation}
\lim_{n\to\infty}\int_{\mathbb{R}^N}\psi{\rm d}F^n_t
=\int_{\mathbb{R}^N}\psi{\rm d}F_t \qquad \forall\, t\ge 0\,,\quad \forall\,\psi\in
C_c(\mathbb{R}^N)\,.
\label{(4.33)}
\end{equation}

Using \eqref{(4.29)} and the fact that $F^n_t$ are conservative
solutions we have
\begin{equation}
\|F_t\|_2 \le \|F_0\|_2\,,\quad \|F_t\|_{s}\le C_{s}^*
\left(1+1/t\right)^{\fr{s-2}{\gm}}
\quad \forall\, t>0\,,\quad \forall\,s\ge 2\,.
\label{(4.34)}
\end{equation}
Also by \eqref{(4.33)} and \eqref{(4.32)} we have
$$\Big|\int_{\mathbb{R}^N}\psi{\rm d}F_{t_1}-\int_{\mathbb{R}^N}\psi{\rm
d}F_{t_2}\Big|\le \Ld_{\psi}(|t_1-t_2|)\,.$$ Hence
\begin{equation}
t\mapsto \int_{\mathbb{R}^N}\psi{\rm d}F_{t} \quad {\rm is \,\,\,
continuous\,\,\, on }\quad [0,\infty)\quad \forall\,\psi\in
C_c(\mathbb{R}^N)\,.
\label{(4.35)}
\end{equation}

We now prove that $F_t$ is a measure weak solution of Eq.~\eqref{(B)}. Given
any $\vp\in C^2_b(\mathbb{R}^N)$, by \eqref{(4.34)} we see that the derivation
of \eqref{(4.3*)} holds also for $F_t$ and so
$$\intt_{\mathbb{R}^N \times \mathbb{R}^N}
\left|L_{B}\left[\Dt\vp\right](v,v_*)\right| {\rm d}F_{t}(v){\rm
d}F_{t}(v_*)<\infty\quad \forall\, t>0 \,.
$$
Next by Proposition \ref{prop2.1} the function $(v,v_*)\mapsto
L_B[\Dt\vp](v,v_*)$ is continuous on $\mathbb{R}^N \times
\mathbb{R}^N$ and
\begin{equation}\label{(4.36)}
\fr{\left|L_B\left[\Dt\vp\right](v,v_*)\right|}{\langle v\rangle^{s}+\langle
v_*\rangle^{s}} \le C_{\vp}A_{2} \fr{|v-v_*|^{2+\gm}}{\langle v\rangle^{s}
+\langle v_*\rangle^{s}}\to
0\quad (|v|^2+|v_*|^2\to \infty)
\end{equation}
for all $s>2+\gm$.  Thus by using
\eqref{(4.29)}-\eqref{(4.33)}-\eqref{(4.36)}, Propositions
\ref{prop2.1} and \ref{prop2.2} we have
\begin{equation}
\left\langle Q(F^n_t, F^n_t),\vp \right\rangle \to \left\langle Q(F_t,
F_t),\vp \right\rangle\quad
(n\to\infty)\qquad \forall\,t>0\,.
\label{(4.37)}
\end{equation}
Similarly by using \eqref{(4.34)}-\eqref{(4.35)}, Propositions \ref{prop2.1}
and \ref{prop2.2} we conclude that
\begin{equation}
t\mapsto  \langle Q(F_t, F_t),\vp\rangle \quad {\rm is \,\,\,
continuous\,\,\, in }\quad (0,\infty)\,.
\label{(4.38)}
\end{equation}

Note that the derivation of \eqref{(4.30)} also holds for $F_t$ and hence we
have for all $T\in(0,\infty)$
\begin{equation}
  \int_{0}^{T}{\rm d}\tau\intt_{\mathbb{R}^N \times \mathbb{R}^N}
 \left|L_{B}\left[\Dt\vp\right](v,v_*)\right|
  {\rm d}F_{t}(v){\rm d}F_{t}(v_*) \le C_{\vp,F_0}\Og(T)<\infty\,.
\label{(4.39)}
\end{equation}
Thus
$$
t\mapsto  \left\langle Q(F_t, F_t),\vp \right\rangle \quad {\rm belongs
\,\,\,to }\quad C((0,\infty))\cap L^1_{\rm loc}([0,\infty))\,.
$$
And it also follows from \eqref{(4.30)}-\eqref{(4.37)} and the
dominated convergence theorem that for all $t>0$ we have
$$
\int_{0}^t \left\langle Q(F^n_{\tau}, F^n_{\tau}),\vp \right\rangle {\rm
  d}\tau \to \int_{0}^t\left\langle
Q(F_{\tau}, F_{\tau}),\vp\right\rangle {\rm d}\tau \quad (n\to\infty)\,.
$$
Thus in the integral equation of measures solutions $F^n_t$, letting
$n\to\infty$ gives
$$
\int_{\mathbb{R}^N}\vp{\rm d}F_t= \int_{\mathbb{R}^N}\vp{\rm d}F_0+ \int_{0}^t
\left \langle Q(F_{\tau}, F_{\tau}),\vp \right \rangle {\rm d}\tau\,\quad
\forall\,t> 0\,.
$$

We have proved that $F_t$ satisfies the conditions (i)-(ii) in the Definition
1.1 of measure weak solutions. So $F_t$ is a measure weak solution of
Eq.~\eqref{(B)} associated with the initial datum $F_0$. Finally from the
moment estimates in \eqref{(4.34)} and Step 1 we conclude that the solution
$F_t$ conserves mass, momentum and energy, and satisfies the moment production
estimates \eqref{(1.12)}-\eqref{(1.13)}. This completes the proof of
Theorem~\ref{theo1}.

\section {Uniqueness and stability for angular cutoff: Proof of
  Theorem \ref{theo2}}
\label{sec5}

This section is devoted to the proof of Theorem \ref{theo2}. We shall first
prove some lemmas on how the sign decomposition of measures behaves with time
integration and with the action of the collision operator.

\subsection{Sign decomposition of measures}
\label{sec:sign-decomp-meas}

As usual we denote $${\mathcal
  B}(\mathbb{R}^N)={\mathcal B}_0(\mathbb{R}^N),\quad
\|\mu\|=\|\mu\|_0=|\mu|(\mathbb{R}^N).$$
For any $\mu\in{\mathcal B}(\mathbb{R}^N)$, let $\mu^{+},\mu^{-}$ be the positive and
negative parts of $\mu$, i.e. $\mu^{\pm}=\fr{1}{2}(|\mu|\pm \mu)$. Let
$h:\mathbb{R}^N\to \mathbb{R}$ be the Borel function satisfying $|h(v)|\equiv 1$ such
that ${\rm d}\mu=h{\rm d}|\mu|$. We may call $h$ the sign function of $\mu$.
Then ${\rm d}\mu^+=\fr{1}{2}(1+h){\rm d}\mu.$ So for any $\mu,\nu\in {\mathcal
  B}(\mathbb{R}^N)$, we have
\begin{equation}
|\mu-\nu|=\nu-\mu+2(\mu-\nu)^{+}\,.
\label{(5.2)}
\end{equation}

Let us now prove that this sign decomposition behaves well with the
time integration.

\begin{lemma}[Sign decomposition and time integration] \label{lem5.1}
Let $\mu_t\in C([a,\infty); {\mathcal
    B}(\mathbb{R}^N)), \nu_a\in {\mathcal B}(\mathbb{R}^N)$, and
$$\nu_t=\nu_{a}+\int_{a}^{t} \mu_{s}{\rm d}s\,,\quad t\ge a, $$
and let $v\mapsto h_t(v)$ be the sign function of the measure $\nu_t$
and let $\kappa_t=(1+h_t)/2$ so that ${\rm d}\nu_t^{+}=\kappa_t
{\rm d}\nu_t$.

Then for any bounded Borel function $\psi$ on $\mathbb{R}^N$, the
functions
\[
t\mapsto
\int_{\mathbb{R}^N} \psi{\rm d}\mu_{t}\, , \quad
t\mapsto \int_{\mathbb{R}^N} \psi{\rm
  d}|\mu_{t}| \quad \mbox{ and }\quad  t\mapsto \int_{\mathbb{R}^N} \psi{\rm
  d}\mu_{t}^{+}
\]
 all belong to $L^1_{{\rm loc}}([a,\infty))$ and for
any $t\in[a,\infty)$ we have
\begin{equation}
\int_{\mathbb{R}^N}\psi{\rm d}\nu_t=\int_{\mathbb{R}^N}\psi{\rm d}\nu_a+
\int_{a}^{t}{\rm d}s\int_{\mathbb{R}^N} \psi{\rm
  d}\mu_{s}\,,
\label{(5.3)}
\end{equation}
\begin{equation}
\int_{\mathbb{R}^N}\psi{\rm d}|\nu_t|=\int_{\mathbb{R}^N}\psi {\rm d}|\nu_a|+
\int_{a}^{t} {\rm d}s\int_{\mathbb{R}^N}\psi h_s{\rm d}\mu_{s}
\,,
\label{(5.4)}
\end{equation}
\begin{equation}
\int_{\mathbb{R}^N}\psi{\rm d}\nu_t^{+}=\int_{\mathbb{R}^N}\psi{\rm d}\nu_a^{+}+
\int_{a}^{t}{\rm d}s\int_{\mathbb{R}^N} \psi{\kappa}_s{\rm
  d}\mu_{s}\,.
\label{(5.5)}
\end{equation}

\end{lemma}

\begin{proof}[Proof of Lemma~\ref{lem5.1}]
  Since the half-sum of \eqref{(5.3)} and \eqref{(5.4)} is equal to
  \eqref{(5.5)}, we only have to prove \eqref{(5.3)} and
  \eqref{(5.4)}. The proof of \eqref{(5.3)} is easy and similar to
  that of \eqref{(5.4)}. By simple function approximation, the proof
  of \eqref{(5.4)} can be reduced to the proof of that for any Borel
  set $E\subset \mathbb{R}^N$, $t\mapsto \int_{E}h_t{\rm d}\mu_t$
  belongs to $L^1_{{\rm loc}}([a,\infty))$ (and so does $t\mapsto
  \int_{\mathbb{R}^N}\psi h_t{\rm d}\mu_t$ for any bounded Borel
  function $\psi$ on $\mathbb{R}^N$) and
\begin{equation}
|\nu_t|(E)=|\nu_a|(E)+
\int_{a}^{t}{\rm d}s\int_{E} h_s{\rm d}\mu_{s}\,,\quad
t\in[a,\infty)\,.
\label{(5.6)}
\end{equation}

By assumption on $\mu_t$, the strong derivative $\fr{{\rm d} }{{\rm
d}t}\nu_t=\mu_t$ exists, and
\[
\|\nu_{t_1}-\nu_{t_2}\|\le \int_{t_1}^{t_2} \|\mu_s\|ds\qquad \forall \, a\le
t_1\le t_2<\infty\,.
\]
This implies that for any Borel set $E\subset \mathbb{R}^N$, $t\mapsto
|\nu_t|(E)$ is Lipschitz on every bounded interval $[a,T]\subset [a,\infty)$:
For all $a\le t_1\le t_2\le T$
\begin{eqnarray*}
  && \left|
    \left|\nu_{t_1}\right|(E)-\left|\nu_{t_2}\right|(E)\right|\le
  |\nu_{t_1}-\nu_{t_2}|(E) \le \int_{t_1}^{t_2} \|\mu_s\|\,{\rm d}s\le
  C_{T}|t_1-t_2|
\end{eqnarray*}
and so $t\mapsto |\nu_t|(E)$ is differentiable for almost every
$t\in[a,\infty)$ and satisfies
$$
|\nu_t|(E)=|\nu_a|(E)+\int_{a}^{t}\fr{{\rm d}}{{\rm d}s}|\nu_s|(E) {\rm
d}s\,\quad \forall\, t\in[a,\infty)\,.
$$
Therefore in order to prove \eqref{(5.6)} we only have to show that for
every Borel set $E\subset\mathbb{R}^N$
\begin{equation}
\fr{{\rm d} }{{\rm d}t}|\nu_t|(E)=\int_{E}h_t {\rm d}\mu_t\,, \quad {\rm
a.e.}\quad t\in[a,\infty)
\label{(5.7)}
\end{equation}
which also implies that
 $t\mapsto \int_{E}h_t{\rm
d}\mu_t$ belongs to $L^1_{{\rm loc}}([a,\infty))$.

For any $t, s\in[a,\infty)$, using
\[
|\nu_{s}|(E)=\int_{E}{\rm d}|\nu_{s}|\ge \int_{E}h_t{\rm d}\nu_s
\]
we have
\begin{equation}\label{(5.8)}
|\nu_{s}|(E) -|\nu_t|(E)\ge \int_{E}h_t{\rm
  d}(\nu_s-\nu_t)\,.
\end{equation}

Now take any $t\in(a,\infty)$ such that the derivative $ \fr{{\rm
    d}}{{\rm d}t}|\nu_t|(E)$ exists.  By \eqref{(5.8)} we have
\begin{eqnarray*}
&&s>t\,\Longrightarrow\,\fr{|\nu_{s}|(E)
  -|\nu_t|(E)}{s-t}\ge \int_{E}h_t{\rm d}\Big( \fr{\nu_s-\nu_t}{s-t}\Big)\,,\\
&& s<t \,\Longrightarrow \,\fr{|\nu_{s}|(E) -|\nu_t|(E)}{s-t}\le
\int_{E}h_t{\rm d}\Big( \fr{\nu_s-\nu_t}{s-t}\Big)\,.
\end{eqnarray*}
Since $(\nu_s-\nu_t)/(s-t)\to \mu_t\,(s\to t)$ in norm $\|\cdot\|$,
it follows that
$$\fr{{\rm d}}{{\rm d}t}|\nu_t|(E)=\lim_{s\to t} \fr{|\nu_{s}|(E)
-|\nu_t|(E)}{s-t}=\int_{E}h_t{\rm d}\mu_t\,.$$ This proves \eqref{(5.7)} and
completes the proof.
\end{proof}

Let us now prove that the sign decomposition on differences of product
measures preserves the invariance by exchanging $v$ and $v_*$.
\begin{lemma}[Sign decomposition and exchange of particles]\label{lem5.3}
  For any $\mu,\nu\in {\mathcal B}_s^{+}(\mathbb{R}^N) $ ($s\ge 0$)
  and any locally bounded Borel function $\psi\in
  L^{\infty}_{-s}(\mathbb{R}^N \times \mathbb{R}^N)$ we have
\begin{equation}
\intt_{\mathbb{R}^N \times \mathbb{R}^N} \psi(v,v_*){\rm d}(\mu\otimes \mu-\nu\otimes\nu) = \intt_{\mathbb{R}^N \times \mathbb{R}^N}
\psi(v_*,v){\rm d}(\mu\otimes \mu-\nu\otimes\nu)\,,
\label{(5.14)}
\end{equation}
\begin{equation}
\intt_{\mathbb{R}^N \times \mathbb{R}^N} \psi(v,v_*){\rm d}|\mu\otimes
\mu-\nu\otimes\nu| = \intt_{\mathbb{R}^N \times \mathbb{R}^N}
\psi(v_*,v){\rm d}|\mu\otimes \mu-\nu\otimes\nu|\,,
\label{(5.15)}
\end{equation}
\begin{equation}
  \intt_{\mathbb{R}^N \times \mathbb{R}^N} \psi(v,v_*)
  {\rm d}(\mu\otimes \mu-\nu\otimes\nu)^{+} =
  \intt_{\mathbb{R}^N \times \mathbb{R}^N} \psi(v_*,v)
  {\rm d}(\mu\otimes \mu-\nu\otimes\nu)^{+}
  \,.
\label{(5.16)}
\end{equation}

\end{lemma}

\begin{proof}[Proof of Lemma~\ref{lem5.3}]
  Equality \eqref{(5.14)} easily follows from Fubini's
  theorem. Equality \eqref{(5.16)} follows from \eqref{(5.15)} and the
  relation
$$
{\rm d}(\mu\otimes \mu-\nu\otimes\nu)^{+}=\fr{1}{2}\Big({\rm d} |\mu\otimes
\mu-\nu\otimes\nu|+{\rm d}(\mu\otimes \mu-\nu\otimes\nu)\Big)\,.
$$
So we only have to prove \eqref{(5.15)}. To do this we split $\psi$ as
$\psi=\psi^{+}-(-\psi)^{+}$ so that we can assume that $\psi\ge
0$. Let $h(v,v_*)$ be the sign function of the measure $\mu\otimes
\mu-\nu\otimes\nu$.  Then applying \eqref{(5.14)} to $\psi(v,v_*)h(v,v_*)$ we
have
\begin{eqnarray*}
  &&\intt_{\mathbb{R}^N \times \mathbb{R}^N}
  \psi(v,v_*){\rm d}|\mu\otimes \mu-\nu\otimes \nu| =\intt_{\mathbb{R}^N
    \times \mathbb{R}^N} \psi(v,v_*)h(v,v_*){\rm d}(\mu\otimes
  \mu-\nu\otimes\nu)
  \\
  &&=\intt_{\mathbb{R}^N \times \mathbb{R}^N} \psi(v_*,v)h(v_*,v){\rm
    d}(\mu\otimes \mu-\nu\otimes\nu)\le\intt_{\mathbb{R}^N \times
    \mathbb{R}^N} \psi(v_*,v){\rm d}|\mu\otimes
  \mu-\nu\otimes\nu|\,. \end{eqnarray*} Replacing $\psi(v,v_*)$ with
$\psi(v_*,v)$ we also obtain the reversed inequality. This proves
\eqref{(5.15)}.
\end{proof}

Finally let us prove a signed estimate on the collision operator.

\begin{lemma} \label{lem5.4} Let $B(z,\sg)$ be given by
  \eqref{(1.B)}-\eqref{(1.2)}-\eqref{(1.3)} with $b(\cdot)$ satisfying {\bf
    (H4)}. Let $\mu\in {\mathcal B}_{2+\gm}^{+}(\mathbb{R}^N)\,, \nu\in
  {\mathcal B}_{2\gm}^{+}(\mathbb{R}^N)$, and let $h(v)$ be the sign function of
  $\mu-\nu$ and let $\kappa=\fr{1}{2}(1+h)$ so that $\kappa {\rm
    d}(\mu-\nu)={\rm d}(\mu-\nu)^{+}$.  Then for any $\vp\in C_b(\mathbb{R}^N)$
  satisfying $0\le \vp(v)\le \langle v\rangle^2$ we have
\begin{eqnarray}\label{(5.A1)} &&
\int_{\mathbb{R}^N}\vp(v) \kappa(v){\rm d} (Q(\mu,\mu)-Q(\nu,\nu))(v)\\
\nonumber &&\le
E_{\vp}+2^{\gm/2}A_0\Big(\|\mu\|_{2+\gm}\|\mu-\nu\|_{0}+\|\mu\|_{2}\|\mu-\nu\|_{\gm}
\Big)\end{eqnarray} where
$$ E_{\vp}=
A_02^{\gm}\|\mu\|_{\gm}\int_{\mathbb{R}^N}(\langle v\rangle^2-\vp(v))\langle
v\rangle^{\gm} {\rm d}\mu(v). $$
\end{lemma}

 \begin{proof}[Proof of Lemma~\ref{lem5.4}] Since $\vp$ is bounded, there is no problem of
integrability in the following derivation. For instance we can write
\begin{eqnarray}\label{(5.AA)}
\int_{\mathbb{R}^N}\vp(v) \kappa(v){\rm d} \Big(Q(\mu,\mu)-Q(\nu,\nu)\Big)(v)
=I^{(+)}-I^{(-)}\end{eqnarray}
where \begin{eqnarray*}&&I^{(+)}=
\intt_{\mathbb{R}^N\times\mathbb{R}^N} L_B[\vp\kappa](v,v_*){\rm d}(\mu\otimes
\mu-\nu\otimes \nu) ,\\
&& I^{(-)}=\intt_{\mathbb{R}^N\times\mathbb{R}^N}A(v-v_*)\vp(v)\kappa(v){\rm
d}(\mu\otimes \mu-\nu\otimes \nu).\end{eqnarray*}
By definition of $B(v-v_*,\sg)$ and
$\vp(v)\kappa(v)\le \langle v\rangle^2$ we have
\begin{eqnarray*}&& L_B[\vp\kappa](v,v_*)+L_B[\vp\kappa](v_*,v)\\
&&\le \int_{\mathbb{S}^{N-1}}B(v-v_*,\sg)(\langle v'\rangle^2+\langle
v_*'\rangle^2)
 d\sg
=A(v-v_*)(\langle v\rangle^2+\langle v_*\rangle^2). \end{eqnarray*} Then
using ${\rm d}(\mu\otimes \mu-\nu\otimes \nu)\le {\rm d}(\mu\otimes
\mu-\nu\otimes \nu)^{+}$ and Lemma \ref{lem5.3} we compute
\begin{eqnarray*} I^{(+)} &\le &
\fr{1}{2}\intt_{\mathbb{R}^N\times \mathbb{R}^N}(
L_B[\vp\kappa](v,v_*)+L_B[\vp\kappa](v_*,v) ){\rm d}(\mu\otimes \mu-\nu\otimes
\nu)^{+} \\
&\le& \intt_{\mathbb{R}^N\times\mathbb{R}^N}A(v-v_*)\langle v\rangle^2{\rm
d}(\mu\otimes \mu-\nu\otimes\nu)^{+}\,.\end{eqnarray*} Since $A(v-v_*)\le
A_02^{\gm}\langle v\rangle^{\gm}\langle v_*\rangle^{\gm}$, $\langle
v\rangle^2-\vp(v)\ge 0$, and $(\mu\otimes \mu-\nu\otimes\nu)^{+}\le \mu\otimes
\mu$, it follows that
\begin{eqnarray*}&&\intt_{\mathbb{R}^N\times\mathbb{R}^N}A(v-v_*)(\langle
v\rangle^2-\vp(v)){\rm d}(\mu\otimes
\mu-\nu\otimes\nu)^{+}\\
&& \le A_02^{\gm}\intt_{\mathbb{R}^N\times\mathbb{R}^N}\langle
v\rangle^{\gm}\langle v_*\rangle^{\gm}(\langle
v\rangle^2-\vp(v)){\rm d}(\mu\otimes \mu)\\
&&= A_02^{\gm}\|\mu\|_{\gm}\int_{\mathbb{R}^N}\langle v\rangle^{\gm}(\langle
v\rangle^2-\vp(v)){\rm d}\mu(v)=E_{\vp}\,.  \end{eqnarray*}  Therefore using
$${\rm d}(\mu\otimes \mu-\nu\otimes
\nu)^{+}(v,v_*) \le {\rm d}\mu(v){\rm d}(\mu-\nu)^{+}(v_*)+{\rm
d}(\mu-\nu)^{+}(v){\rm d}\nu(v_*)$$ we have
\begin{eqnarray}\label{(5.A2)} I^{(+)} &\le&
E_{\vp}+\intt_{\mathbb{R}^N\times\mathbb{R}^N}A(v-v_*)\vp(v){\rm d}\mu(v)
{\rm d}(\mu-\nu)^{+}(v_*)
\\ \nonumber
&+&\intt_{\mathbb{R}^N\times\mathbb{R}^N}A(v-v_*)\vp(v){\rm
d}(\mu-\nu)^{+}(v) {\rm d}\nu(v_*)\,.\quad \quad
\end{eqnarray}
Similarly using ${\rm d}(\mu\otimes \mu-\nu\otimes \nu)(v,v_*) ={\rm
d}\mu(v){\rm d}(\mu-\nu)(v_*)+{\rm d}(\mu-\nu)(v){\rm d}\nu(v_*)$ and \\
$\kappa(v){\rm d}(\mu-\nu)(v)= {\rm d}(\mu-\nu)^{+}(v)$  we have
\begin{eqnarray}\label{(5.A3)}
I^{(-)}&=&\intt_{\mathbb{R}^N\times \mathbb{R}^N}
A(v-v_*)\vp(v)\kappa(v){\rm d}\mu(v){\rm d}(\mu-\nu)(v_*)\\
\nonumber &+& \intt_{\mathbb{R}^N\times \mathbb{R}^N}A(v-v_*)\vp(v){\rm
d}(\mu-\nu)^{+}(v){\rm d}\nu(v_*).\qquad \qquad \qquad  \qquad \end{eqnarray}
Canceling the common term in \eqref{(5.A2)} and \eqref{(5.A3)} and noticing
that
\[ {\rm d}(\mu-\nu)^{+}(v_*)\le {\rm d}(\mu-\nu)(v_*)+ {\rm
  d}|\mu-\nu|(v_*)
\]
we obtain from \eqref{(5.AA)},\eqref{(5.A2)},\eqref{(5.A3)} that
\begin{eqnarray}\label{(5.A4)}
&&\int_{\mathbb{R}^N}\vp(v)\kappa(v) {\rm d}(Q(\mu,\mu)-Q(\nu,\nu))\\
\nonumber && \le E_{\vp} +\intt_{\mathbb{R}^N\times
\mathbb{R}^N}A(v-v_*)\vp(v){\rm d}\mu(v) {\rm d}|\mu-\nu|(v_*)\,.
\end{eqnarray}
Since $A(v-v_*)\vp(v)\le A_02^{\gm/2}(\langle v\rangle^{\gm}+\langle
v_*\rangle^{\gm})\langle v\rangle^2$, it follows that
$$\intt_{\mathbb{R}^N\times \mathbb{R}^N}A(v-v_*)\vp(v){\rm d}\mu(v)
{\rm d}|\mu-\nu|(v_*)\le A_0
2^{\gm/2}(\|\mu\|_{2+\gm}\|\mu-\nu\|_{0}+\|\mu\|_2\|\mu-\nu\|_{\gm})$$ which
together with \eqref{(5.A4)} proves \eqref{(5.A1)}.\end{proof}

\subsection{Proof of Theorem~\ref{theo2}}
We shall consider each part step by step.

\subsubsection*{Proof of part (a)} Recall that
$B(z,\sg)=|z|^{\gm}b(\cos\theta)$ satisfies $A_0<\infty$ \,and
$0<\gm\le 2$. Let $F_{t}$ be a conservative measure weak solution of
Eq.~\eqref{(B)} with $F_{t}|_{t=0}=F_0\in {\mathcal
  B}_2^{+}(\mathbb{R}^N)$. We prove that $F_t$ is a measure strong
solution.

  First of all by $\|F_t\|_0,\, \|F_t\|_{\gm}\le
  \|F_0\|_2$ and Proposition \ref{prop1.4} we have
$$\|Q^{\pm}(F_t, F_t)\|_0\le 4A_0\|F_0\|_{2}^2 \, , \quad \forall\, t\ge 0,$$
$$
 \langle Q(F_t, F_t),\,\vp\rangle=\int_{\mathbb{R}^N}\vp{\rm
   d}Q(F_t,F_t)\quad \forall\, \vp\in C^2_b(\mathbb{R}^N),\quad \forall\,t\ge 0.$$
   Since
$$
t\mapsto \int_{\mathbb{R}^N}\vp{\rm d}Q(F_t,F_t)=\langle Q(F_t,
F_t),\,\vp\rangle \,\,\,{\rm  belongs \,\,\,to}\,\,\, C((0,\infty))\cap
L^1_{{\rm loc}}([0,\infty))$$ there is no problem of integrability and the
integral equation for a measure weak solutions becomes
\begin{equation}
\int_{\mathbb{R}^N}\vp{\rm d}F_t=
\int_{\mathbb{R}^N}\vp{\rm d}F_0+\int_{0}^{t}{\rm d}s\int_{\mathbb{R}^N}\vp{\rm
d}Q(F_s,F_s) \,.
\label{(5.20)}
\end{equation}
Now take any $\vp\in C_c^2(\mathbb{R}^N)$ satisfying $\|\vp\|_{L^{\infty}}\le
1$. We have
\begin{equation*}
\Big|\int_{\mathbb{R}^N}\vp{\rm d}Q(F_t,F_t)\Big| \le \|Q(F_t,F_t)\|_0\le 8
A_0\|F_0\|_2^2\, , \quad \forall\,t\ge
0\,.
\end{equation*}
and thus using \eqref{(5.20)}, for all $0\le t_1<t_2<\infty$
\begin{eqnarray*}
&&\left|\int_{\mathbb{R}^N}\vp {\rm d}(F_{t_2}-F_{t_1}) \right|\le
\int_{t_1}^{t_2}\Big|\int_{\mathbb{R}^N}\vp{\rm d}Q(F_s,F_s)\Big|ds \le
8A_0\|F_0\|^2_2|t_1-t_2|\,.
\end{eqnarray*}
Applying \eqref{(5.1)} this gives
\begin{equation}
\|F_{t_1}-F_{t_2}\|_{0}\le 8A_0\|F_0\|^2_2|t_1-t_2|\,,\quad \forall\,
t_1,t_2\in[0,\infty)
\label{(5.22)}
\end{equation}
which enables us to prove the strong continuity:
\begin{equation}
t\mapsto F_t\in C([0,\infty); {\mathcal B}_2(\mathbb{R}^N)),\quad t\mapsto
Q^{\pm}(F_t,F_t)\in C([0,\infty); {\mathcal
  B}_0(\mathbb{R}^N))\,.
\label{(5.23)}
\end{equation}
In fact applying the inequality \eqref{(5.10)} in Proposition \ref{prop1.4}
with $s=0$ (recall that $0<\gm \le 2$) we have
\begin{eqnarray}\label{(5.A5)} \|Q^{\pm}(F_{t},F_{t})-Q^{\pm}(F_{t_0},F_{t_0})\|_0\le 8A_0\|F_0\|_2
\|F_t-F_{t_0}\|_2,\quad t,t_0\ge 0.\end{eqnarray} Fix $t_0\in[0,\infty)$.
Using \eqref{(5.2)}, the conservation of mass and energy, ${\rm
d}(F_{t_0}-F_t)^{+}\le {\rm d}F_{t_0}$, and \eqref{(5.22)} we have for any
$R\ge 1$
\begin{eqnarray*}
\|F_t-F_{t_0}\|_2&=&2\int_{\mathbb{R}^N}\langle v\rangle^2 {\rm
d}(F_{t_0}-F_t)^{+}(v)
\\
&\le& 2R^2\int_{\langle v\rangle \le R} {\rm d}(F_{t_0}-F_t)^{+}(v)+
2\int_{\langle
v\rangle > R}\langle v\rangle^2 {\rm d}F_{t_0}(v)\\
&\le&  2^{4}A_0R^2|t-t_0|+2\int_{\langle v\rangle > R}\langle v\rangle^2 {\rm
d}F_{t_0}(v)\,.
\end{eqnarray*}
Thus first letting $t\to t_0$ and then letting $R\to\infty$ leads to
$\lim\sup\limits_{t\to t_0}\|F_t-F_{t_0}\|_2=0\,.$ This together with \eqref
{(5.A5)} proves \eqref{(5.23)}.

From the strong continuity in \eqref{(5.23)} we have for all $\vp\in
C^2_b(\mathbb{R}^N)$
$$
\int_{0}^{t}{\rm d}s\int_{\mathbb{R}^N}\vp{\rm d}Q(F_{s},F_{s})
=\int_{\mathbb{R}^N}\vp{\rm
  d}\left(\int_{0}^{t}Q(F_{s},F_{s})ds\right)
$$
which together with \eqref{(5.20)} yields
$$
\int_{\mathbb{R}^N}\vp{\rm d}F_t= \int_{\mathbb{R}^N}\vp{\rm
  d}F_0+\int_{\mathbb{R}^N}\vp{\rm
  d}\left(\int_{0}^{t}Q(F_{s},F_{s})ds\right)\,.$$ Therefore applying
\eqref{(5.1)} we obtain
$$F_t=F_0+\int_{0}^{t}Q(F_{s},F_{s}){\rm
d}s\,,\quad t\ge 0\,.
$$
Since $t\mapsto Q^{\pm}(F_{t},F_{t})\in C([0,\infty);{\mathcal
B}_0(\mathbb{R}^N))$, it follows  that $t\mapsto F_t\in
C^1([0,\infty);{\mathcal
  B}_0(\mathbb{R}^N)) $ and
\[
\fr{{\rm d}}{{\rm d}t}F_t=Q(F_{t},F_{t}),\quad t\ge 0\,.
\]
So $F_t$ is a measure strong solution.

The converse is obvious because of (\ref{def-strong-bis}) and
(\ref{(5.9)}) with $s=0$: Every measure strong solution is a measure
weak solution.  \medskip

\subsubsection*{Proof of parts (b)-(c)-(d)} The proof of these three
parts can be reduced to the proof of the following lemma:

\begin{lemma}\label{lem:interm}
Let  $F_0\in {\mathcal B}_2^{+}(\mathbb{R}^N)$ with
  $\|F_0\|_0\neq 0$ and let $F_t$ be a conservative measure strong solution
  of Eq.~\eqref{(B)} with the initial datum
 $F_0$ and satisfy the moment production estimate \eqref{(1.12)}-\eqref{(1.12*)} in Theorem
  \ref{theo1}. Let $G_t$ be any measure strong solution of
  Eq.~\eqref{(B)} on the time interval $[\tau,\infty)$ with initial
  data
  \[
  G_t|_{t=\tau}=G_{\tau}\in {\mathcal B}_2^{+}(\mathbb{R}^N)
  \]
  for some $\tau\ge 0$, and satisfying $\|G_t\|_2\le \|G_{\tau}\|_2$
  for all $t\in[\tau,\infty)$.

  Then the stability estimates \eqref{(1.21)} (for $\tau=0$) and
  \eqref{(1.21*)} (for $\tau>0$) hold true.
\end{lemma}

Note that the existence of such a solution $F_t$ as in the statement has been
proven by Theorem \ref{theo1} and part (a) of the present theorem.  Therefore
if Lemma~\ref{lem:interm} holds true, then by taking $G_0=F_0$ (for the case
$\tau=0$) we get $G_t\equiv F_t$ on $[0,\infty)$ and hence this proves parts
(b), (c) and (d).

\begin{proof}[Proof of Lemma~\ref{lem:interm}]
  Our proof is divided into several steps. First of all for notation
  convenience we denote
$$H_t=F_t-G_t. $$

\noindent{\bf Step 1.} Given any $0<r\in[\tau,\infty)$. We prove that
\begin{eqnarray}\label{(5.A6)} &&
\|H_t\|_{2}\le \|G_\tau\|_2-\|F_\tau\|_2+2\|(H_r)^+\|_2\\ \nonumber && +
4A_0\left({\mathcal K}_{2+\gm}(F_0) \int_{r}^{t}(1+1/s)\|H_s\|_{0}{\rm d}s+
\|F_0\|_2 \int_{r}^{t}\|H_s\|_{\gm}{\rm d}s\right)\,, \quad t\ge r\,.
\end{eqnarray}
Here ${\mathcal K}_{2+\gm}(F_0)$ is the constant in
 \eqref{(1.12*)} with $s=2+\gm$.
 To prove \eqref{(5.A6)}, we consider approximation: By ${\rm
d}|H_t|={\rm d}G_t-{\rm d}F_t+2{\rm d}(H_t)^{+}$ we have
$$\|H_t\|_{2}=\|G_t\|_2-\|H_t\|_2+2\lim_{n\to\infty}
\int_{\mathbb{R}^N}\langle{v}\rangle^2_n {\rm d}(H_t)^{+}\quad{\rm
  with}\quad \langle{v}\rangle^2_n=\min\{ \langle v\rangle^2,\, n\}.$$ Let $
v\mapsto h_t(v)$ be the sign function of $H_t$ and
$\kappa_t(v)=\fr{1}{2}(1+h_t(v))$ so that $\kappa_t{\rm d}H_t={\rm
  d}(H_t)^{+}$. Then applying Lemma \ref{lem5.1} to the measure
\[
H_t=H_r+\int_{r}^t ( Q(F_s,F_s)-Q(G_s,G_s)){\rm d}s \quad
\mbox{ for } \quad t\ge r
\]
and then using Lemma \ref{lem5.4} we have \begin{eqnarray*}&&
  \int_{\mathbb{R}^N}\langle{v}\rangle^2_n {\rm d}(H_t)^{+}=
  \int_{\mathbb{R}^N}\langle{v}\rangle^2_n {\rm
    d}(H_r)^{+}+\int_{r}^{t}{\rm
    d}s\int_{\mathbb{R}^N}\langle{v}\rangle^2_n\kappa_s(v)
  {\rm d}\Big(Q(F_s, F_s)-Q(G_s, G_s)\Big)\\
  &&\le
  \|(H_r)^{+}\|_2+E_n(t)+2A_0\Big(\int_{r}^{t}\|F_{s}\|_{2+\gm}\|H_{s}\|_{0}{\rm
    d}s + \|F_0\|_2\int_{r}^{t}\|H_{s}\|_{\gm}{\rm d}s\Big),\quad
  t\in[r, \infty)\end{eqnarray*} where
$$E_n(t)=4A_0\int_{r}^{t}\|F_{s}\|_{\gm}\left(\int_{\mathbb{R}^N} (\langle
v\rangle^2-\langle v\rangle_n^2)\langle{v}\rangle^{\gm} {\rm d}F_s\right){\rm
d}s .$$ Since, by moment estimate \eqref{(1.12)},
$$\int_{r}^{t}\|F_{s}\|_{\gm}\left(\int_{\mathbb{R}^N}
\langle{v}\rangle^{2+\gm} {\rm d}F_s(v)\right){\rm d}s \le
\|F_0\|_2\int_{r}^{t}\|F_s\|_{2+\gm}{\rm d}s<\infty\,,\quad
t\in[r,\infty)\,,$$ it follows from dominated convergence that
$\lim\limits_{n\to\infty}E_n(t)=0$ and thus
 \begin{eqnarray*} \|H_t\|_{2}&\le &\|G_t\|_2-\|F_t\|_2+ 2\|(H_r)^{+}\|_2\\
&+&4A_0\Big(\int_{r}^{t}\|F_{s}\|_{2+\gm}\|H_{s}\|_{0}{\rm d}s +
\|F_{0}\|_{2}\int_{r}^{t}\|H_{s}\|_{\gm}{\rm d}s\Big)\,,\quad \forall\,t\in
[r,\infty)\,.\end{eqnarray*}  By assumption on $F_t$ and $G_t$ we have
$\|G_t\|_2-\|F_t\|_2\le \|G_\tau\|_2-\|F_\tau\|_2$ and $\|F_s\|_{2+\gm}\leq
{\mathcal K}_{2+\gm}(F_0)(1+1/s)$. This  proves \eqref{(5.A6)}.
\smallskip

\noindent{\bf Step 2.} Suppose $\tau>0$. Then taking $r=\tau$ in
\eqref{(5.A6)} and using
$\|G_\tau\|_2-\|F_\tau\|_2+2\|(H_\tau)^+\|_2=\|H_{\tau}\|_2$ we obtain
\[
 \|H_t\|_{2}\le \|H_\tau\|_2+
c_{\tau}\int_{\tau}^{t}\|H_s\|_2{\rm d}s\qquad \forall\, t\in[\tau,\infty)
\]
with $c_{\tau}=4A_0({\mathcal
  K}_{2+\gm}(F_0)+\|F_0\|_2)(1+\fr{1}{\tau})$. This gives
\eqref{(1.21*)} by Gronwall's Lemma.

The remaining steps deal with the case $\tau=0$ and prove
\eqref{(1.21)}.
\smallskip

\noindent{\bf Step 3.} If $\|H_0\|_2\ge 1$, then using $\|F_t\|_2=\|F_0\|_2,\,
\|G_t\|_2\le \|G_0\|_2$ we have
$$\|H_t\|_2\le (1+2\|F_0\|_2)\|H_0\|_2\qquad \forall\,
t\in[0, \infty)\,.$$ So in the following we assume that
$\|H_0\|_2<1\,.$ Note that in this case we
have \begin{equation}\label{(5.A7)}
\|F_t\pm G_t\|_{2}\le
  1+2\|F_0\|_2=:C_0\qquad \forall\, t\ge 0\,.
\end{equation}
Using Proposition \ref{prop1.4} we have
\begin{eqnarray*}
  \|H_t\|_{0}&\le &\|H_0\|_{0}+\int_{0}^{t}\|Q(F_s,F_s)-Q(G_s,G_s)\|_{0}{\rm d}s
  \\
 & \le& \| H_0\|_{0}+4A_0 \int_{0}^{t}
  \Big(\|F_s+G_s\|_{\gm}\|H_s\|_{0}+\|F_s+
  G_s\|_{0}\|H_s\|_{\gm}\Big){\rm d}s
\end{eqnarray*}
and thus by $0<\gm\le 2$ and \eqref{(5.A7)} we obtain
\begin{equation}\label{(5.A8)}
\|H_t\|_{0}\le \|H_0\|_{0}+8A_0C_0 \int_{0}^{t} \|H_s\|_{2}{\rm d}s\,, \quad
\forall\, t\ge 0.
\end{equation}
\smallskip

\noindent
{\bf Step 4.} Let $r>0$ satisfy $\|H_0\|_2\le r \le 1$.  We prove that
\begin{equation}\label{(5.A9)}
U(r):=\sup_{0\le t\le r}\|H_t\|_2 \le 4(1+9A_0C_0^2)\Psi_{F_0}(r)\,.
 \end{equation}
First of all using \eqref{(5.2)} and $\|G_t\|_2-\|F_t\|_2\le
\|G_0\|_2-\|F_0\|_2\le r$ we have
\begin{equation}\label{(5.A10)}
\|H_t\|_2=\|G_t\|_2-\|F_t\|_2+2\|(H_t)^{+}\|_2\le r+2\|(H_t)^{+}\|_2
 \end{equation}
and for any $R\ge 1$
\begin{equation}\label{(5.A11)}2\|(H_t)^{+}\|_2\le
4R^{2}\|H_t\|_{0}+2\int_{|v|> R}\langle v\rangle^2 {\rm d}F_t(v)\,.
 \end{equation}
 Next by $\|H_0\|_2\le r$ and
\eqref{(5.A8)} we have
 \begin{equation}\label{(5.A12)}
4R^{2}\|H_t\|_{0}\le 4(1+8A_0C_0^2)R^{2} \,r\,\quad \forall\, t\in [0,r].
 \end{equation}
Using the conservation of mass and energy we compute
\begin{eqnarray*}
\int_{|v|> R}\langle v\rangle^2{\rm d}F_t(v)&=&\int_{\mathbb{R}^N} \langle
v\rangle^2{\rm
d}F_t(v)-\int_{|v|\le  R}\langle v\rangle^2{\rm d}F_t(v)\\
&=&\int_{\mathbb{R}^N}\langle v\rangle^2{\rm d}F_0(v)- \int_{|v|\le R}\langle
v\rangle^2{\rm
d}F_0(v)- \int_{0}^{t}{\rm d}s\int_{|v|\le R}\langle v\rangle^2{\rm d} Q(F_s,F_s)\\
&\le& \int_{|v|>R}\langle v\rangle^2{\rm d}F_0(v)+\int_{0}^{t}{\rm
d}s\int_{|v|\le R}\langle v\rangle^2{\rm d} Q^{-}(F_s,F_s)\,.
\end{eqnarray*}
For the last term we use $|v-v_*|^{\gm}\le \langle v\rangle^{\gm} \langle
v_*\rangle^{\gm}\le \langle v\rangle^{2} \langle v_*\rangle^{2}$ to get for
all $t\in [0, r]$
\[
\int_{0}^{t}{\rm d}s\int_{|v|\le R}\langle v\rangle^2{\rm d} Q^{-}(F_s,F_s)
\le 2R^2 \int_{0}^{t}{\rm d}s\int_{\mathbb{R}^N}{\rm d} Q^{-}(F_s,F_s) \le
2A_0\|F_0\|_2^2 R^2\,r\,.
\]
Thus
\begin{equation}\label{(5.A13)}
\int_{|v|> R}\langle v\rangle^2{\rm d}F_t(v) \le \int_{|v|>R}\langle
v\rangle^2{\rm d}F_0(v) + 2A_0\,\|F_0\|_2^2R^{2}\,r\quad \forall\,
t\in[0,r]\,.
 \end{equation}
 Combining \eqref{(5.A11)}-\eqref{(5.A12)}-\eqref{(5.A13)}
gives
\begin{equation}\label{(5.A14)}
2\|(H_t)^{+}\|_2\le
4(1+9A_0C_0^2)R^{2}r + 4\int_{|v|>R}|v|^2{\rm d}F_0(v)\,,\quad t\in [0,r]\,.
 \end{equation}
Now choose $R=r^{-1/3}$. Then from \eqref{(5.A10)}, \eqref{(5.A14)}
 we obtain
\[
\|H_t\|_2\le  r+
4(1+9A_0C_0^2) r^{1/3} + 4\int_{|v|>r^{-1/3}}|v|^2{\rm d}F_0(v)\,,\quad t\in
[0,r]\,.
\]
This gives \eqref{(5.A9)} by definition of $\Psi_{F_0}(r)$ in
\eqref{(1.20)}.  \smallskip

\noindent{\bf Step 5.}  In the following we denote $C_i={\mathcal
  R}_i(\gm, A_0, A_2, \|F_0\|_0,\|F_0\|_2)$ for $(i=1,2,\dots,6)$,
where ${\mathcal R}_i(x_1,x_2, \dots, x_5)$ are some explicit positive
continuous functions in $(\mathbb{R}_{>0})^5$.

In \eqref{(5.A6)} setting $\tau=0, r=1$ we have
\[
\|H_t\|_{2}\le \|H_0\|_2+2\|H_1\|_2+ C_1 \int_{1}^{t}\|H_{s}\|_{2}{\rm
  d}s \,,\quad t\ge 1
\]
so that Gronwall's Lemma applies to get
\begin{equation}\label{(5.A15)}\|H_t\|_{2}\le
\Big(\|H_0\|_2+2\|H_1\|_2 \Big)\exp(C_1(t-1))
 \,,\quad t\ge 1\,.\end{equation}
Now we concentrate our estimate for $t\in[0,1]\,.$ In what follows we assume
$r$ satisfy
\begin{equation}\label{(5.A16)}
r>0\,,\quad \|H_0\|_2\le r< 1\,.
\end{equation}
 Using \eqref{(5.A6)} (with $\tau=0$),
$\|G_0\|_2-\|F_0\|_2\le \|H_0\|_2\le r$, and $\|H_r\|_2\le U(r)$ we have
$$ \|H_t\|_{2}\le r+2U(r)+ C_2\left(\int_{r}^{t}\fr{1}{s}\|H_{s}\|_{0}{\rm d}s
+\int_{r}^{t}\|H_{s}\|_{\gm}{\rm d}s\right), \quad t\in[r,1]\,.$$
Further, using \eqref{(5.A8)} we compute for all $t\in [r, 1]$
\begin{eqnarray*}\int_{r}^{t}\fr{1}{s}\|H_{s}\|_{0}{\rm d}s &\le&
r\log(t/r)+8A_0C_0\int_{r}^{t}\fr{1}{s}\int_{0}^{s} \|H_{\tau}\|_2 {\rm d}\tau{\rm d}s\\
&\le & r|\log r |+8A_0C_0\int_{0}^{t} \|H_{\tau}\|_{2}|\log \tau | {\rm
  d}\tau\,. \end{eqnarray*}
Thus for all $t\in [r,\,1]$
\begin{equation}\label{(5.A17)}
\|H_t\|_{2}\le r+2U(r)+ C_2 r|\log r| + C_3\int_{0}^{t} \|H_{s}\|_{2}(1+|\log
s|){\rm d}s\,. \end{equation}
 Since $\|H_t\|_{2}\le U(r)$ for all $t\in[0, r]$, the
inequality \eqref{(5.A17)} holds for all $t\in [0,1]$. Therefore by Gronwall's
Lemma we conclude
\begin{equation}\label{(5.A18)}
\|H_t\|_{2}\le C_4(r+U(r)+ r|\log r|\,)\qquad \forall\, t\in [0,\,1]\,.
\end{equation}
 In
particular taking $t=1$ yields the estimate for $\|H_1\|_{2}$ and thus from
\eqref{(5.A15)}-\eqref{(5.A16)} we obtain
\begin{equation}\label{(5.A19)}
\|H_t\|_{2}\le C_5(r+U(r)+r|\log r|\,)\exp(C_1(t-1))\,,\quad \forall\, t\in
[1,\infty)\,.
\end{equation}
 Combining \eqref{(5.A18)}-\eqref{(5.A19)}
and the inequality $r|\log r|\le r^{1/3}$ we conclude
\begin{equation}\label{(5.A20)}
  \|H_t\|_{2}\le
  \Psi_{F_0}(r)\exp(C_6(1+t))\quad \forall\, t\ge 0\,.
\end{equation}
Finally if $\|H_0\|_2=0$, then in \eqref{(5.A20)} letting $r\to 0+$
leads to $\|H_t\|_{2}\equiv 0$; if $\|H_0\|_2>0$, we take
$r=\|H_0\|_2$. This proves \eqref{(1.21)} and completes the proof of
the lemma.
\end{proof}
\medskip

\subsubsection*{Proof of part (e)} Let ${\rm d}F_0(v)=f_0(v){\rm d}v$
with $0\le f_0\in L^1_2(\mathbb{R}^N)$, and let $F_t$ be the unique
conservative measure strong solution of Eq.~\eqref{(B)} with the
initial datum $F_0$. By the Lebesgue-Radon-Nikodym theorem, for every
$t\ge 0$ we have a decomposition ${\rm d}F_t(v)=f_t(v){\rm d}v+{\rm
  d}\mu_t(v)$ where $0\le f_t\in L^1_2(\mathbb{R}^N)$, $\mu_t\in
{\mathcal B}_2^{+}(\mathbb{R}^N)$ and $\mu_t$ concentrates on a Lebesgue null
set. By the uniqueness of $F_t$ we can assume that $\|f_0\|_{L^1}\neq 0$.
Let
\[
f^n_0(v)=\min\{f_0(v),\, n\} e^{-|v|^2/n}\,, \ \mbox{ and } {\rm
  d}F_0^n(v)=f^n_0(v){\rm d}v\,.
\]
By Step 2 of the proof of Theorem \ref{theo1}, for every $n$ there is a
conservative measure weak solution $F^n_t$ with the initial datum $F^n_0$ and
${\rm d}F^n_t(v)=f^n_t(v){\rm d} v$, $0\le f^n_t\in L^1_2(\mathbb{R}^N)$ for
all $t\ge 0$.  By part (a), $F^n_t$ is also a measure strong solution. Since
${\rm d}(F_t-F^n_t)=(f_t-f^n_t){\rm d} v+{\rm d}\mu_t$ we have
$\|F_t-F^n_t\|_2= \|f_t-f^n_t\|_{L^1_2}+\|\mu_t\|_2$. Since
\begin{eqnarray*}
  \|F_0-F_0^n\|_2&=&\|f_0-f^n_0\|_{L^1_2} \\
  &\le& \int_{f_0(v)>n}f_0(v)\langle v\rangle^2{\rm
    d}v+\int_{\mathbb{R}^N}f_0(v)(1-e^{-|v|^2/n})\langle v\rangle^2{\rm d}v \to
  0\quad (n\to\infty)
\end{eqnarray*}
it follows from the stability estimate that for every fixed $t\ge 0$ we have
$$\|f_t-f^n_t\|_{L^1_2}+\|\mu_t\|_2=
\|F_t-F^n_t\|_2\le e^{C(1+t)}\Psi_{F_0}(\|F_0-F^n_0\|_2)\xrightarrow[n
  \to 0]{}0 $$
and therefore $\mu_t\equiv 0$. Thus ${\rm d}F_t(v)=f_t(v){\rm
  d}v$ for all $t\ge 0$ and hence $f_t$ is the unique conservative mild
solution of Eq.~\eqref{(B)} associated with the initial datum
$f_0$. This proves part (e). \medskip

\subsubsection*{Proof of part (f)} Suppose $F_0\in {\mathcal
  B}^{+}(\mathbb{R}^N)$ is not a Dirac mass. We can assume that
$\|F_0\|_0\neq 0$. Let $f^n_0(v)$ be defined by \eqref{(4.22)}-\eqref{(4.24)}
(the Mehler transform of $F_0$). By part (e), for every $n\ge 1$ there exists
a unique conservative $L^1$-solution $f^n_t$ of Eq.\eqref{(B)} associated with
the initial datum $f^n_t|_{t=0}=f^n_0$. If we define $F^n_0, F^n_t$ by ${\rm
  d}F^n_0(v)=f^n_0(v){\rm d}v$ and ${\rm d}F^n_t(v)=f^n_t(v){\rm d}v$,
then by uniqueness and Theorem \ref{theo1} we see that $F^n_t$ satisfies the
moment production estimates.  Thus it is easily checked that the Step 3 (where
there is no need of introducing $\wt{f}^n_0$ for the present case) in the
proof of Theorem \ref{theo1} is totally valid here. Therefore there is a
subsequence, which we still denote as $\{f^n_t\}_{n=1}^{\infty}$, such that
for the unique measure solution $F_t$ of Eq.~\eqref{(B)} with
$F_t|_{t=0}=F_0$, the weak convergence \eqref{(1.23)} holds true. This
completes the proof of Theorem \ref{theo2}.

\medskip\par\noindent\emph{Acknowledgements.} This work was started
while the first author was visiting the University Paris-Dauphine as
an invited professor during the autumn 2006, and the support of this
university is acknowledged. The first author also acknowledges support
of National Natural Science Foundation of China, Grant No. 10571101.

\par\noindent{\scriptsize\copyright\ 2011 by the authors. This paper
  may be reproduced, in its entirety, for non-commercial purposes.}

%%%%%%%%%%%%%%%%%%%%%%%%%%%%%%%%%%%%%%%%%%%%%%%%%%%%%%%%%%%%%%%%%%%%%%%%
\bibliographystyle{acm}
\bibliography{Biblio-LM}

\signxl \signcm

\end{document}